\documentclass[12pt]{article}
\usepackage{mathrsfs}
\usepackage{amsmath}
\usepackage{amssymb}
\usepackage{amsfonts}
\usepackage{mathtools}
\usepackage[top=2cm,bottom=2cm,left=2.5cm,right=2cm]{geometry}
\usepackage{graphicx}
\usepackage[colorinlistoftodos]{todonotes}
\setuptodonotes{
  inline,
  backgroundcolor=blue!8!white,
  bordercolor=blue!65!black,
  linecolor=blue!65!black
}
\usepackage{amsthm}
\usepackage{bbm}
\usepackage{enumerate}
\usepackage{enumitem}
\usepackage{verbatim}
\usepackage[backend=biber,style=alphabetic,sorting=ynt,maxnames=99]{biblatex}
\usepackage{float}
\usepackage{tikz-cd}
\usepackage[all]{xy}
\usepackage[OT2,T1]{fontenc}
\usepackage{xcolor}
\usepackage[most]{tcolorbox}
\usepackage{adjustbox}
\usepackage[bottom]{footmisc}

\newtcolorbox{commentbox}[1][]{colback=yellow!10!white, colframe=red!50!black, title=#1}

\usepackage[affil-it]{authblk}
\usepackage{commands}
\usepackage[colorlinks=true, allcolors=blue]{hyperref}

\usepackage{xparse}

\newif\ifshowcomments{}
\showcommentstrue{}

\title{Arithmetic Wu Formulas and the Generalized Hecke Theorem}
\author{Shachar Carmeli\thanks{shachar.carmeli@gmail.com}}
\author{Mark Shusterman\thanks{mark.shusterman@weizmann.ac.il}}
\author{Sa'ar Zehavi\thanks{saarzehavi@gmail.com}}
\affil{Weizmann Institute of Science}

\begin{document}
\maketitle

\begin{abstract}
We construct canonical stable Steenrod squares on modified
compactly supported \'etale cohomology, in the sense of Milne
and Geisser--Schmidt, for separated finite-type schemes over
rings of \(S\)-integers in number fields with \(2\) invertible.
This allows us to extend Feng's absolute \'etale Wu classes from finite
fields to arithmetic bases. Using a modified compactly
supported relative Wu formula extending Benoist's theorem,
we prove that for a flat projective morphism \(f\colon X\to B\)
of pure relative dimension, with \(X\) regular and \(B\) such
an arithmetic base,
\[
        v_X
        =
        \Sq^{-1}\!\bigl(
                w_{\mathrm{et}}(\tau_f+\mathcal O_X^{\oplus3})
        \bigr)
\]
in completed mod-\(2\) \'etale cohomology. Here \(\tau_f\) is
the virtual relative tangent bundle, \(w_{\mathrm{et}}\) is
the total \'etale Stiefel--Whitney class, and \(\Sq^{-1}\)
is the inverse total Steenrod square. Over finite fields
of odd characteristic, the formula holds without the three
trivial summands. As an application, we obtain a generalized
Hecke theorem giving universal mod-\(2\) congruences involving
Chern classes and the Kummer class of \(-1\), governed by an
arithmetic deformation of Hirzebruch's \(2\)-Todd series.
These congruences hold modulo an explicit archimedean ideal
and become vanishing identities over finite fields, over
totally imaginary arithmetic bases, or when \(-1\) is a
square on \(X\). We describe the real-place defects and show
that nonempty real loci force infinitely many nonzero Wu
components. Applications recover Hecke's theorem on the
different away from \(2\) and the finite-field analog of
Atiyah's theorem on theta characteristics, and yield new
higher-dimensional congruences. We also revisit Serre's
Riemann--Hurwitz theorem for spin bundles and recover the
Shusterman--Sawin theorem for smooth branched covers of
closed \(3\)-manifolds via the topological Wu formula.
Finally, we prove a function-field analog of the classical
Lusztig--Milnor--Peterson formula, expressing the difference
between mod-\(2\) and \(2\)-adic semicharacteristics as a
characteristic number involving the middle Wu class and
its Tate-twisted Bockstein.
\end{abstract}

\tableofcontents
\section{Introduction}
At first glance, a number of results scattered across topology, geometry, and arithmetic seem to have little to do with one another. They are proved in different categories and by rather different methods. Representative examples include:
\begin{enumerate}\label{enum:examples}
    \item Hecke's theorem that the ideal class of the different is a square;
    \item Atiyah's theorem on the existence of a square root of the canonical bundle of a compact Riemann surface;
    \item Serre's Riemann--Hurwitz formula for spin bundles on compact Riemann surfaces;
    \item the theorem of the second author and Sawin that the branch locus of a branched cover of closed \(3\)-manifolds is trivial mod \(2\); and
    \item the Lusztig--Milnor--Peterson formula expressing the semicharacteristic defect of a closed \(4d+1\)-manifold.
\end{enumerate}
We refer to this family of statements as \emph{Hecke-type theorems}. One of the guiding themes of this paper is that Hecke-type theorems admit a common explanation: they arise from Wu-type formulas. More precisely, once an appropriate analog of Wu's formula is available in a given cohomological setting, the corresponding Hecke-type theorem follows from the mod-\(2\) vanishing of certain universal polynomial expressions in characteristic classes.

This should not be understood literally as a direct consequence of Wu's original theorem in algebraic topology---except for the last couple of examples; see \S~\ref{sec:applications_and_examples}. Rather, there is a family of analogs of Wu's formula, each living in a different cohomological theory, and the theorems listed above can be recovered by parallel arguments once the relevant Wu formula has been established.
Furthermore, this framework allows results that were previously isolated to either topology, geometry, or arithmetic to be transported to the remaining settings via their underlying Wu formulas. We provide a concrete example of this in \S\ref{sssec:function-field-semicharacteristic-defect}, where we deduce a function-field analog of the Lusztig--Milnor--Peterson formula which, to the best of our knowledge, is new to the literature.

The first aim of this paper is to develop the arithmetic incarnation of this picture in the \'{e}tale cohomology of projective regular flat schemes over arithmetic bases on which \(2\) is invertible. Our second main result is a complete description, in every dimension, of the mod-\(2\) congruences among Chern classes that arise formally from this arithmetic Wu formula. We refer to this collection of congruences as the \emph{generalized Hecke theorem}. Before stating our results, we briefly recall Wu's formula in topology.

Let \(M\) be a closed connected smooth \(n\)-manifold. The total Steenrod square
\[
\Sq \colon H^*(M;\mathbf{Z}/2)\longrightarrow H^*(M;\mathbf{Z}/2)
\]
is an automorphism of the cohomology ring, and decomposes as \(\Sq=\sum_{i\ge 0}\Sq^i\), where \(\Sq^i\) raises degree by \(i\). The instability relations assert that for \(x_j\in H^j(M;\mathbf{Z}/2)\),
\begin{equation}\label{eq:instability_intro}
i>j \implies \Sq^i(x_j)=0.    
\end{equation}
Moreover, if \(i=j\), then \(\Sq^i(x_j)=x_j^2\). Let
\[
\int_M \colon H^*(M;\mathbf{Z}/2)\to \mathbf{Z}/2
\]
denote the trace map furnished by Poincar\'e duality. The total Wu class of \(M\) is the unique element \(v_M\in H^*(M;\mathbf{Z}/2)\) such that
\[
\int_M \Sq(x)=\int_M x\smile v_M,\qquad \text{for all } x\in H^*(M;\mathbf{Z}/2).
\]
Write
\[
v_M=\sum_{i=0}^n v_{M,i},\qquad v_{M,i}\in H^i(M;\mathbf{Z}/2).
\]
Combining the instability relations~\eqref{eq:instability_intro} with Poincar\'e duality, one obtains the familiar ``top-half'' vanishing:
\[
v_{M,i}=0\qquad \text{for all } i>n/2.
\]

To extract concrete congruences from this vanishing, one expresses the Wu class in terms of the Stiefel--Whitney class. Let
\[
\phi \colon H^*(M;\mathbf{Z}/2)\to H^{*+n}(TM,TM\setminus M;\mathbf{Z}/2)
\]
be the Thom isomorphism for the tangent bundle \(TM\), where \(M\subset TM\) denotes the zero section. The total Stiefel--Whitney class \(w_M\) is defined by
\[
w_M:=\phi^{-1}(\Sq(\phi(1))).
\]
Wu's theorem then reads as follows.

\begin{theorem}[Wu's theorem~{\cite{WU-PAPER}}, see also~{\cite[\S~11.6, Thm.~11.14]{MILNOR_STASHEFF}}]
With notation as above, the total Wu and Stiefel--Whitney classes satisfy
\[
\Sq(v_M)=w_M.
\]
\end{theorem}

Since \(\Sq\) is invertible, Wu's formula may be written
as \(v_M=\Sq^{-1}(w_M)\).
For a compact Riemann surface \(M\), the complex structure gives
\[
        w_{M,1}=0,
        \qquad
        w_{M,2}=\overline c_1(TM).
\]
Comparing degrees in Wu's formula yields
\(v_{M,1}=0\) and \(v_{M,2}=\overline c_1(TM)\).
Thus top-half vanishing implies \(\overline c_1(TM)=0\),
the characteristic-class expression of
\[
        \chi(M)=2-2g\equiv0\pmod2.
\]
Section~\ref{sec:applications_and_examples} develops these
connections, including a derivation of Hecke's theorem
on the different away from \(2\).

Arithmetic duality uses modified compactly supported cohomology,
which incorporates a Tate correction at infinity.
For a separated morphism of finite type
\(s\colon X\to\Spec\mathbf Z\), write
\(R\Gamma_{c,\mathrm{fin}}(X;\FF_2)
:=R\Gamma_{\acute et}(\Spec\mathbf Z;Rs_!\FF_2)\).
Let \(G_{\mathbf R}:=\Gal(\mathbf C/\mathbf R)\) act by
complex conjugation on the geometric compact-support complex
\(M_X:=R\Gamma_c(X_{\mathbf C};\FF_2)\).
Modified compact support is defined by
\begin{equation}
\label{eq:intro-modified-compact-support}
        \widehat R\Gamma_c(X;\FF_2)
        :=
        \operatorname{Fib}\!\left(
                R\Gamma_{c,\mathrm{fin}}(X;\FF_2)
                \longrightarrow M_X^{tG_{\mathbf R}}
        \right),
\end{equation}
where \(tG_{\mathbf R}\) denotes the Tate construction.
Its cohomology groups are denoted \(\widehat H_c^q(X;\FF_2)\).
Replacing the Tate construction by homotopy fixed points
\(M_X^{hG_{\mathbf R}}\) gives ordinary arithmetic compact
support, with cohomology groups \(H_c^q(X;\FF_2)\).
See Definition~\ref{def:compact-support-over-X}
in~\S\ref{sec:modified-compact-support}.
We identify mod-\(2\) Tate twists with \(\FF_2\) using
\(1\mapsto-1\), and write
\(\beta_Z\in H_{\acute et}^1(Z;\FF_2)\) for the Kummer
class of \(-1\), equivalently the Bockstein class of
Definition~\ref{def:bockstein_class}.
Write
\[
        H_{\acute et}^*(Z;\FF_2)^{\wedge,+}
        :=
        \prod_{n\geq0}H_{\acute et}^n(Z;\FF_2)
\]
for the forward-degree completion.
All identities involving total operations and characteristic
classes are interpreted degreewise in these completions.

\begin{theorem}[Arithmetic Wu formula]
\label{thm:main-wu}
Let \(B=\Spec\mathcal O_{K,S}\), where \(K\) is a number
field and \(2\in\mathcal O_{K,S}^{\times}\).
Let \(f\colon X\to B\) be flat and projective of pure
relative dimension \(d\), with \(X\) regular.

Modified compactly supported cohomology carries canonical
stable Steenrod operations
\[
        \widehat\Sq_c^r\colon
        \widehat H_c^q(X;\FF_2)
        \longrightarrow
        \widehat H_c^{q+r}(X;\FF_2)
        \qquad(q\in\mathbf Z,\ r\geq0).
\]
Under Artin--Verdier duality, these determine the total
absolute Wu class
\[
        v_X=\sum_{n\geq0}v_{X,n}
        \in H_{\acute et}^*(X;\FF_2)^{\wedge,+},
\]
characterized by
\[
        \int_X\widehat\Sq_c^n(x)
        =
        \int_X x\cdot v_{X,n}
        \qquad
        \bigl(x\in\widehat H_c^{2d+3-n}(X;\FF_2)\bigr).
\]
One has
\[
        v_X
        =
        \Sq^{-1}\!\left(
                w^{\acute et}\!\left(
                        \tau_f+\mathcal O_X^{\oplus3}
                \right)
        \right),
\]
where \(\tau_f\in K^0(X)\) is the virtual relative tangent
bundle, \(w^{\acute et}\) is the total \'etale
Stiefel--Whitney class, and \(\Sq^{-1}\) is the inverse
of the total Steenrod square on ordinary \'etale cohomology.
\end{theorem}

\begin{remark}\label{rem:intro-base-wu}
Set \(T(t):=1+\sum_{s\geq0}t^{2^s}\in\FF_2[[t]]\).
The absolute Wu class of the arithmetic base satisfies
\[
        v_B
        =
        T(\beta_B)^3
        =
        \Sq^{-1}\!\bigl(
                w^{\acute et}(\mathcal O_B^{\oplus3})
        \bigr).
\]
Thus, the arithmetic contribution is represented by three
trivial algebraic summands, whose total \'etale
Stiefel--Whitney class is \((1+\beta_B)^3\), leaving room for imagination.
\end{remark}

The augmented formulation follows from
Corollary~\ref{cor:2td-universal-polynomials}, using the
absolute formula and base computation of
Theorems~\ref{thm:absolute-wu-formula}
and~\ref{thm:absolute-arithmetic-wu-class}.
Over finite fields of odd characteristic, the formula is
\(v_X=\Sq^{-1}(w^{\acute et}(\tau_f))\), using ordinary
compact support and duality degree \(2d+1\).

Feng proved an \'etale Wu formula for smooth proper
geometrically connected varieties over finite fields of odd
characteristic~\cite[Theorem~6.5]{FengEtaleSteenrod}, and used
it to prove Tate's conjecture on the alternation of the
Artin--Tate pairing for surfaces.
Feng and the first author subsequently treated
characteristic \(2\) using a syntomic Wu formula~\cite{CARMELI-FENG}.
Our finite-field statement remains within the scope of
Feng's methods; the arithmetic case requires the construction
and analysis of stable operations on modified compact support.

\begin{remark}
\label{rem:feng_vs_us}
One may ask whether our arithmetic Wu formula follows from
Feng's finite-field theorem. Absolute Wu classes need not
commute with restriction, as illustrated by
\(\mathbf{RP}^2\hookrightarrow S^4\).
Define the \emph{relative Wu class} by
\begin{equation}
\label{eq:rel_wu_intro}
        v_{X/B}:=v_X\cdot(f^*v_B)^{-1}.
\end{equation}
Our Wu formula identifies this quotient with
\(\Sq^{-1}(w^{\acute et}(\tau_f))\), making it compatible
with Cartesian base changes preserving the hypotheses
and recovering the finite-field Wu formula on fibers
of good reduction.

Appendix~\ref{par:feng-not-formal} studies the potential
defect of this identity and shows that the kernel of
restriction to all closed fibers can be nonzero already
in degree \(2\). Thus, even vanishing of the defect on
every closed fiber does not by itself establish the
global arithmetic Wu formula.
\end{remark}

For our second main result, let \(B\) be either
\(\Spec(\mathcal O_{K,S})\), with \(2\) invertible, or the
spectrum of a finite field of odd characteristic.
Let \(f\colon X\to B\) be flat and projective of pure
relative dimension \(d\), with \(X\) regular, and put
\[
        \epsilon_B:=\dim B,
        \qquad
        D_X:=d+\epsilon_B,
        \qquad
        N_X:=2D_X+1.
\]
Define the total \'etale \(2\)-Todd class by
\[
        2td(\xi):=\Sq^{-1}\!\bigl(w^{\acute et}(\xi)\bigr)
        \qquad(\xi\in K^0(X)).
\]

In the arithmetic case, let
\[
        \gamma_X^q\colon
        H_c^q(X;\FF_2)\longrightarrow
        \widehat H_c^q(X;\FF_2)
\]
be the ordinary-to-modified comparison, and set
\[
        J_\infty^m(X)
        :=
        \bigl(\operatorname{im}\gamma_X^{N_X-m}\bigr)^\perp,
        \qquad
        J_\infty(X):=\bigoplus_{m\geq0}J_\infty^m(X),
\]
where orthogonality is taken under Artin--Verdier duality.
This is a graded ideal determined by the real-place
comparison maps; see~\S\ref{subsubsec:hecke-comparison}.
In the finite-field case, set \(J_\infty(X):=0\).

\begin{theorem}[Generalized Hecke theorem]
\label{thm:main-hecke}
For \(m\geq0\) and \(r\in\mathbf Z\), there are canonical
polynomials
\[
        P_{m,r}\in
        \FF_2[\beta,c_1,\dots,c_{\lfloor m/2\rfloor}],
\]
homogeneous of weighted degree \(m\), where
\(\deg\beta=1\) and \(\deg c_i=2i\), such that
\[
        2td(\xi)_m
        =
        P_{m,r}\bigl(\beta_X,\overline c_1(\xi),\dots\bigr)
\]
for every virtual bundle \(\xi\in K^0(X)\) of constant rank \(r\).
In particular,
\[
        v_X
        =
        2td\!\left(\tau_f+\mathcal O_X^{\oplus3\epsilon_B}\right).
\]
For every \(m\geq0\) with \(2m>N_X\), one has
\[
        v_{X,m}
        =
        P_{m,d+3\epsilon_B}
        \bigl(\beta_X,\overline c_1(\tau_f),\dots\bigr)
        \equiv0\pmod{J_\infty(X)}.
\]
These classes vanish over finite fields and, in the arithmetic
case, whenever \(K\) is totally imaginary or \(\beta_X=0\).
\end{theorem}

The polynomial formulas and congruences are proved in
Corollary~\ref{cor:2td-universal-polynomials} and
Theorem~\ref{thm:gen_hecke}.
Real places can produce nonzero classes in infinitely many
degrees: if \(X_\nu(\mathbf R)\neq\varnothing\) for some real
place \(\nu\), then
\[
        v_{X,2^s}\neq0
        \qquad(2^s\geq d+3).
\]
This follows from the real-place description in
\S\ref{subsubsec:real-place-wu}; see
Corollary~\ref{cor:real-place-wu-nonvanishing}.

\begin{example}[See also~\S\ref{subsubsec:untwisted-d-folds}]
Let \(f\colon X\to\Spec\mathbf Z[i,1/2]\) be flat and
projective of pure relative dimension \(4\), with \(X\) regular.
Here \(\beta_X=0\), \(J_\infty(X)=0\), and \(N_X=11\).
Writing \(\overline c_i:=\overline c_i(\tau_f)\), the
degree-\(6\) and degree-\(8\) Wu relations give
\[
        \overline c_1\,\overline c_2=0,
        \qquad
        \overline c_1^4+\overline c_1\,\overline c_3
        +\overline c_2^2+\overline c_4=0.
\]
\end{example}

\begin{remark}
The series \(T(t)\) of Remark~\ref{rem:intro-base-wu} satisfies
\[
        T(t)
        =
        1+\sum_{s\geq0}t^{2^s}
        =
        \left(\frac{2t}{1-e^{-2t}}\right)\bmod 2.
\]
Here the rescaled Todd series has \(2\)-integral coefficients,
and reduction is coefficientwise; see Lemma~\ref{lem:T-series}.
We therefore call \(T\) the \emph{\(2\)-Todd series}.

For a vector bundle \(E\) of rank \(r\), let
\(\pi\colon\operatorname{Fl}(E)\to X\) be its complete flag
bundle, with mod-\(2\) Chern roots \(x_1,\dots,x_r\), and put
\(\beta:=\pi^*\beta_X\). Then
\begin{equation}
\label{eq:intro-2td-roots}
        \pi^*2td(E)
        =
        T(\beta)^r
        \prod_{i=1}^r T\!\bigl(T(\beta)^{-2}x_i\bigr).
\end{equation}
For \(\beta_X=0\), this recovers Hirzebruch's Chern-root
expression \(\prod_iT(x_i)\) for \(\Sq^{-1}(w(E))\)
for complex vector bundles, where \(w(E)\) denotes the
total Stiefel--Whitney class of the underlying real bundle;
see~\cite{HIRZEBRUCH}.
\end{remark}

\subsection{Proof strategies}

We now sketch the proofs of Theorems~\ref{thm:main-wu} and~\ref{thm:main-hecke}. The arithmetic Wu formula is discussed in \S~\ref{subsubsec:main-wu}, and the generalized Hecke theorem in \S~\ref{subsubsec:main-hecke}.

\subsubsection{The arithmetic Wu formula}
\label{subsubsec:main-wu}

Here let \(B=\Spec\mathcal O_{K,S}\), with
\(2\in\mathcal O_{K,S}^{\times}\), and let \(f\colon X\to B\)
be flat and projective of pure relative dimension \(d\), with
\(X\) regular. Generalized Artin--Verdier duality
(Theorem~\ref{thm:milne-7-6}) gives perfect pairings
\[
        H_{\acute et}^r(X;\FF_2)
        \times
        \widehat H_c^{2d+3-r}
        (X;\mu_2^{\otimes(d+1)})
        \longrightarrow\FF_2.
\]
The absolute Wu classes represent the functionals obtained
by composing stable Steenrod operations with the trace.

\paragraph{Stable operations on arithmetic cohomology.}
Chapter~\ref{chapter:background} recalls relative cohomology
of pro-anima and its comparison with \'etale cohomology
with supports. We use the relative \'etale homotopy theory
developed by Harpaz--Schlank~\cite{HarpazSchlank2013}
and Barnea--Schlank~\cite{BarneaSchlank2016}, following the
\(\infty\)-categorical treatment of
Arad--Carmeli--Schlank~\cite[\S2]{ACS}.
Chapter~\ref{sec:stable-steenrod-arithmetic-boundary}
develops the corresponding construction with coefficient spectra.

Using relative \'etale shapes, we construct the boundary
diagram underlying~\eqref{eq:intro-modified-compact-support}
functorially in the coefficient spectrum.
Consequently, the stable maps of spectra
\[
        \operatorname{sq}^r\colon
        H\FF_2\longrightarrow\Sigma^rH\FF_2
        \qquad(r\geq0)
\]
act compatibly on all its terms.
Theorem~\ref{thm:canonical-arithmetic-boundary-operations}
establishes compactification independence, naturality,
and the Adem relations; the internal and mixed Cartan
formulas are proved in
Theorem~\ref{thm:arithmetic-boundary-internal-mixed-cartan}.

On ordinary relative cohomology these operations recover
the classical Steenrod squares and satisfy instability.
On modified compact support, instability can fail:
Theorem~\ref{thm:instability-modulo-tate-boundary}
places the defects in the image of the Tate boundary.

\paragraph{The relative formula and the arithmetic base.}
The principal geometric input is the modified compact-support
relative Wu formula. Let \(g\colon Y\to Z\) be a proper
smoothable local complete intersection morphism of pure virtual
relative dimension \(-c\), where \(Y\) and \(Z\) are regular
separated schemes of finite type over \(B\) admitting ample
invertible sheaves. Riou's Gysin morphisms induce pushforwards
\[
        \widehat g_{*,c}\colon
        \widehat H_c^q(Y;\mu_2^{\otimes r})
        \longrightarrow
        \widehat H_c^{q+2c}(Z;\mu_2^{\otimes(r+c)}).
\]
Theorem~\ref{thm:modified-compactly-supported-relative-wu-declared}
gives
\begin{equation}
\label{eq:mod-bdry-rel-wu-intro}
        \widehat\Sq_c\bigl(\widehat g_{*,c}(x)\bigr)
        =
        \widehat g_{*,c}\!\left(
                \widehat\Sq_c(x)\cdot w^{\acute et}(\nu_g)
        \right),
\end{equation}
where \(\nu_g\) is the virtual normal bundle of \(g\).

Following the geometric strategy of Benoist's ordinary
relative Wu formula~\cite[Theorem~2.5]{Benoist}, we factor
\(g\) into a regular closed immersion followed by a
projective-space bundle. The immersion case follows from
the Thom-class characterization and the supported mixed
Cartan formula. The projective-space case uses the projective
bundle formula and Benoist's ordinary identity.
Compatibility with composition then gives the general formula.

Applied to \(f\colon X\to B\), this relative formula,
together with the projection and trace formulas, shows that
\[
        \Sq^{-1}\!\bigl(w^{\acute et}(\tau_f)\bigr)\cdot f^*v_B
\]
satisfies the defining Wu identity. Artin--Verdier duality
therefore identifies it with \(v_X\).

It remains to determine \(v_B\). In degrees at least \(3\),
duality reduces this computation to the real-place Tate
cohomology rings \(\FF_2[\beta_\nu^{\pm1}]\), on which
\[
        \Sq^r(\beta_\nu^a)
        =
        \binom{a}{r}\beta_\nu^{a+r}
        \qquad(a\in\mathbf Z,\ r\geq0).
\]
Together with the low-degree calculation, this yields
\[
        v_B
        =
        \sum_{n\geq0}\binom{2-n}{n}\beta_B^n
        =
        T(\beta_B)^3;
\]
see~Theorem~\ref{thm:absolute-arithmetic-wu-class} and
Lemma~\ref{lem:T-power-coefficients}.
The generalized integral binomial coefficients are read
modulo \(2\).

\subsubsection{The generalized Hecke theorem}
\label{subsubsec:main-hecke}
The Chern-root formula~\eqref{eq:intro-2td-roots}, proved in
Theorem~\ref{thm:etale-2td-splitting}, reduces to the line-bundle case
by multiplicativity and the splitting principle. The line-bundle computation itself is carried out
in Lemma~\ref{lem:2td-line-bundle}. 

Since the Chern-root formula is invariant under permutations of roots, 
the 2-Todd class can be expanded in terms of symmetric polynomials,
which yields the polynomials \(P_{m,r}\) (see Corollary~\ref{cor:2td-universal-polynomials}).

Recall our notation $\epsilon_B := \dim B$. 
Since \(2td(\mathcal O_X)=T(\beta_X)\), the base contribution
\(T(\beta_X)^{3\epsilon_B}\) is absorbed by adding
\(3\epsilon_B\) trivial summands to \(\tau_f\).
This produces our particular polynomial expression for \(v_X\).

It remains to establish the Hecke relations. Over finite
fields, the classical instability relations recalled in
Proposition~\ref{prop:classic_sq_properties}, together with
duality, give \(v_{X,m}=0\) for \(2m>N_X\).
The same argument applies in the arithmetic case when \(K\)
is totally imaginary or \(\beta_X=0\): the Tate boundary
vanishes, so modified compact support satisfies instability
by Theorem~\ref{thm:instability-modulo-tate-boundary}.

For a general arithmetic base, this vanishing holds modulo
\(J_\infty(X)\). For \(2m>N_X\), put \(q=N_X-m\).
By definition, it suffices to show
\(v_{X,m}\perp\operatorname{im}\gamma_X^q\).
If \(q<0\), the source group is zero. Otherwise, for
\(x\in H_c^q(X;\FF_2)\), the Wu identity and compatibility
with stable operations give
\[
        \int_X\gamma_X(x)\cdot v_{X,m}
        =
        \int_X\widehat\Sq_c^m(\gamma_X(x))
        =
        \int_X\gamma_X\!\bigl(\Sq_c^m(x)\bigr)
        =0.
\]
The last equality follows from \(m>q\) and ordinary
compact-support instability
(Proposition~\ref{prop:instability-relative-pro-anima-terms}).
Thus \(v_{X,m}\in J_\infty^m(X)\).

\begin{remark}
For each real place \(\nu\) of \(K\), put
\(X_\nu:=X\times_B\Spec(\mathbf R)\).
Proposition~\ref{prop:archimedean-comparison-image}
identifies \(J_\infty^m(X)\) with the image of the canonical
archimedean map of degree \(3\)
\[
        a_\infty\colon
        \bigoplus_{\nu\text{ real}}
        H_{G_{\mathbf R}}^{m-3}
        (X_\nu(\mathbf C);\FF_2)
        \longrightarrow H_{\acute et}^m(X;\FF_2),
\]
where \(H_{G_{\mathbf R}}^*\) denotes ordinary Borel
equivariant cohomology for complex conjugation.
By Lemma~\ref{lem:archimedean-comparison-ideal},
\(J_\infty^m(X)=H_{\acute et}^m(X;\FF_2)\) for \(m\geq N_X\).
Consequently, the Hecke congruences impose potentially
nontrivial restrictions only for
\[
        d+2\leq m<N_X=2d+3.
\]
In \S\ref{subsubsec:real-place-wu}, we describe the Wu classes
in degrees \(m\geq N_X\) explicitly in terms of the real loci.
\end{remark}

\subsection{Organization of the paper}
Chapter~\ref{chapter:background} recalls pro-anima, \'etale
shapes, classical relative Steenrod squares, and ordinary
and modified arithmetic compact support.

Chapter~\ref{sec:stable-steenrod-arithmetic-boundary}
constructs stable Steenrod operations on the arithmetic
boundary diagram, establishes their formal properties,
and analyzes instability at the Tate boundary.

Chapter~\ref{chapter:char_classes} develops \'etale
Stiefel--Whitney and Wu classes, proves the relative and
absolute Wu formulas, and computes the Wu class of the
arithmetic base.

Chapter~\ref{chap:generalised-hecke} establishes the generalized
Hecke theorem, describes its archimedean defects, and gives
specializations and examples.

Appendix~\ref{par:feng-not-formal} examines the limits of
recovering a global Wu formula from identities on closed fibers.

\section*{Acknowledgments}

Mark Shusterman's and Sa'ar Zehavi's research is co-funded by the European Union (ERC, Function Fields, 101161909).

Mark Shusterman is The Dr.\ A. Edward Friedmann Career Development Chair in Mathematics.

Shachar Carmeli is partially supported by ISF Beresheet grant 4093/25, BSF grant 2024766, Minerva grant 715294, and the Azrieli foundation through an Early Career Researcher Fellowship.

We thank Eric Chen, Tony Feng, and Artane Siad for helpful discussions. 
We also thank the anonymous referee for identifying a gap in an
earlier version of this manuscript concerning our use of
the instability relations for Steenrod squares.
\section{Background}\label{chapter:background}
This chapter supplies the cohomological framework for the
stable Steenrod operations constructed in
\S\ref{sec:stable-steenrod-arithmetic-boundary}.
We recall relative cohomology of pro-anima and classical
Steenrod squares in~\S\ref{subsec:steenrod_basic},
\'etale homotopy types and support comparisons
in~\S\ref{subsec:schlank-barnea}, and arithmetic compact
support and the Tate boundary
in~\S\ref{sec:modified-compact-support}.

\subsection{Anima, cochains, and classical Steenrod squares}
\label{subsec:steenrod_basic}

We fix our conventions for pro-anima and relative cohomology,
then recall the classical Steenrod squares on simplicial pairs.

Let \(\mathcal S:=\mathrm{Ani}\) denote the \(\infty\)-category
of anima, and let \(\Pro(\mathcal S)\) denote its pro-category; see~\cite[\S2.1]{CarchediElmanto}.
We work intrinsically in these categories, using simplicial
sets as concrete presentations of anima.

\begin{definition}[Derived categories;
cf.~{\cite[\S1.3.5]{LurieHA}}]
\label{def:stable-derived-category}
For a commutative ring \(R\), let \(\mathcal D(R)\) denote the
stable \(\infty\)-category obtained from unbounded complexes
of \(R\)-modules by inverting quasi-isomorphisms. Its homotopy
category is the triangulated derived category \(D(R)\).

For a Grothendieck topos \(\mathcal T\), let
\(\mathcal T_\infty\) denote its associated \(\infty\)-topos
of anima-valued sheaves, without hypercompletion.
Following~\cite[\S2.3]{ACS}, put
\[
        \mathcal D(\mathcal T,R)
        :=
        \operatorname{Shv}(\mathcal T_\infty;\mathcal D(R)),
\]
the stable \(\infty\)-category of \(\mathcal D(R)\)-valued
sheaves. A superscript \(b\) denotes the full subcategory
of objects bounded for the standard \(t\)-structure.
\end{definition}

We regard ordinary \(R\)-module sheaves as objects of the heart.
Sheaf functors, fibers, and colimits below are interpreted
in the corresponding stable \(\infty\)-categories.

\subsubsection{Relative cohomology of pro-anima}
\label{subsubsec:rel_coh_pro_space}

For an anima \(K\) and an abelian group \(\Lambda\), write
\(C^\bullet(K;\Lambda)\in\mathcal D(\mathbf Z)\) for its
ordinary cochains, computed by normalized cochains on a
simplicial presentation of \(K\).

An object \(X\in\Pro(\mathcal S)\) admits a presentation
\(X\simeq\{X_i\}_{i\in I}\) indexed by a small cofiltered
\(\infty\)-category. We put
\[
        C^\bullet(X;\Lambda)
        :=
        \varinjlim_{i\in I^{\mathrm{op}}}
        C^\bullet(X_i;\Lambda),
        \qquad
        H^n(X;\Lambda):=H^nC^\bullet(X;\Lambda).
\]
This construction is independent of the presentation:
it is the extension of
\(C^\bullet(-;\Lambda)\colon\mathcal S^{\mathrm{op}}
\to\mathcal D(\mathbf Z)\)
to \(\Pro(\mathcal S)^{\mathrm{op}}\) preserving filtered
colimits, supplied by the universal property of
Ind-completion; see~\cite[Proposition~5.3.5.10]{HTT}.
Exactness of filtered colimits gives
\[
        H^n(X;\Lambda)
        \simeq
        \varinjlim_{i\in I^{\mathrm{op}}}
        H^n(X_i;\Lambda).
\]

\begin{definition}[Relative cohomology]
Let \(f\colon A\to X\) be a morphism in
\(\Pro(\mathcal S)\). Its relative cochains and cohomology are
\[
        C^\bullet(f;\Lambda)
        :=
        \operatorname{fib}\!\left(
                C^\bullet(X;\Lambda)
                \xrightarrow{f^*}
                C^\bullet(A;\Lambda)
        \right),
        \qquad
        H^n(f;\Lambda):=H^nC^\bullet(f;\Lambda).
\]
\end{definition}

The defining fiber sequence yields the relative long exact
sequence
\[
        \cdots\longrightarrow H^{n-1}(A;\Lambda)
        \longrightarrow H^n(f;\Lambda)
        \longrightarrow H^n(X;\Lambda)
        \longrightarrow H^n(A;\Lambda)
        \longrightarrow\cdots.
\]
In particular,
\(C^\bullet(\varnothing\to X;\Lambda)
\simeq C^\bullet(X;\Lambda)\).

Since \([1]\) is a finite poset, \(f\) admits a level
presentation \(f_i\colon A_i\to X_i\) over a common small
cofiltered \(\infty\)-category \(I\);
see, dually,~\cite[Proposition~5.3.5.15]{HTT}.
Filtered colimits commute with fibers, so
\[
\begin{aligned}
        C^\bullet(f;\Lambda)
        &\simeq
        \varinjlim_{i\in I^{\mathrm{op}}}
        C^\bullet(f_i;\Lambda),
        \\
        H^n(f;\Lambda)
        &\simeq
        \varinjlim_{i\in I^{\mathrm{op}}}
        H^n(f_i;\Lambda).
\end{aligned}
\]
For a simplicial presentation of \(f_i\), its relative
cochains are computed by the mapping-cylinder pair
\((M(f_i),A_i)\). If \(f_i\) is a simplicial inclusion,
one may instead use the pair \((X_i,A_i)\).
Thus, the definition recovers ordinary relative cohomology.
In particular,
\[
        H^n(f;\Lambda)=0
        \qquad(n<0).
\]

\subsubsection{Classical relative Steenrod squares}
\label{subsubsec:classic_steenrod}

Let \(L\hookrightarrow K\) be a simplicial inclusion.
Its normalized relative cochains are
\[
        N^\bullet(K,L;\FF_2)
        :=
        \ker\!\left(
                N^\bullet(K;\FF_2)
                \longrightarrow N^\bullet(L;\FF_2)
        \right).
\]
Since restriction is degreewise surjective, this complex
represents \(C^\bullet(L\to K;\FF_2)\).

Normalized simplicial cochains carry natural cup-\(i\)
products, with \(\smile_0\) the cup product and
\(\smile_i=0\) for \(i<0\);
see~\cite[\S3.3 and Remark~3.6]{FengEtaleSteenrod}. 
We write \(\smile\,:=\,\smile_0\) and also use \(\cdot\) for the
induced product on cohomology.
Naturality implies that these products preserve relative
cochains.

\begin{definition}[Classical relative Steenrod squares;
cf.~{\cite[\S3.4]{FengEtaleSteenrod}}]
\label{def:classic_sq}
For \(n,r\geq0\) and a cocycle
\(u\in N^n(K,L;\FF_2)\), define
\[
        \Sq^r([u])
        :=
        [u\smile_{n-r}u]
        \in H^{n+r}(K,L;\FF_2).
\]
\end{definition}

We recall the standard relative forms of the classical properties.

\begin{proposition}[Classical properties;
cf.~{\cite[\S3.4]{FengEtaleSteenrod}}]
\label{prop:classic_sq_properties}
Let \(x\in H^n(K,L;\FF_2)\) and
\(y\in H^p(K,L;\FF_2)\), with \(n,p\geq0\).
The relative Steenrod squares satisfy the following properties.
\begin{enumerate}
\item[(i)] \textup{(Additivity)}
For \(r\geq0\) and \(x,x'\in H^n(K,L;\FF_2)\),
\[
        \Sq^r(x+x')=\Sq^r(x)+\Sq^r(x').
\]

\item[(ii)] \textup{(Naturality)}
For \(r\geq0\), a map of simplicial pairs \(f\colon(X,A)\to(Y,B)\)
and \(z\in H^n(Y,B;\FF_2)\), one has
\[
        f^*\Sq^r(z)=\Sq^r(f^*z).
\]

\item[(iii)] \textup{(Stability)}
For simplicial inclusions \(J\hookrightarrow L\hookrightarrow K\),
let
\[
        \delta\colon H^n(L,J;\FF_2)
        \longrightarrow H^{n+1}(K,L;\FF_2)
\]
be the connecting morphism of the triple.
Then, for \(r\geq0\) and \(z\in H^n(L,J;\FF_2)\),
\[
        \Sq^r(\delta z)=\delta(\Sq^r z).
\]

\item[(iv)] \textup{(Cartan formula)}
For \(r\geq0\),
\[
        \Sq^r(x\smile y)
        =
        \sum_{\substack{a,b\geq0\\a+b=r}}
        \Sq^a(x)\smile\Sq^b(y).
\]

\item[(v)] \textup{(Adem relations)}
For integers \(0<a<2b\),
\[
        \Sq^a\Sq^b
        =
        \sum_{j=0}^{\lfloor a/2\rfloor}
        \binom{b-j-1}{a-2j}
        \Sq^{a+b-j}\Sq^j,
\]
with binomial coefficients read modulo \(2\).

\item[(vi)] \textup{(Normalization)}
One has \(\Sq^0=\mathrm{id}\).

\item[(vii)] \textup{(Top square)}
One has \(\Sq^n(x)=x\smile x\).

\item[(viii)] \textup{(Instability)}
One has \(\Sq^r(x)=0\) for \(r>n\).

\item[(ix)] \textup{(Bockstein)}
The operation \(\Sq^1\) is the connecting morphism associated
with the coefficient sequence
\[
        0\longrightarrow\mathbf Z/2
        \xrightarrow{\,2\,}\mathbf Z/4
        \longrightarrow\mathbf Z/2
        \longrightarrow0.
\]
\end{enumerate}
\end{proposition}

In~\S\ref{subsec:spectral-relative-steenrod-construction},
we realize these squares by stable maps of coefficient
spectra and extend them to relative cochain spectra.
The operations of
Definition~\ref{def:spectral-relative-steenrod-operations}
recover the classical squares on simplicial pairs and
their filtered colimits on pro-anima.
Consequently, the instability relations above hold on
ordinary relative cohomology of pro-anima;
see~Proposition~\ref{prop:instability-relative-pro-anima-terms}.

\subsection{\'Etale homotopy types and cohomology with supports}
\label{subsec:schlank-barnea}

We recall the intrinsic shape construction and its cohomology
comparison, then identify cohomology with support conditions
with relative cohomology of pro-anima. For the shape formalism, we follow
Arad--Carmeli--Schlank~\cite[\S2]{ACS}.

\subsubsection{Shapes and cohomology}
\label{subsubsec:bs_topological_type}

Let \(\mathcal T\) be a Grothendieck topos, with associated
\(\infty\)-topos \(\mathcal T_\infty\) as above.
Write
\[
        \Gamma_{\mathcal T}^*\colon
        \mathcal S\longrightarrow\mathcal T_\infty
\]
for the constant-object functor.
The induced functor \(\Pro(\Gamma_{\mathcal T}^*)\)
admits a left adjoint \((\Gamma_{\mathcal T})_\sharp\);
see~\cite[Proposition~2.0.2]{ACS}.

Below, we implicitly identify
the objects of a category \(\mathcal C\) with their images under the fully faithful
constant embedding \(\mathcal C\hookrightarrow\Pro(\mathcal C)\).
\begin{definition}[Shape]
The shape of \(\mathcal T\) is
\[
        h(\mathcal T)
        :=
        (\Gamma_{\mathcal T})_\sharp(1_{\mathcal T})
        \in\Pro(\mathcal S),
\]
where \(1_{\mathcal T}\) is the terminal object.
It is characterized by natural equivalences
\[
        \operatorname{Map}_{\Pro(\mathcal S)}
        \bigl(h(\mathcal T),K\bigr)
        \simeq
        \operatorname{Map}_{\mathcal T_\infty}
        \bigl(1_{\mathcal T},\Gamma_{\mathcal T}^*K\bigr)
        \qquad(K\in\mathcal S);
\]
see~\cite[Definition~2.8]{CarchediElmanto}.
\end{definition}

For a scheme \(Y\), write
\[
        h(Y)
        :=
        \Pi_\infty\operatorname{Shv}(Y_{\acute et})
        \in\Pro(\mathcal S)
\]
for its \emph{\'etale homotopy type}, where
\(\operatorname{Shv}(Y_{\acute et})\) is the
\(\infty\)-topos of anima-valued sheaves on the small
\'etale site and \(\Pi_\infty\) denotes the shape functor.
This construction is functorial in \(Y\);
see~\cite[Corollary~2.12]{CarchediElmanto}.

\begin{proposition}[Comparison with topos cohomology]
\label{prop:chough_comp}
For every abelian group \(\Lambda\), there is a canonical
equivalence
\[
        \chi_{\mathcal T}\colon
        C^\bullet(h(\mathcal T);\Lambda)
        \xrightarrow{\ \sim\ }
        R\Gamma(\mathcal T;\underline{\Lambda}_{\mathcal T})
\]
in \(\mathcal D(\mathbf Z)\), natural in geometric morphisms
and in \(\Lambda\).
Here \(\underline{\Lambda}_{\mathcal T}\) denotes the
constant abelian sheaf.
In particular,
\[
        H^n(h(Y);\Lambda)
        \simeq
        H_{\acute et}^n(Y;\underline{\Lambda}_Y).
\]
\end{proposition}

\begin{proof}
For \(n\geq0\), apply the defining equivalence of the shape
to the Eilenberg--Mac Lane anima \(K(\Lambda,n)\).
Since inverse image preserves Eilenberg--Mac Lane objects,
this gives
\[
        \operatorname{Map}_{\Pro(\mathcal S)}
        \bigl(h(\mathcal T),K(\Lambda,n)\bigr)
        \simeq
        \operatorname{Map}_{\mathcal T_\infty}
        \bigl(1_{\mathcal T},
              K(\underline{\Lambda}_{\mathcal T},n)\bigr).
\]
These equivalences respect the Eilenberg--Mac Lane
structures, including deloopings and module structures,
and therefore assemble into the asserted equivalence
of cochain objects. Naturality follows from the
naturality of the shape construction.
\end{proof}

\subsubsection{Relative shapes and relative cohomology}
\label{subsubsec:maps_topoi_relcoh}

Let \(g\colon\mathcal Z\to\mathcal T\) be a geometric
morphism. We use the same notation for its induced
morphism of \(\infty\)-topoi.
By~\cite[Proposition~2.0.2 and Definition~2.0.3]{ACS},
there is an adjunction
\[
        g_\sharp\colon
        \Pro(\mathcal Z_\infty)
        \rightleftarrows
        \Pro(\mathcal T_\infty)
        \mathbin{:}\Pro(g^*).
\]
The \emph{relative shape} of \(g\) is
\[
        g_\sharp(1_{\mathcal Z})
        \in\Pro(\mathcal T_\infty).
\]
Its defining property is
\[
        \operatorname{Map}_{\Pro(\mathcal T_\infty)}
        \bigl(g_\sharp(1_{\mathcal Z}),F\bigr)
        \simeq
        \operatorname{Map}_{\mathcal Z_\infty}
        \bigl(1_{\mathcal Z},g^*F\bigr)
        \qquad(F\in\mathcal T_\infty).
\]

Applying absolute shape to \(g\) gives a morphism
\[
        h(g)\colon h(\mathcal Z)\longrightarrow h(\mathcal T)
\]
in \(\Pro(\mathcal S)\).
The relative cochains of this arrow are those of
\S\ref{subsubsec:rel_coh_pro_space}.
By Proposition~\ref{prop:chough_comp}, they satisfy
\[
        C^\bullet(h(g);\Lambda)
        \simeq
        \operatorname{fib}\!\left(
                R\Gamma(\mathcal T;
                        \underline{\Lambda}_{\mathcal T})
                \longrightarrow
                R\Gamma(\mathcal Z;
                        \underline{\Lambda}_{\mathcal Z})
        \right),
\]
naturally in \(g\) and \(\Lambda\).

The relative shape retains information about the base
topos. In particular, for an \(\mathbf R\)-scheme \(Y\)
with structural geometric morphism \(p_Y\), it gives
\[
        h_{\mathbf R}(Y)
        :=
        (p_Y)_\sharp(1)
        \in
        \Pro\!\left(
                \operatorname{Shv}(\mathbf R_{\acute et})
        \right)
        \simeq
        \Pro\!\left(
                \operatorname{Fun}(BG_{\mathbf R},\mathcal S)
        \right),
        \qquad
        G_{\mathbf R}:=\operatorname{Gal}(\mathbf C/\mathbf R);
\]
see~\cite[Definition~4.6 and Proposition~4.9]{CarchediElmanto}.
This is the equivariant pro-anima used in
\S\ref{subsec:arithmetic-boundary-spectra-comparison},
where its comparison with \(h(Y_{\mathbf C})\) is established.

\subsubsection{Open--closed decompositions and support conditions}
\label{subsubsec:open_closed_compact_support}

Let
\[
        j\colon\mathcal U\hookrightarrow\mathcal T,
        \qquad
        i\colon\mathcal Z\hookrightarrow\mathcal T
\]
be an open--closed decomposition of Grothendieck topoi,
and let \(\Lambda\) be an abelian group.

\begin{definition}[Cohomology with support conditions]
Define objects of \(\mathcal D(\mathbf Z)\) by
\[
\begin{aligned}
        R\Gamma_c(\mathcal U/\mathcal T;
                  \underline{\Lambda}_{\mathcal U})
        &:=
        R\Gamma(\mathcal T;
                j_!\underline{\Lambda}_{\mathcal U}),
        \\
        R\Gamma_{\mathcal Z}(\mathcal T;
                            \underline{\Lambda}_{\mathcal T})
        &:=
        \operatorname{fib}\!\left(
                R\Gamma(\mathcal T;
                        \underline{\Lambda}_{\mathcal T})
                \longrightarrow
                R\Gamma(\mathcal U;
                        \underline{\Lambda}_{\mathcal U})
        \right).
\end{aligned}
\]
Here \(j_!\) is extension by zero.
Their cohomology groups are denoted by
\(H_c^n(\mathcal U/\mathcal T;
\underline{\Lambda}_{\mathcal U})\) and
\(H_{\mathcal Z}^n(\mathcal T;
\underline{\Lambda}_{\mathcal T})\), respectively.
\end{definition}

The second complex also identifies with
\[
        R\Gamma(\mathcal Z;
                Ri^!\underline{\Lambda}_{\mathcal T}),
\]
where \(Ri^!\) is right adjoint to \(i_*\).

\begin{proposition}[Relative comparison with local cohomology]
\label{prop:relative-comparison}
There are canonical equivalences in \(\mathcal D(\mathbf Z)\)
\[
\begin{aligned}
        \alpha_c\colon
        C^\bullet(h(i);\Lambda)
        &\xrightarrow{\ \sim\ }
        R\Gamma_c(\mathcal U/\mathcal T;
                  \underline{\Lambda}_{\mathcal U}),
        \\
        \alpha_Z\colon
        C^\bullet(h(j);\Lambda)
        &\xrightarrow{\ \sim\ }
        R\Gamma_{\mathcal Z}(\mathcal T;
                            \underline{\Lambda}_{\mathcal T}).
\end{aligned}
\]
They are natural in \(\Lambda\) and in Cartesian morphisms
of open--closed decompositions.
\end{proposition}

\begin{proof}
The open--closed recollement gives fiber sequences
\[
        R\Gamma_c(\mathcal U/\mathcal T;
                  \underline{\Lambda}_{\mathcal U})
        \longrightarrow
        R\Gamma(\mathcal T;
                \underline{\Lambda}_{\mathcal T})
        \longrightarrow
        R\Gamma(\mathcal Z;
                \underline{\Lambda}_{\mathcal Z})
\]
and
\[
        R\Gamma_{\mathcal Z}(\mathcal T;
                            \underline{\Lambda}_{\mathcal T})
        \longrightarrow
        R\Gamma(\mathcal T;
                \underline{\Lambda}_{\mathcal T})
        \longrightarrow
        R\Gamma(\mathcal U;
                \underline{\Lambda}_{\mathcal U}).
\]
In the \'etale setting, see~\cite[Tags~095L and~09XP]
{STACKS-PROJECT}.
Taking fibers of the natural comparisons of
Proposition~\ref{prop:chough_comp} gives
\(\alpha_c\) and \(\alpha_Z\).
Their naturality follows from that of the absolute
comparison, the recollement sequences, and the fiber
construction.
\end{proof}

For \(\Lambda=\FF_2\), the Steenrod operations on these
ordinary support groups are the operations of
Definition~\ref{def:spectral-relative-steenrod-operations},
transported through the preceding comparisons.
We denote them by \(\Sq_c^r\) and \(\Sq_{\mathcal Z}^r\),
respectively. Their extension to arithmetic boundary
cohomology is developed in
\S\ref{sec:stable-steenrod-arithmetic-boundary}.

\subsection{Modified compactly supported cohomology over arithmetic bases}
\label{sec:modified-compact-support}

We review the Tate construction and its role in ordinary
and modified arithmetic compact support; our main reference is
Geisser--Schmidt~\cite[\S2]{GeisserSchmidt}.
Retain \(G_{\mathbf R}\), and let
\[
        v\colon\Spec(\mathbf C)\longrightarrow\Spec(\mathbf Z)
\]
be the geometric point at infinity.

\subsubsection{The Tate construction and the real place}
\label{sssec:tate-real-place}

\begin{definition}[The Tate construction;
cf.~{\cite[\S I.1]{NikolausScholze2018}}]
\label{def:tate-construction}
Let \(G\) be a finite group.  Put
\[
        \mathcal D(\mathbf Z)^{BG}
        :=
        \operatorname{Fun}(BG,\mathcal D(\mathbf Z)).
\]
For \(M\in\mathcal D(\mathbf Z)^{BG}\), set
\[
        M_{hG}:=\operatorname*{colim}_{BG}M,
        \qquad
        M^{hG}:=\operatorname*{lim}_{BG}M.
\]
Let
\[
        N_M\colon M_{hG}\longrightarrow M^{hG}
\]
be the norm morphism.  The Tate construction is
\[
        M^{tG}
        :=
        \operatorname{cofib}\!\left(
        M_{hG}\xrightarrow{\,N_M\,}M^{hG}
        \right).
\]
The canonical ordinary-to-Tate morphism is denoted
\[
        \operatorname{can}_M\colon M^{hG}\longrightarrow M^{tG}.
\]
\end{definition}

\begin{definition}[Ordinary and modified cohomology over
\(\Spec\mathbf R\);
cf.~{\cite[Definition~2.1 and (2.5)]{GeisserSchmidt}}]
\label{def:modified-specR}
Under the canonical equivalence
\[
        \mathcal D((\Spec\mathbf R)_{\acute et},\mathbf Z)
        \simeq
        \operatorname{Fun}(BG_{\mathbf R},\mathcal D(\mathbf Z)),
\]
view
\[
        \mathcal H^\bullet\in
        \mathcal D^b((\Spec\mathbf R)_{\acute et},\mathbf Z)
\]
as an object of \(\mathcal D(\mathbf Z)^{BG_{\mathbf R}}\).
Its ordinary and modified cohomology complexes are
\[
\begin{aligned}
        R\Gamma_{\acute et}(\Spec\mathbf R,\mathcal H^\bullet)
        &\simeq(\mathcal H^\bullet)^{hG_{\mathbf R}},
        \\
        \widehat R\Gamma_{\acute et}
                (\Spec\mathbf R,\mathcal H^\bullet)
        &:=(\mathcal H^\bullet)^{tG_{\mathbf R}}.
\end{aligned}
\]
Their degree-\(n\) cohomology groups are denoted by
\(H^n_{\acute et}(\Spec\mathbf R;\mathcal H^\bullet)\) and
\(\widehat H^n_{\acute et}(\Spec\mathbf R;\mathcal H^\bullet)\).
\end{definition}

\subsubsection{Arithmetic boundary complexes}
\label{sssec:arithmetic-boundary-complexes}

\begin{definition}[Arithmetic boundary complexes over
\(\Spec\mathbf Z\);
cf.~{\cite[Definition~2.5]{GeisserSchmidt}} and
{\cite[Proposition~6.13]{FlachMorin}}]
\label{def:compact-support-over-specZ}
Let
\[
        \mathcal G^\bullet\in
        \mathcal D^b((\Spec\mathbf Z)_{\acute et},\mathbf Z)
\]
and put
\[
        R\Gamma_{\mathrm{fin}}(\mathcal G^\bullet)
        :=
        R\Gamma_{\acute et}(\Spec\mathbf Z,\mathcal G^\bullet),
        \qquad
        M_\infty(\mathcal G^\bullet)
        :=
        R\Gamma_{\acute et}(\Spec\mathbf C,v^*\mathcal G^\bullet)
        \in
        \mathcal D(\mathbf Z)^{BG_{\mathbf R}}.
\]
Define the ordinary and Tate archimedean boundary complexes by
\[
        R\Gamma_\infty(\mathcal G^\bullet)
        :=
        M_\infty(\mathcal G^\bullet)^{hG_{\mathbf R}},
        \qquad
        \widehat R\Gamma_\infty(\mathcal G^\bullet)
        :=
        M_\infty(\mathcal G^\bullet)^{tG_{\mathbf R}}.
\]
Restriction along \(v\colon\Spec\mathbf C\to\Spec\mathbf Z\) gives, by adjunction,
the ordinary boundary morphism
\[
        \delta^{\mathrm{ord}}_{\mathcal G}\colon
        R\Gamma_{\mathrm{fin}}(\mathcal G^\bullet)
        \longrightarrow
        R\Gamma_\infty(\mathcal G^\bullet).
\]
The Tate and modified boundary morphisms are
\[
        \delta^{\mathrm{Tate}}_{\mathcal G}
        :=
        \operatorname{can}_{M_\infty(\mathcal G^\bullet)}
        \colon
        R\Gamma_\infty(\mathcal G^\bullet)
        \longrightarrow
        \widehat R\Gamma_\infty(\mathcal G^\bullet),
        \qquad
        \delta^{\mathrm{mod}}_{\mathcal G}
        :=
        \delta^{\mathrm{Tate}}_{\mathcal G}\circ
        \delta^{\mathrm{ord}}_{\mathcal G}.
\]
Finally define
\[
\begin{array}{rcl}
R\Gamma_{c}(\mathcal G^\bullet)
&:=&
\operatorname{Fib}\!\left(
R\Gamma_{\mathrm{fin}}(\mathcal G^\bullet)
\xrightarrow{\ \delta^{\mathrm{ord}}_{\mathcal G}\ }
R\Gamma_\infty(\mathcal G^\bullet)
\right),
\\[0.8em]
\widehat R\Gamma_{\acute et}(\Spec\mathbf Z,\mathcal G^\bullet)
&:=&
\operatorname{Fib}\!\left(
R\Gamma_{\mathrm{fin}}(\mathcal G^\bullet)
\xrightarrow{\ \delta^{\mathrm{mod}}_{\mathcal G}\ }
\widehat R\Gamma_\infty(\mathcal G^\bullet)
\right).
\end{array}
\]
\end{definition}

\begin{proposition}[Functoriality]
\label{prop:modified-functoriality}
The functors \(R\Gamma_{\mathrm{fin}}\),
\(R\Gamma_\infty\), \(\widehat R\Gamma_\infty\),
\(R\Gamma_c\), and
\(\widehat R\Gamma_{\acute et}(\Spec\mathbf Z,-)\)
are exact from
\(\mathcal D^b((\Spec\mathbf Z)_{\acute et},\mathbf Z)\)
to \(\mathcal D(\mathbf Z)\).
\end{proposition}

\begin{proof}
Derived global sections, inverse image, homotopy orbits,
and homotopy fixed points are exact.
Naturality of the norm and restriction morphisms,
together with exactness of fibers and cofibers in stable
\(\infty\)-categories, gives the assertion.
\end{proof}

\subsubsection{Compact support over arithmetic schemes}
\label{sssec:compact-support-arithmetic-schemes}
Let \(s\colon X\to\Spec(\mathbf Z)\) be separated and of
finite type.

\begin{definition}[Compact support over \(X\)]
\label{def:compact-support-over-X}
Let \(F^\bullet\in\mathcal D^b(X_{\acute et},\mathbf Z)\)
have torsion cohomology sheaves, and suppose that
\[
        \mathcal G_X^\bullet:=Rs_!F^\bullet
        \in
        \mathcal D^b((\Spec\mathbf Z)_{\acute et},\mathbf Z).
\]
Applying Definition~\ref{def:compact-support-over-specZ},
set
\[
\begin{aligned}
        R\Gamma_{c,\mathrm{fin}}(X,F^\bullet)
        &:=R\Gamma_{\mathrm{fin}}(\mathcal G_X^\bullet),
        \\
        M_\infty(X,F^\bullet)
        &:=M_\infty(\mathcal G_X^\bullet),
        \\
        R\Gamma_\infty(X,F^\bullet)
        &:=R\Gamma_\infty(\mathcal G_X^\bullet),
        \\
        \widehat R\Gamma_\infty(X,F^\bullet)
        &:=\widehat R\Gamma_\infty(\mathcal G_X^\bullet),
        \\
        R\Gamma_c(X,F^\bullet)
        &:=R\Gamma_c(\mathcal G_X^\bullet),
        \\
        \widehat R\Gamma_c(X,F^\bullet)
        &:=\widehat R\Gamma_{\acute et}
                (\Spec\mathbf Z,\mathcal G_X^\bullet).
\end{aligned}
\]
We write
\(H^n_{c,\mathrm{fin}}\), \(H^n_\infty\),
\(\widehat H^n_\infty\), \(H^n_c\), and
\(\widehat H^n_c\) for the cohomology of the
correspondingly named complexes.
The associated boundary morphisms are denoted by
\(\delta^{\mathrm{ord}}_{X,F}\),
\(\delta^{\mathrm{Tate}}_{X,F}\), and
\(\delta^{\mathrm{mod}}_{X,F}\).
\end{definition}

In our notation, Geisser--Schmidt's ordinary and modified
compact support are \(R\Gamma_{c,\mathrm{fin}}\) and
\(\widehat R\Gamma_c\), respectively;
see~\cite[Definition~2.6 and (2.13)]{GeisserSchmidt}.
The intermediate complex \(R\Gamma_c\) denotes ordinary
compact support with respect to the Artin--Verdier
compactification of \(\Spec\mathbf Z\).

Put \(X_{\mathbf C}:=X\times_{\Spec\mathbf Z}\Spec\mathbf C\),
and let \(F^\bullet_{\mathbf C}\) be the pullback of
\(F^\bullet\). Base change for \(Rs_!\) gives a natural
\(G_{\mathbf R}\)-equivariant equivalence
\[
        M_\infty(X,F^\bullet)
        \simeq
        R\Gamma_c(X_{\mathbf C};F^\bullet_{\mathbf C});
\]
see~\cite[Tag~0F7L]{STACKS-PROJECT}.
Here compact support on \(X_{\mathbf C}\) is taken over
\(\Spec\mathbf C\).

\begin{remark}[Boundary fiber sequences]
\label{rem:ordinary-modified-boundary-triangles}
There is a canonical morphism of fiber sequences in \(\mathcal D(\mathbf Z)\)
\[
\begin{tikzcd}[column sep=large,row sep=large]
        R\Gamma_c(X,F^\bullet)
        \arrow[r]
        \arrow[d,"\gamma_{X,F}"']
&
        R\Gamma_{c,\mathrm{fin}}(X,F^\bullet)
        \arrow[r,"\delta^{\mathrm{ord}}_{X,F}"]
        \arrow[d,equal]
&
        R\Gamma_\infty(X,F^\bullet)
        \arrow[d,"\delta^{\mathrm{Tate}}_{X,F}"]
\\
        \widehat R\Gamma_c(X,F^\bullet)
        \arrow[r]
&
        R\Gamma_{c,\mathrm{fin}}(X,F^\bullet)
        \arrow[r,"\delta^{\mathrm{mod}}_{X,F}"]
&
        \widehat R\Gamma_\infty(X,F^\bullet).
\end{tikzcd}
\]
The defining Tate sequence identifies the fiber of
\(\delta^{\mathrm{Tate}}_{X,F}\) with
\(M_\infty(X,F^\bullet)_{hG_{\mathbf R}}\).
Taking fibers of the composite
\(\delta^{\mathrm{mod}}_{X,F}
=\delta^{\mathrm{Tate}}_{X,F}\circ\delta^{\mathrm{ord}}_{X,F}\)
therefore gives a natural fiber sequence
\begin{equation}
\label{eq:ordinary-modified-comparison-fiber}
        R\Gamma_c(X,F^\bullet)
        \xrightarrow{\ \gamma_{X,F}\ }
        \widehat R\Gamma_c(X,F^\bullet)
        \longrightarrow
        M_\infty(X,F^\bullet)_{hG_{\mathbf R}}.
\end{equation}
\end{remark}

\begin{remark}[Local meaning of the modification]
\label{rem:modified-local-vanishing}
Complex-place summands are induced \(G_{\mathbf R}\)-objects,
so their Tate constructions vanish;
see~\cite[Remark~2.4]{GeisserSchmidt}.
Their homotopy orbits need not vanish, however, so they
can contribute to the comparison
\eqref{eq:ordinary-modified-comparison-fiber}.

If \(\widehat R\Gamma_\infty(X,F^\bullet)\simeq0\), the
defining boundary sequence gives
\[
        \widehat R\Gamma_c(X,F^\bullet)
        \simeq
        R\Gamma_{c,\mathrm{fin}}(X,F^\bullet).
\]
The comparison with \(R\Gamma_c(X,F^\bullet)\) is instead
controlled by \(M_\infty(X,F^\bullet)_{hG_{\mathbf R}}\).

Over non-archimedean local fields, modified cohomology
agrees with ordinary cohomology, including compact supports;
see~\cite[(2.9)]{GeisserSchmidt}.
\end{remark}

\section{Stable Steenrod operations on arithmetic boundary cohomology}
\label{sec:stable-steenrod-arithmetic-boundary}

\subsection{Spectral construction of relative Steenrod operations}
\label{subsec:spectral-relative-steenrod-construction}
Let \(\mathcal S:=\mathrm{Ani}\) denote the \(\infty\)-category of anima
(also called spaces), let \(\Pro(\mathcal S)\) denote its pro-category, and let
\(\mathrm{Sp}\) denote the stable \(\infty\)-category of spectra.  For a finite
group \(G\), write
\[
        \mathcal S^{BG}:=\operatorname{Fun}(BG,\mathcal S),
        \qquad
        \mathrm{Sp}^{BG}:=\operatorname{Fun}(BG,\mathrm{Sp}).
\]
All equivariant objects below are naive.  By an equivariant pro-anima we mean
an object of \(\Pro(\mathcal S^{BG})\).  We work intrinsically in
\(\mathcal S\), \(\Pro(\mathcal S)\), and their equivariant variants, using
simplicial sets and pro-simplicial sets only as concrete presentations.

Let \(F(-,-)\) denote the internal function spectrum in \(\mathrm{Sp}\), and let
\(H\Lambda\) be the Eilenberg--Mac Lane spectrum of an abelian group
\(\Lambda\).  We use cohomological grading, so that, for \(K\in\mathcal S\)
and \(E\in\mathrm{Sp}\),
\[
        H^n(K;E)
        :=
        \pi_{-n}F(\Sigma^\infty_+K,E).
\]
All suspension spectra, fibers, and cofibers are taken in the corresponding
\(\infty\)-categories.

\begin{definition}[Relative cochain spectrum]
\label{def:relative-cochain-spectrum}
Let
\[
        f\colon A\longrightarrow X
\]
be a morphism in \(\mathcal S\), and let \(E\in\mathrm{Sp}\).  We define its
\emph{relative \(E\)-cochain spectrum} by
\[
        C(f;E)
        :=
        F\!\left(
                \operatorname{Cof}
                \bigl(\Sigma^\infty_+A\longrightarrow\Sigma^\infty_+X\bigr),
                E
        \right),
\]
and put
\[
        H^n(f;E):=\pi_{-n}C(f;E).
\]
For \(X\in\mathcal S\), we also write
\[
        C(X;E):=F(\Sigma^\infty_+X,E).
\]
\end{definition}

Since function spectra carry cofiber sequences in the first variable to fiber sequences,
there is a natural fiber sequence
\begin{equation}
\label{eq:relative-cochain-fiber-sequence}
        C(f;E)\longrightarrow C(X;E)\longrightarrow C(A;E).
\end{equation}
Thus, \(C(\varnothing\to X;E)\simeq C(X;E)\). If a simplicial
presentation of \(f\) is a monomorphism \(A\hookrightarrow X\), then
\[
        C(f;E)\simeq F(\Sigma^\infty(X/A),E).
\]
For an arbitrary simplicial presentation, the same spectrum is computed after
replacing it by the mapping-cylinder cofibration
\(A\hookrightarrow M(f)\).  Thus, Definition~\ref{def:relative-cochain-spectrum}
recovers the usual cohomology of a pair.

Now let \(f\colon A\to X\) be a morphism in \(\Pro(\mathcal S)\).
Since \([1]\) is a finite poset, the canonical functor
\[
        \Pro\bigl(\mathcal S^{[1]}\bigr)
        \xrightarrow{\;\simeq\;}
        \Pro(\mathcal S)^{[1]}
\]
is an equivalence; see, dually,~\cite[Proposition~5.3.5.15]{HTT}.
Consequently, \(f\) admits a level
presentation over a small cofiltered \(\infty\)-category \(T\),
\[
        f_t\colon A_t\longrightarrow X_t
        \qquad (t\in T),
\]
where \(t\mapsto f_t\) is a diagram in \(\mathcal S^{[1]}\). We set
\begin{equation}
\label{eq:relative-cochain-pro-anima}
        C(f;E)
        :=
        \varinjlim_{t\in T^{\mathrm{op}}}C(f_t;E),
        \qquad
        H^n(f;E):=\pi_{-n}C(f;E).
\end{equation}
The absolute construction \(C(X;E)\) is defined in the same way.

\begin{proposition}[Formal properties of relative cochain spectra]
\label{prop:relative-cochain-spectrum-formalism}
For every \(E\in\mathrm{Sp}\), the relative-cochain functor
\[
        c_E\colon
        \bigl(\mathcal S^{[1]}\bigr)^{\op}
        \longrightarrow
        \mathrm{Sp},
        \qquad
        f\longmapsto C(f;E),
\]
admits, up to a contractible space of choices, a unique
filtered-colimit-preserving extension
\[
        C(-;E)\colon
        \bigl(\Pro(\mathcal S)^{[1]}\bigr)^{\op}
        \longrightarrow
        \mathrm{Sp}.
\]
If \(f\colon A\to X\) is represented over a small cofiltered
\(\infty\)-category \(T\) by maps \(f_t\colon A_t\to X_t\), then
\[
        C(f;E)
        \simeq
        \varinjlim_{t\in T^{\op}}C(f_t;E).
\]
In particular, this expression is independent, up to canonical equivalence,
of the chosen level presentation.

For fixed \(f\), the functor \(C(f;-)\colon\mathrm{Sp}\to\mathrm{Sp}\) is
exact, and there are natural equivalences
\[
        C(f;\Sigma^rE)
        \simeq
        \Sigma^rC(f;E)
        \qquad (r\in\mathbf Z).
\]
Moreover, there is a natural fiber sequence
\[
        C(f;E)\longrightarrow C(X;E)\longrightarrow C(A;E)
\]
and hence a long exact sequence
\[
        \cdots\longrightarrow H^n(f;E)
        \longrightarrow H^n(X;E)
        \longrightarrow H^n(A;E)
        \xrightarrow{\partial}H^{n+1}(f;E)
        \longrightarrow\cdots .
\]
Finally,
\[
        H^n(f;E)
        \cong
        \varinjlim_{t\in T^{\op}}H^n(f_t;E).
\]

For every finite group \(G\), the analogous equivariant construction defines
a functor
\[
        \bigl(\Pro(\mathcal S^{BG})^{[1]}\bigr)^{\op}
        \longrightarrow
        \mathrm{Sp}^{BG}.
\]
\end{proposition}

\begin{proof}
Since \([1]\) is a finite poset, the dual
of~\cite[Proposition~5.3.5.15]{HTT} gives a canonical equivalence
\[
        \Pro\bigl(\mathcal S^{[1]}\bigr)
        \simeq
        \Pro(\mathcal S)^{[1]}.
\]
Taking opposites gives
\begin{equation}
\label{eq:opposite-arrow-pro-anima-as-ind}
        \bigl(\Pro(\mathcal S)^{[1]}\bigr)^{\op}
        \simeq
        \Ind\bigl((\mathcal S^{[1]})^{\op}\bigr).
\end{equation}
The universal property of the
Ind-completion~\cite[Proposition~5.3.5.10]{HTT} therefore extends \(c_E\), essentially
uniquely, to a filtered-colimit-preserving functor
\[
        C(-;E)\colon
        \bigl(\Pro(\mathcal S)^{[1]}\bigr)^{\op}
        \longrightarrow
        \mathrm{Sp}.
\]
Under~\eqref{eq:opposite-arrow-pro-anima-as-ind}, a level presentation
\((f_t)_{t\in T}\) of \(f\) exhibits \(f\) as the filtered colimit of the
diagram
\[
        T^{\op}\longrightarrow
        \bigl(\mathcal S^{[1]}\bigr)^{\op},
        \qquad t\longmapsto f_t.
\]
Consequently,
\[
        C(f;E)
        \simeq
        \varinjlim_{t\in T^{\op}}C(f_t;E).
\]
This proves independence of the level presentation and contravariant
functoriality in \(f\). The same argument with \(\mathcal S^{BG}\) and
\(\mathrm{Sp}^{BG}\) proves the equivariant assertion.

Fix a level presentation \((f_t)_{t\in T}\), and put
\[
        P_t:=
        \operatorname{Cof}\bigl(
        \Sigma^\infty_+A_t\longrightarrow\Sigma^\infty_+X_t
        \bigr).
\]
The functor \(F(P_t,-)\) is exact and satisfies
\[
        F(P_t,\Sigma^rE)
        \simeq
        \Sigma^rF(P_t,E).
\]
Since filtered colimits in \(\mathrm{Sp}\) are exact and commute with
suspension, the same assertions hold for \(C(f;-)\).

For every \(t\), the defining cofiber sequence for \(P_t\) induces a fiber
sequence
\[
        C(f_t;E)
        \longrightarrow
        C(X_t;E)
        \longrightarrow
        C(A_t;E).
\]
Taking the filtered colimit over \(T^{\op}\) gives
\[
        C(f;E)
        \longrightarrow
        C(X;E)
        \longrightarrow
        C(A;E),
\]
and applying homotopy groups gives the asserted long exact sequence.

Finally, the sphere spectra are compact, so stable homotopy groups commute
with filtered colimits. Hence,
\[
\begin{aligned}
        H^n(f;E)
        &=
        \pi_{-n}C(f;E)\\
        &\cong
        \varinjlim_{t\in T^{\op}}\pi_{-n}C(f_t;E)\\
        &=
        \varinjlim_{t\in T^{\op}}H^n(f_t;E).
\end{aligned}
\]
\end{proof}

For spectra \(P,Q\), write
\[
        [P,Q]_{\mathrm{Sp}}
        :=
        \pi_0F(P,Q)
\]
for the abelian group of morphisms \(P\to Q\) in the homotopy category of
spectra. We now specialize to \(H\FF_2\)-coefficients.
The graded endomorphism algebra
\[
        \mathcal A_2^r
        :=
        [H\FF_2,\Sigma^rH\FF_2]_{\mathrm{Sp}},
        \qquad
        \mathcal A_2
        :=
        \bigoplus_{r\geq0}\mathcal A_2^r,
\]
is the ordinary mod-\(2\) stable Steenrod algebra;
see~\cite[Part~III, \S12]{AdamsStableHomotopy}. Representability identifies
\(\mathcal A_2^r\) with the group of stable degree-\(r\) natural operations
in reduced mod-\(2\) cohomology;
see~\cite[Remark~3.1.4]{BeaudryCampbell}. Its multiplication is given by
suspended composition: for \(\theta\in\mathcal A_2^a\) and
\(\phi\in\mathcal A_2^b\),
\[
        \theta\phi
        :=
        (\Sigma^b\theta)\circ\phi.
\]
Under this correspondence, let
\[
        \operatorname{sq}^r\in\mathcal A_2^r
        \qquad(r\geq0)
\]
denote the class corresponding to the classical \(r\)-th Steenrod square.
Equivalently, if \(K\in\mathcal S_*\) and
\(x\in\widetilde H^n(K;\FF_2)\) is represented by
\[
        \bar x\colon\Sigma^\infty K\longrightarrow\Sigma^nH\FF_2,
\]
then \(\Sq^r(x)\) is represented by
\[
        (\Sigma^n\operatorname{sq}^r)\circ\bar x.
\]
Functoriality of \(C(f;-)\) in the coefficient spectrum consequently gives
a homotopy class of maps
\[
        C(f;\operatorname{sq}^r)\colon
        C(f;H\FF_2)
        \longrightarrow
        C(f;\Sigma^rH\FF_2)
        \simeq
        \Sigma^rC(f;H\FF_2).
\]
With our grading convention,
\[
        \pi_{-n}\!\left(\Sigma^rC(f;H\FF_2)\right)
        =
        \pi_{-n-r}C(f;H\FF_2)
        =
        H^{n+r}(f;H\FF_2).
\]

\begin{definition}[Relative Steenrod operations]
\label{def:spectral-relative-steenrod-operations}
Let \(f\colon A\to X\) be a morphism in \(\Pro(\mathcal S)\). For
\(n\in\mathbf Z\) and \(r\geq0\), define
\[
        \Sq_f^r\colon
        H^n(f;H\FF_2)
        \longrightarrow
        H^{n+r}(f;H\FF_2)
\]
to be the homomorphism induced on \(\pi_{-n}\) by
\(C(f;\operatorname{sq}^r)\).
\end{definition}

\begin{remark}
\label{rem:steenrod-operations-underlying-spectra}
For \(r>0\), the class \(\operatorname{sq}^r\) is not represented by an
\(H\FF_2\)-linear map, since
\[
        [H\FF_2,\Sigma^rH\FF_2]_{\Mod_{H\FF_2}}=0.
\]
Thus, the maps defining the relative Steenrod operations are formed in
\(\mathrm{Sp}\), even though \(C(f;H\FF_2)\) is naturally an
\(H\FF_2\)-module.
\end{remark}

\begin{theorem}[Relative stable Steenrod operations]
\label{thm:relative-stable-steenrod-operations}
The operations of Definition~\ref{def:spectral-relative-steenrod-operations} endow
\(H^*(f;H\FF_2)\) with a natural graded left \(\mathcal A_2\)-module structure. They commute
with the connecting morphisms in the relative long exact sequence, satisfy
\(\Sq_f^0=\mathrm{id}\), and satisfy the ordinary
Adem relations
\[
        \Sq_f^a\Sq_f^b
        =
        \sum_{j=0}^{\lfloor a/2\rfloor}
        \binom{b-j-1}{a-2j}
        \Sq_f^{a+b-j}\Sq_f^j
        \qquad (0<a<2b).
\]
Here the binomial coefficients are read in \(\FF_2\).
Moreover, \(\Sq_f^1\) is the Bockstein associated with
\[
        0\longrightarrow\FF_2
        \longrightarrow\mathbf Z/4
        \longrightarrow\FF_2
        \longrightarrow0.
\]
\end{theorem}

\begin{proof}
Additivity and naturality follow from the naturality of postcomposition. Postcomposition
with \(\operatorname{sq}^r\) gives a morphism from the fiber sequence of
Proposition~\ref{prop:relative-cochain-spectrum-formalism} to its \(r\)-fold suspension, and
hence the operations commute with the connecting morphisms. The identities
\(\Sq^0=\mathrm{id}\) and the Adem relations are the corresponding identities in the stable
Steenrod algebra.  Likewise, \(\operatorname{sq}^1\) is the connecting map in the cofiber
sequence
\[
        H\FF_2\longrightarrow H\mathbf Z/4
        \longrightarrow H\FF_2
        \xrightarrow{\operatorname{sq}^1}\Sigma H\FF_2,
\]
which proves the Bockstein assertion.
\end{proof}
\subsection{Arithmetic boundary spectra and comparison}
\label{subsec:arithmetic-boundary-spectra-comparison}
We now apply the construction of
\S\ref{subsec:spectral-relative-steenrod-construction} to the arithmetic
boundary. Its functoriality in the boundary arrow and in the coefficient
spectrum allows stable cohomology operations to act on the resulting boundary
diagram.

Let
\(\mathrm{Sch}^{\mathrm{ft,sep}}_{\mathbf Z,\mathrm{prop}}\)
denote the category of schemes separated and of finite type over
\(\Spec(\mathbf Z)\), with proper \(\mathbf Z\)-morphisms.  Let
\(\operatorname{Comp}^{\partial}_{\mathbf Z}\) denote the category whose
objects are dense compactifications over \(\Spec(\mathbf Z)\),
\[
        j_X\colon X\hookrightarrow\bar X,
\]
where \(X\) is separated and of finite type over \(\Spec(\mathbf Z)\) and
\(\bar X\) is proper over \(\Spec(\mathbf Z)\).  A morphism
\(j_Y\to j_X\) is a Cartesian square of \(\mathbf Z\)-schemes
\[
\begin{tikzcd}
        Y\arrow[r,"f"]\arrow[d,hook,"j_Y"']
&
        X\arrow[d,hook,"j_X"]
\\
        \bar Y\arrow[r,"\bar f"']
&
        \bar X
\end{tikzcd}
\]
with \(\bar f\) proper.  In particular, \(f\) is proper. This is the subcategory of the category of compactifications described
in~\cite[Tag~0ATT]{STACKS-PROJECT} consisting of dense compactifications and
morphisms for which the displayed square of open immersions is Cartesian.

Forgetting the compactification defines a functor
\[
        q\colon
        \operatorname{Comp}^{\partial}_{\mathbf Z}
        \longrightarrow
        \mathrm{Sch}^{\mathrm{ft,sep}}_{\mathbf Z,\mathrm{prop}}.
\]
For fixed \(X\), write \(\operatorname{Comp}_{\mathbf Z}(X)\) for the fiber
of \(q\) over \(X\).

Fix an object
\[
        j\colon X\hookrightarrow\bar X
\]
of \(\operatorname{Comp}^{\partial}_{\mathbf Z}\), write
\[
        s\colon X\longrightarrow\Spec(\mathbf Z)
\]
for its structural morphism, and put
\[
        i\colon
        \partial_jX:=\bar X\setminus X
        \hookrightarrow
        \bar X.
\]

For a scheme \(Y\), let \(Y_{\acute et}\) denote its small \'{e}tale site,
let \(\operatorname{Shv}(Y_{\acute et})\) denote the \(\infty\)-topos of
anima-valued sheaves on \(Y_{\acute et}\), and let \(\Pi_\infty\) denote
the shape functor from \(\infty\)-topoi to pro-anima.
Following~\cite[Corollary~2.12]{CarchediElmanto}, write
\[
        \Pi_\infty^{\acute et}(Y)
        :=
        \Pi_\infty\!\left(
                \operatorname{Shv}(Y_{\acute et})
        \right)
        \in
        \Pro(\mathcal S)
\]
for the intrinsic \'{e}tale homotopy type of \(Y\), and abbreviate
\[
        h(Y):=\Pi_\infty^{\acute et}(Y).
\]
A morphism \(u\colon Y\to Z\) induces a morphism
\[
        h(u)\colon h(Y)\longrightarrow h(Z).
\]

Write \(i_{\mathbf R}\) and \(i_{\mathbf C}\) for the base changes of \(i\)
to \(\Spec(\mathbf R)\) and \(\Spec(\mathbf C)\), respectively, and put
\[
        G_{\mathbf R}:=\operatorname{Gal}(\mathbf C/\mathbf R).
\]

For an \(\mathbf R\)-scheme \(Y\), let
\[
        p_Y\colon
        \operatorname{Shv}(Y_{\acute et})
        \longrightarrow
        \operatorname{Shv}((\Spec(\mathbf R))_{\acute et})
\]
be the structural geometric morphism, and set
\[
        h_{\mathbf R}(Y)
        :=
        (p_Y)_\sharp(*)
        \in
        \Pro\!\left(
                \operatorname{Shv}(
                        (\Spec(\mathbf R))_{\acute et}
                )
        \right)
        \simeq
        \Pro(\mathcal S^{BG_{\mathbf R}}).
\]
Here \((p_Y)_\sharp\) is left adjoint to \(\Pro(p_Y^*)\);
see~\cite[Proposition~2.0.2 and Definition~2.0.3]{ACS}
and~\cite[Definition~4.6 and Proposition~4.9]{CarchediElmanto}.
We use \(h_{\mathbf R}\) also for the arrows induced by morphisms of
\(\mathbf R\)-schemes.

Let
\[
        q_{\mathbf R}\colon\Spec(\mathbf C)\longrightarrow\Spec(\mathbf R),
        \qquad
        q_Y\colon Y_{\mathbf C}\longrightarrow Y.
\]
Under the displayed equivalence, \(\Pro(q_{\mathbf R}^*)\) identifies with
the underlying-object functor
\[
        U\colon
        \Pro(\mathcal S^{BG_{\mathbf R}})
        \longrightarrow
        \Pro(\mathcal S).
\]

The commutative square of inverse-image functors
\[
        q_Y^*p_Y^*
        \simeq
        p_{Y_{\mathbf C}}^*q_{\mathbf R}^*
\]
has right Beck--Chevalley map
\[
        p_Y^*(q_{\mathbf R})_*
        \longrightarrow
        (q_Y)_*p_{Y_{\mathbf C}}^*;
\]
see~\cite[Definition~2.1.3]{ACS}. This map is an equivalence.  Indeed, the
geometric morphism induced by \(q_{\mathbf R}\) identifies with the slice projection
associated with the finite locally constant object represented by
\(\Spec(\mathbf C)\), and \(q_Y\) is its pullback along \(p_Y\).
The corresponding direct images are dependent products along these objects.
Locally for the \'{e}tale topology, the object represented by
\(\Spec(\mathbf C)\) is a coproduct of two terminal objects, so the
Beck--Chevalley map reduces to the preservation of a binary product by
\(p_Y^*\).  It is therefore an equivalence because \(p_Y^*\) is left exact.

After passage to pro-categories, the two sides of this equivalence are the
right adjoints of
\[
        \Pro(q_{\mathbf R}^*)(p_Y)_\sharp
        \qquad\text{and}\qquad
        (p_{Y_{\mathbf C}})_\sharp\Pro(q_Y^*),
\]
respectively.  Hence, uniqueness of left adjoints gives a canonical
equivalence
\[
        (p_{Y_{\mathbf C}})_\sharp\Pro(q_Y^*)
        \simeq
        \Pro(q_{\mathbf R}^*)(p_Y)_\sharp.
\]
Evaluating at the terminal object and using
\(q_Y^*(*)\simeq *\) yields a natural equivalence
\[
        h(Y_{\mathbf C})
        \xrightarrow{\ \sim\ }
        U\bigl(h_{\mathbf R}(Y)\bigr).
\]
Compare~\cite[Proposition~2.12]{HarpazSchlank2013} for the corresponding
comparison in the pro-homotopy-category formulation.

Set
\[
        \widetilde h(i_{\mathbf C})
        :=
        h_{\mathbf R}(i_{\mathbf R})
        \in
        \Pro(\mathcal S^{BG_{\mathbf R}})^{[1]}.
\]
Thus,
\[
        U\bigl(\widetilde h(i_{\mathbf C})\bigr)
        \simeq
        h(i_{\mathbf C}).
\]

Let
\[
        (-)/\!/G_{\mathbf R}\colon
        \Pro(\mathcal S^{BG_{\mathbf R}})
        \rightleftarrows
        \Pro(\mathcal S)
        \mathbin{:}\operatorname{triv}
\]
denote the homotopy-quotient--trivial-action adjunction.  The
relative-to-absolute comparison gives, naturally in every
\(\mathbf R\)-scheme \(Y\), an equivalence
\[
        h_{\mathbf R}(Y)/\!/G_{\mathbf R}
        \simeq
        h(Y);
\]
see~\cite[Remark~4.7 and the discussion following
Proposition~4.11]{CarchediElmanto}.  Applied to \(i_{\mathbf R}\), this
gives an equivalence of arrows
\[
        \widetilde h(i_{\mathbf C})/\!/G_{\mathbf R}
        \simeq
        h(i_{\mathbf R}).
\]

Composing with the morphism \(h(i_{\mathbf R})\to h(i)\) induced by base
change and adjunction gives
\begin{equation}
\label{eq:equivariant-boundary-base-change-arrow}
        \widetilde\rho_j\colon
        \widetilde h(i_{\mathbf C})
        \longrightarrow
        \operatorname{triv}\bigl(h(i)\bigr)
\end{equation}
in \(\Pro(\mathcal S^{BG_{\mathbf R}})^{[1]}\).  Its underlying morphism is
the composite
\[
        U\bigl(\widetilde h(i_{\mathbf C})\bigr)
        \longrightarrow
        \widetilde h(i_{\mathbf C})/\!/G_{\mathbf R}
        \xrightarrow{\ \sim\ }
        h(i_{\mathbf R})
        \longrightarrow
        h(i),
\]
where the first arrow is obtained by applying \(U\) to the adjunction
unit.  Under the preceding comparison equivalences, this identifies with
\[
        h(i_{\mathbf C})
        \longrightarrow
        h(i_{\mathbf R})
        \longrightarrow
        h(i),
\]
and hence with the morphism induced by
\[
\begin{tikzcd}
        (\partial_jX)_{\mathbf C}\arrow[r]\arrow[d,"i_{\mathbf C}"']
&
        \partial_jX\arrow[d,"i"]
\\
        \bar X_{\mathbf C}\arrow[r]
&
        \bar X.
\end{tikzcd}
\]

Coefficient spectra are endowed with the trivial
\(G_{\mathbf R}\)-action. By the equivariant construction of
Proposition~\ref{prop:relative-cochain-spectrum-formalism}, there is a
functor
\[
        C(-;E)\colon
        \bigl(\Pro(\mathcal S^{BG_{\mathbf R}})^{[1]}\bigr)^{\mathrm{op}}
        \longrightarrow
        \mathrm{Sp}^{BG_{\mathbf R}}.
\]
The underlying spectrum of
\(C(\widetilde h(i_{\mathbf C});E)\) is canonically equivalent to
\(C(h(i_{\mathbf C});E)\). Using
\[
        C\bigl(\operatorname{triv}(h(i));E\bigr)
        \simeq
        \operatorname{triv}\bigl(C(h(i);E)\bigr),
\]
contravariance applied
to~\eqref{eq:equivariant-boundary-base-change-arrow} gives
\[
        \rho_{X,j,E}\colon
        \operatorname{triv}\bigl(C(h(i);E)\bigr)
        \longrightarrow
        C\bigl(\widetilde h(i_{\mathbf C});E\bigr)
\]
in \(\mathrm{Sp}^{BG_{\mathbf R}}\).

\begin{definition}[Boundary data]
\label{def:boundary-data-category}
Let
\[
        \mathsf{Bnd}_{G_{\mathbf R}}
        :=
        \mathrm{Sp}
        \times_{\mathrm{Sp}^{BG_{\mathbf R}}}
        \operatorname{Fun}\!\left([1],\mathrm{Sp}^{BG_{\mathbf R}}\right),
\]
where the functors to \(\mathrm{Sp}^{BG_{\mathbf R}}\) are
\(\operatorname{triv}\) and evaluation at the source.  Thus, an object is a
triple
\[
        \bigl(A,M,\rho\colon\operatorname{triv}(A)\to M\bigr).
\]
We equip this category with the termwise symmetric monoidal structure
\[
\begin{aligned}
 &(A,M,\rho)\wedge(A',M',\rho')
\\
 &\qquad:=
 \bigl(
 A\wedge A',
 M\wedge M',
 \operatorname{triv}(A\wedge A')
 \simeq
 \operatorname{triv}(A)\wedge\operatorname{triv}(A')
 \xrightarrow{\rho\wedge\rho'}
 M\wedge M'
 \bigr).
\end{aligned}
\]
Since \(\operatorname{triv}\) and evaluation at the source are exact
symmetric monoidal functors, \(\mathsf{Bnd}_{G_{\mathbf R}}\) is a stable
symmetric monoidal \(\infty\)-category.
\end{definition}

Homotopy fixed points, the Tate construction, and fibers functorially
associate a boundary diagram to every object of
\(\mathsf{Bnd}_{G_{\mathbf R}}\), as follows.
\begin{definition}[Spectral arithmetic boundary diagram]
\label{def:spectral-arithmetic-boundary-diagram}
For \(E\in\mathrm{Sp}\), set
\[
\begin{aligned}
        C_{c,\mathrm{fin},j}(X;E)
        &:= C(h(i);E),
&
        C_{\mathbf C,j}(X;E)
        &:= C\bigl(\widetilde h(i_{\mathbf C});E\bigr),
\\
        C_{R,j}(X;E)
        &:= C_{\mathbf C,j}(X;E)^{hG_{\mathbf R}},
&
        C_{\infty,j}(X;E)
        &:= C_{\mathbf C,j}(X;E)^{tG_{\mathbf R}}.
\end{aligned}
\]
Under the adjunction
\[
        \operatorname{triv}\colon\mathrm{Sp}
        \rightleftarrows
        \mathrm{Sp}^{BG_{\mathbf R}}
        \mathbin{:}(-)^{hG_{\mathbf R}},
\]
let
\[
        \delta^{\mathrm{ord}}_{X,j,E}\colon
        C_{c,\mathrm{fin},j}(X;E)
        \longrightarrow
        C_{R,j}(X;E)
\]
be the adjunct of \(\rho_{X,j,E}\). Let
\[
        \delta^{\mathrm{Tate}}_{X,j,E}\colon
        C_{R,j}(X;E)
        \longrightarrow
        C_{\infty,j}(X;E)
\]
be the canonical map from homotopy fixed points to the Tate construction,
and set
\[
        \delta^{\mathrm{mod}}_{X,j,E}
        :=
        \delta^{\mathrm{Tate}}_{X,j,E}
        \circ
        \delta^{\mathrm{ord}}_{X,j,E}.
\]
Define
\[
        C_{c,j}(X;E)
        :=
        \Fib\bigl(\delta^{\mathrm{ord}}_{X,j,E}\bigr),
        \qquad
        \widehat C_{c,j}(X;E)
        :=
        \Fib\bigl(\delta^{\mathrm{mod}}_{X,j,E}\bigr).
\]
The identity of \(C_{c,\mathrm{fin},j}(X;E)\), together with
\(\delta^{\mathrm{Tate}}_{X,j,E}\), induces a canonical morphism of fiber
sequences
\begin{equation}
\label{eq:spectral-arithmetic-boundary-diagram}
\begin{tikzcd}[column sep=large,row sep=large]
        C_{c,j}(X;E)
        \arrow[r]
        \arrow[d,"\gamma_{X,j,E}"']
&
        C_{c,\mathrm{fin},j}(X;E)
        \arrow[r,"\delta^{\mathrm{ord}}_{X,j,E}"]
        \arrow[d,equal]
&
        C_{R,j}(X;E)
        \arrow[d,"\delta^{\mathrm{Tate}}_{X,j,E}"]
\\
        \widehat C_{c,j}(X;E)
        \arrow[r]
&
        C_{c,\mathrm{fin},j}(X;E)
        \arrow[r,"\delta^{\mathrm{mod}}_{X,j,E}"']
&
        C_{\infty,j}(X;E).
\end{tikzcd}
\end{equation}
We write
\[
        \mathscr B_{X,j}(E)
        :=
        \bigl(
                C_{c,\mathrm{fin},j}(X;E),
                C_{\mathbf C,j}(X;E),
                \rho_{X,j,E}
        \bigr)
        \in
        \mathsf{Bnd}_{G_{\mathbf R}}.
\]
Diagram~\eqref{eq:spectral-arithmetic-boundary-diagram} is the arithmetic
boundary diagram functorially associated with this boundary datum.
\end{definition}

\begin{proposition}[Functoriality of the spectral boundary datum]
\label{prop:arithmetic-boundary-coefficient-functoriality}
The construction above defines a bifunctor
\[
        \mathscr B\colon
        \bigl(\operatorname{Comp}^{\partial}_{\mathbf{Z}}\bigr)^{\mathrm{op}}
        \times\mathrm{Sp}
        \longrightarrow
        \mathsf{Bnd}_{G_{\mathbf R}},
        \qquad
        (j_X,E)\longmapsto\mathscr B_{X,j_X}(E).
\]
For every \(j_X\), the functor \(E\mapsto\mathscr B_{X,j_X}(E)\) is exact.
Consequently, there are natural equivalences
\[
        \mathscr B_{X,j_X}(\Sigma^rE)
        \simeq
        \Sigma^r\mathscr B_{X,j_X}(E)
        \qquad(r\in\mathbf Z).
\]
For fixed \(E\), write
\[
        \mathscr B_E:=\mathscr B(-,E).
\]
\end{proposition}

\begin{proof}
Functoriality of \(h\) and \(h_{\mathbf R}\), together with the naturality
of the relative-to-absolute comparison and of the adjunction defining
\(\widetilde\rho_j\), makes \(\widetilde\rho_j\), and hence
\(\rho_{X,j,E}\), functorial in compactified morphisms. Contravariance of relative cochains in the
boundary arrow and functoriality in \(E\) therefore give the asserted
bifunctor.

Proposition~\ref{prop:relative-cochain-spectrum-formalism} and its
equivariant analog show that its finite and geometric terms are exact in
\(E\).  The trivial-action, homotopy-fixed-point, Tate, and fiber functors
are exact.  Hence, every associated term is exact in \(E\), and exact functors
between stable \(\infty\)-categories commute with suspension.
\end{proof}

We regard the complexes of Definition~\ref{def:compact-support-over-X} as
spectra through the Eilenberg--Mac Lane equivalence
\[
        \mathcal D(\mathbf Z)\simeq\Mod_{H\mathbf Z},
\]
and equivariant complexes through its objectwise extension
\[
        \operatorname{Fun}\!\left(
                BG_{\mathbf R},\mathcal D(\mathbf Z)
        \right)
        \simeq
        \operatorname{Fun}\!\left(
                BG_{\mathbf R},\Mod_{H\mathbf Z}
        \right).
\]
We retain the same notation for the resulting spectra;
see~\cite[Proposition~2.10]{ShipleyHZAlgebraSpectra}.

For a finite abelian group \(\Lambda\), let
\[
        \rho^{\mathrm{arith}}_{X,\Lambda}\colon
        \operatorname{triv}\bigl(
                R\Gamma_{c,\mathrm{fin}}(X,\Lambda)
        \bigr)
        \longrightarrow
        R\Gamma_c(X_{\mathbf C},\Lambda)
\]
be equivariant restriction to the geometric fiber at infinity; its adjunct
is \(\delta^{\mathrm{ord}}_{X,\Lambda}\).  Set
\[
        \mathscr D_{\Lambda}(X)
        :=
        \bigl(
                R\Gamma_{c,\mathrm{fin}}(X,\Lambda),
                R\Gamma_c(X_{\mathbf C},\Lambda),
                \rho^{\mathrm{arith}}_{X,\Lambda}
        \bigr)
        \in
        \mathsf{Bnd}_{G_{\mathbf R}}.
\]
Its associated boundary diagram is the intrinsic arithmetic boundary diagram
of Definition~\ref{def:compact-support-over-X}.  Proper pullback defines a
functor
\[
        \mathscr D_\Lambda\colon
        \bigl(
                \mathrm{Sch}^{\mathrm{ft,sep}}_{\mathbf{Z},\mathrm{prop}}
        \bigr)^{\mathrm{op}}
        \longrightarrow
        \mathsf{Bnd}_{G_{\mathbf R}},
        \qquad
        X\longmapsto\mathscr D_{\Lambda}(X),
\]
naturally in \(\Lambda\).

\begin{theorem}[Functorial arithmetic boundary comparison]
\label{thm:spectral-arithmetic-boundary-comparison}
For every finite abelian group \(\Lambda\), there is a canonical natural
equivalence
\[
        \eta_\Lambda\colon
        \mathscr B_{H\Lambda}
        \xrightarrow{\ \sim\ }
        \mathscr D_\Lambda\circ q^{\mathrm{op}}
\]
in
\[
        \operatorname{Fun}\!\left(
                \bigl(\operatorname{Comp}^{\partial}_{\mathbf{Z}}\bigr)^{\mathrm{op}},
                \mathsf{Bnd}_{G_{\mathbf R}}
        \right),
\]
natural in \(\Lambda\).  Write \(\eta_{\Lambda,X,j}\) for its component at
\(j\colon X\hookrightarrow\bar X\).  Its two components are compatible
equivalences
\[
        C_{c,\mathrm{fin},j}(X;H\Lambda)
        \simeq
        R\Gamma_{c,\mathrm{fin}}(X,\Lambda)
\]
and
\[
        C_{\mathbf C,j}(X;H\Lambda)
        \simeq
        R\Gamma_c(X_{\mathbf C},\Lambda)
\]
in \(\mathrm{Sp}\) and \(\mathrm{Sp}^{BG_{\mathbf R}}\), respectively, which
identify \(\rho_{X,j,H\Lambda}\) with
\(\rho^{\mathrm{arith}}_{X,\Lambda}\).

Consequently, there are natural equivalences
\[
\begin{aligned}
        C_{R,j}(X;H\Lambda)
        &\simeq R\Gamma_\infty(X,\Lambda),
&
        C_{\infty,j}(X;H\Lambda)
        &\simeq\widehat R\Gamma_\infty(X,\Lambda),
\\
        C_{c,j}(X;H\Lambda)
        &\simeq R\Gamma_c(X,\Lambda),
&
        \widehat C_{c,j}(X;H\Lambda)
        &\simeq\widehat R\Gamma_c(X,\Lambda),
\end{aligned}
\]
compatible with all structure morphisms, and
\(\gamma_{X,j,H\Lambda}\) corresponds to the canonical map from ordinary to
modified compact support.

For fixed \(X\), restriction to \(\operatorname{Comp}_{\mathbf{Z}}(X)\) gives a natural
equivalence
\[
        \left.
        \mathscr B_{H\Lambda}
        \right|_{\operatorname{Comp}_{\mathbf{Z}}(X)^{\mathrm{op}}}
        \simeq
        \operatorname{const}_{\mathscr D_{\Lambda}(X)}.
\]
\end{theorem}

\begin{proof}
For \(j\colon X\hookrightarrow\bar X\), let
\(\bar s\colon\bar X\to \mathbf{Z}\) be the structural morphism.  Since
\(s=\bar s\circ j\) and \(\bar s\) is proper,
\[
        Rs_!\Lambda\simeq R\bar s_*j_!\Lambda.
\]
The relative open--closed comparison of
Proposition~\ref{prop:relative-comparison}, together with proper base change,
gives natural equivalences
\[
        C_{c,\mathrm{fin},j}(X;H\Lambda)
        \simeq
        R\Gamma_{c,\mathrm{fin}}(X,\Lambda)
\]
and
\[
        C_{\mathbf C,j}(X;H\Lambda)
        \simeq
        R\Gamma_c(X_{\mathbf C},\Lambda),
\]
the second in \(\mathrm{Sp}^{BG_{\mathbf R}}\).  Naturality of the relative
comparison for~\eqref{eq:equivariant-boundary-base-change-arrow} identifies
\(\rho_{X,j,H\Lambda}\) with
\(\rho^{\mathrm{arith}}_{X,\Lambda}\).  These equivalences are natural in
compactified morphisms and in \(\Lambda\), and therefore assemble into
\(\eta_\Lambda\).

Applying homotopy fixed points, the Tate construction, and fibers gives the
remaining equivalences.  On \(\operatorname{Comp}_{\mathbf{Z}}(X)\), the functor \(q\)
is constant at \(X\), which proves the final assertion.
\end{proof}

For the multiplicative constructions below, let
\(\mathcal E_2\subset\mathrm{Sp}\) be the full subcategory spanned by finite
direct sums of suspensions of \(H\FF_2\).  Morphisms in \(\mathcal E_2\)
are maps of underlying spectra.

For a scheme \(Y\), write
\[
        \operatorname{Shv}(Y_{\acute et};\mathrm{Sp})
\]
for the stable \(\infty\)-category of \'etale sheaves of spectra.
The coefficient forgetful functor
\(\mathcal D(R)\simeq\operatorname{Mod}_{HR}\to\mathrm{Sp}\)
induces
\[
        \mathcal D(Y_{\acute et},R)
        \longrightarrow
        \operatorname{Shv}(Y_{\acute et};\mathrm{Sp}),
\]
compatibly with inverse image, global sections, proper pushforward,
and extension by zero.

For \(E\in\mathcal E_2\), let \(\underline E_Y\) denote the
constant sheaf of spectra. Using extension by zero and proper
pushforward, define
\[
        R\Gamma_{c,\mathrm{fin}}(X;\underline E_X)
        \qquad\text{and}\qquad
        R\Gamma_c(X_{\mathbf C};\underline E_{X_{\mathbf C}})
\]
by the formulas of Definition~\ref{def:compact-support-over-X},
interpreted spectrally. For \(E=H\FF_2\), these recover the
underlying spectra of the complexes defined there, compatibly
with the boundary morphisms and products.

Let
\[
        \rho^{\mathrm{arith,sp}}_{X,E}\colon
        \operatorname{triv}\bigl(
                R\Gamma_{c,\mathrm{fin}}(X;\underline E_X)
        \bigr)
        \longrightarrow
        R\Gamma_c(X_{\mathbf C};\underline E_{X_{\mathbf C}})
\]
be equivariant restriction to the geometric fiber at infinity, and set
\[
        \mathscr D_E^{\mathrm{sp}}(X)
        :=
        \bigl(
                R\Gamma_{c,\mathrm{fin}}(X;\underline E_X),
                R\Gamma_c(X_{\mathbf C};\underline E_{X_{\mathbf C}}),
                \rho^{\mathrm{arith,sp}}_{X,E}
        \bigr)
        \in
        \mathsf{Bnd}_{G_{\mathbf R}}.
\]
Proper pullback makes this construction functorial in \(X\).

\begin{proposition}[Spectrum-valued multiplicative refinement]
\label{prop:multiplicative-spectrum-valued-comparison}
For every \(E\in\mathcal E_2\), there is a canonical natural equivalence
\[
        \eta_E^{\mathrm{sp}}\colon
        \mathscr B_E
        \xrightarrow{\ \sim\ }
        \mathscr D_E^{\mathrm{sp}}\circ q^{\op}
\]
in
\[
        \operatorname{Fun}\!\left(
                \bigl(\operatorname{Comp}^{\partial}_{\mathbf Z}\bigr)^{\op},
                \mathsf{Bnd}_{G_{\mathbf R}}
        \right).
\]
It is natural in arbitrary maps of the coefficient spectra.  Its components
are compatible equivalences
\begin{equation}
\label{eq:spectrum-valued-multiplicative-boundary-comparison}
\begin{aligned}
        C_{c,\mathrm{fin},j}(X;E)
        &\simeq
        R\Gamma_{c,\mathrm{fin}}(X;\underline E_X),
\\
        C_{\mathbf C,j}(X;E)
        &\simeq
        R\Gamma_c(X_{\mathbf C};\underline E_{X_{\mathbf C}}),
\end{aligned}
\end{equation}
the second in \(\mathrm{Sp}^{BG_{\mathbf R}}\).  These equivalences identify
\(\rho_{X,j,E}\) with
\(\rho^{\mathrm{arith,sp}}_{X,E}\) and are compatible with external
coefficient pairings and cup products.  For \(E=H\FF_2\), they recover
\(\eta_{\FF_2}\) of
Theorem~\ref{thm:spectral-arithmetic-boundary-comparison}.
\end{proposition}

\begin{proof}
For every \(\infty\)-topos \(\mathcal T\), the defining adjunction for its
shape gives, for every \(m\in\mathbf Z\),
\[
\begin{aligned}
        \Omega^\infty\Sigma^m
        C(\Pi_\infty\mathcal T;E)
        &\simeq
        \operatorname{Map}_{\Pro(\mathcal S)}
        \bigl(
                \Pi_\infty\mathcal T,
                \Omega^\infty\Sigma^mE
        \bigr)
\\
        &\simeq
        \Gamma\!\left(
                \mathcal T,
                \underline{\Omega^\infty\Sigma^mE}_{\mathcal T}
        \right)
\\
        &\simeq
        \Omega^\infty\Sigma^m
        R\Gamma(\mathcal T,\underline E_{\mathcal T}).
\end{aligned}
\]
These equivalences are compatible with the spectrum structures and hence
assemble into a natural equivalence
\begin{equation}
\label{eq:stable-etale-shape-comparison}
        C(\Pi_\infty\mathcal T;E)
        \simeq
        R\Gamma(\mathcal T,\underline E_{\mathcal T}).
\end{equation}
The same argument internal to
\(\operatorname{Shv}((\Spec(\mathbf R))_{\acute et})\) gives the
corresponding relative-shape equivalence in
\(\mathrm{Sp}^{BG_{\mathbf R}}\).

Apply these comparisons to the open--closed decompositions determined by
\(j\) and \(j_{\mathbf C}\).  The relative-cochain fiber sequences correspond
under~\eqref{eq:stable-etale-shape-comparison} to the
recollement fiber sequences, and proper base change identifies the resulting
fibers with the two compact-support spectra
in~\eqref{eq:spectrum-valued-multiplicative-boundary-comparison}.  Naturality
for the base-change square defining \(\widetilde\rho_j\) identifies
\(\rho_{X,j,E}\) with equivariant restriction.

Finally, the constant-spectrum functor is symmetric monoidal and global
sections are lax symmetric monoidal.  Hence, the comparison is compatible
with external coefficient pairings and cup products.  All constructions are
natural in compactified morphisms and in arbitrary maps of \(E\).
\end{proof}

We now specialize to \(E=H\FF_2\). For every \(r\geq0\), functoriality in
the coefficient spectrum associates to \(\operatorname{sq}^r\) a homotopy
class of morphisms
\begin{equation}
\label{eq:steenrod-action-on-arithmetic-boundary-diagram}
        \mathscr B_{X,j}(H\FF_2)
        \longrightarrow
        \Sigma^r\mathscr B_{X,j}(H\FF_2).
\end{equation}

\begin{definition}[Arithmetic boundary Steenrod operations]
\label{def:arithmetic-boundary-steenrod-operations}
Via the component \(\eta_{\FF_2,X,j}\) of
Theorem~\ref{thm:spectral-arithmetic-boundary-comparison}, the operations
induced by~\eqref{eq:steenrod-action-on-arithmetic-boundary-diagram} are
denoted
\[
\begin{aligned}
        &\Sq^r_{c,\mathrm{fin},j},\quad
        \Sq^r_{\infty,j},\quad
        \widehat\Sq^r_{\infty,j},
\\
        &\Sq^r_{c,j},\quad
        \widehat\Sq^r_{c,j}
\end{aligned}
\qquad(r\geq0)
\]
on
\[
\begin{aligned}
        &H^*_{c,\mathrm{fin}}(X;\FF_2),\quad
        H^*_{\infty}(X;\FF_2),\quad
        \widehat H^*_{\infty}(X;\FF_2),
\\
        &H^*_c(X;\FF_2),\quad
        \widehat H^*_c(X;\FF_2),
\end{aligned}
\]
respectively. We also denote by
\[
        \Sq^r_{\mathbf C,j}\colon
        H^n_c(X_{\mathbf C};\FF_2)
        \longrightarrow
        H^{n+r}_c(X_{\mathbf C};\FF_2)
\]
the induced \(G_{\mathbf R}\)-equivariant auxiliary operation.
\end{definition}

\begin{theorem}[Canonical arithmetic boundary operations]
\label{thm:canonical-arithmetic-boundary-operations}
Transport along \(\eta_{\FF_2,X,j}\) makes the operations of
Definition~\ref{def:arithmetic-boundary-steenrod-operations} independent of
\(j\).  They define natural graded left \(\mathcal A_2\)-module structures
on the intrinsic arithmetic groups, and every morphism in the ordinary and
modified boundary long exact sequences is \(\mathcal A_2\)-linear.  They
satisfy
\[
        \Sq^0=\mathrm{id},
        \qquad
        \Sq^1=\beta,
\]
and the ordinary Adem relations, and are natural under proper morphisms.  On
the finite and geometric terms they agree with the relative stable Steenrod
operations of Theorem~\ref{thm:relative-stable-steenrod-operations}.
\end{theorem}

\begin{proof}
Let \(p\colon j'\to j\) be a morphism in
\(\operatorname{Comp}_{\mathbf{Z}}(X)\), and let
\[
        t_p(E)\colon
        \mathscr B_{X,j}(E)
        \longrightarrow
        \mathscr B_{X,j'}(E)
\]
be the induced transition morphism.  Naturality of \(\eta_{\FF_2}\) gives
\[
        \eta_{\FF_2,X,j'}\circ t_p(H\FF_2)
        \simeq
        \eta_{\FF_2,X,j}.
\]
Thus, \(t_p(H\FF_2)\) becomes the identity of
\(\mathscr D_{\FF_2}(X)\) after transport.  Bifunctoriality of \(\mathscr B\)
also gives
\[
        \bigl(\Sigma^rt_p(H\FF_2)\bigr)
        \circ
        \mathscr B_{X,j}(\operatorname{sq}^r)
        \simeq
        \mathscr B_{X,j'}(\operatorname{sq}^r)
        \circ
        t_p(H\FF_2).
\]
Hence, the transported operations agree whenever one compactification
dominates the other.  Since any two dense compactifications admit a common
domination, they agree for arbitrary compactifications;
see~\cite[Tag~0ATU]{STACKS-PROJECT}.

The \(\mathcal A_2\)-module identities and the Adem relations follow from
the corresponding identities among the stable coefficient operations.
Compatibility with the boundary morphisms follows because each coefficient
operation induces a morphism of the entire boundary diagram.  Applying
exactness in the coefficient spectrum to
\[
        H\FF_2\longrightarrow H\mathbf Z/4
        \longrightarrow H\FF_2
        \xrightarrow{\operatorname{sq}^1}\Sigma H\FF_2
\]
and using the coefficient naturality of \(\eta\) identifies \(\Sq^1\) with
the Bockstein on every term.

For a proper morphism \(f\colon Y\to X\), choose a morphism of
compactifications \((f,\bar f)\) lying over \(f\);
see~\cite[Tag~0ATT]{STACKS-PROJECT}.  Naturality of \(\eta_{\FF_2}\) gives
\[
        \eta_{\FF_2,Y,j_Y}\circ
        \mathscr B_{H\FF_2}(f,\bar f)
        \simeq
        \mathscr D_{\FF_2}(f)\circ
        \eta_{\FF_2,X,j_X}.
\]
Thus, the transported morphism is the intrinsic proper pullback and is
independent of the chosen compactified lift.  Bifunctoriality in the
coefficient spectrum shows that it commutes with all Steenrod operations.

Finally, the finite term is the relative cochain spectrum associated with
\(h(i)\), while, after forgetting the action, the geometric term is the one
associated with \(h(i_{\mathbf C})\).  The final assertion therefore follows
from Theorem~\ref{thm:relative-stable-steenrod-operations}.
\end{proof}

We henceforth omit \(j\) from the notation for the five intrinsic operations.  The theorem
provides the unary stable formalism used below; the internal and mixed products, and their
Cartan formulas, are constructed in the next subsection.

\subsection{Products and formal properties}
\label{subsec:steenrod-products-formal-properties}
We next construct the products used throughout the paper and prove their compatibility with
the Steenrod operations.

We use the standard pushout-product symmetric monoidal structure on arrow
\(\infty\)-categories; equivalently, it is the Day-convolution structure on
\(\operatorname{Fun}([1],-)\), with \([1]\) endowed with the minimum
operation;
see~\cite[Example~2.2.6.17 and Notation~4.1.8.10]{LurieHA}.

Let
\[
        f\colon A\longrightarrow X,
        \qquad
        g\colon B\longrightarrow Y
\]
be morphisms in \(\Pro(\mathcal S)\). Their pushout-product is
\[
        f\mathbin\square g\colon
        (A\times Y)\coprod_{A\times B}(X\times B)
        \longrightarrow
        X\times Y.
\]
The unit for this structure is \(\varnothing\to *\). Finite diagrams in
\(\Pro(\mathcal S)\) admit common level representations. Thus, if \(f\)
and \(g\) are represented over small cofiltered \(\infty\)-categories
\(T\) and \(S\), respectively, then \(f\mathbin\square g\) is represented
over \(T\times S\) by the arrows \(f_t\mathbin\square g_s\).

For an arrow \(k\colon K_0\to K_1\) in \(\mathcal S\), put
\[
        P(k):=
        \operatorname{Cof}\!\left(
                \Sigma^\infty_+K_0
                \longrightarrow
                \Sigma^\infty_+K_1
        \right).
\]
Equivalently, \(P\) is the composite
\[
        \bigl(\mathcal S^{[1]},\mathbin\square\bigr)
        \xrightarrow{\,(\Sigma^\infty_+)^{[1]}\,}
        \bigl(\mathrm{Sp}^{[1]},\mathbin\square\bigr)
        \xrightarrow{\,\operatorname{Cof}\,}
        (\mathrm{Sp},\wedge).
\]
The first functor is symmetric monoidal because \(\Sigma^\infty_+\) is
colimit-preserving and symmetric monoidal. Since smash products preserve
pushouts separately, the cofiber functor is also symmetric monoidal.
Consequently, there are coherent natural equivalences
\begin{equation}
\label{eq:cofiber-of-relative-pushout-product}
        P(f_t\mathbin\square g_s)
        \simeq
        P(f_t)\wedge P(g_s).
\end{equation}

\begin{proposition}[Products on relative cochain spectra]
\label{prop:products-relative-cochain-spectra}
Let
\[
        \mu\colon E_1\wedge E_2\longrightarrow E_3
\]
be a pairing of spectra. There is a natural external product
\begin{equation}
\label{eq:external-product-relative-cochain-spectra}
        C(f;E_1)\wedge C(g;E_2)
        \longrightarrow
        C(f\mathbin\square g;E_3).
\end{equation}
The construction is natural in the arrows and in \(\mu\). For compatible
systems of coefficient pairings, the resulting products respect
associativity, symmetry, and units. In particular, if
\(E\) is a commutative algebra spectrum, these products make
\[
        C(-;E)\colon
        \bigl(\Pro(\mathcal S)^{[1]}\bigr)^{\op}
        \longrightarrow
        \mathrm{Sp}
\]
lax symmetric monoidal, where the arrow category is equipped with the
pushout-product. The same assertions hold equivariantly.

For \(f\colon A\to X\) and \(E=H\FF_2\), the relative diagonal and the
graph of \(f\) give canonical morphisms of arrows
\[
        f\longrightarrow f\mathbin\square f,
        \qquad
        f\longrightarrow
        f\mathbin\square(\varnothing\longrightarrow X).
\]
In the first morphism, the maps on source and target are induced by
\(\Delta_A\) and \(\Delta_X\), respectively; in the second, they are
\[
        (\mathrm{id}_A,f)\colon A\longrightarrow A\times X
        \qquad\text{and}\qquad
        \Delta_X\colon X\longrightarrow X\times X.
\]
These morphisms induce, respectively, an internal product
\begin{equation}
\label{eq:internal-product-relative-cochain-spectra}
        C(f;H\FF_2)\wedge C(f;H\FF_2)
        \longrightarrow
        C(f;H\FF_2)
\end{equation}
and a relative--absolute product
\begin{equation}
\label{eq:relative-absolute-product-relative-cochain-spectra}
        C(f;H\FF_2)\wedge C(X;H\FF_2)
        \longrightarrow
        C(f;H\FF_2).
\end{equation}
Consequently, \(C(f;H\FF_2)\) is canonically a non-unital commutative
algebra object and a module over the commutative algebra
\(C(X;H\FF_2)\).
\end{proposition}

\begin{proof}
For arrows \(k,\ell\) in \(\mathcal S\), evaluation and the coefficient
pairing give a natural composite
\[
\begin{aligned}
        F(P(k),E_1)\wedge F(P(\ell),E_2)
        &\longrightarrow
        F\bigl(P(k)\wedge P(\ell),E_1\wedge E_2\bigr)
\\
        &\longrightarrow
        F\bigl(P(k\mathbin\square\ell),E_3\bigr),
\end{aligned}
\]
where the second map
uses~\eqref{eq:cofiber-of-relative-pushout-product}. For level
representations \((f_t)_{t\in T}\) and \((g_s)_{s\in S}\), these maps
induce
\[
\begin{aligned}
        C(f;E_1)\wedge C(g;E_2)
        &\simeq
        \varinjlim_{(t,s)\in(T\times S)^{\op}}
        \bigl(C(f_t;E_1)\wedge C(g_s;E_2)\bigr)
\\
        &\longrightarrow
        \varinjlim_{(t,s)\in(T\times S)^{\op}}
        C(f_t\mathbin\square g_s;E_3)
\\
        &\simeq
        C(f\mathbin\square g;E_3).
\end{aligned}
\]
Here smash products preserve colimits separately. The universal property of
the Ind-completion shows that the resulting pairing is independent of the
chosen presentations.

The asserted coherence properties follow from the symmetric monoidality of
\(P\) and the coherent lax symmetric monoidality of the function-spectrum
pairing. These coherences pass to the filtered-colimit extensions.

Finally, the relative diagonal equips \(f\) with a cocommutative
noncounital coalgebra structure for the pushout-product, while the graph
morphism equips it with a comodule structure over the cocommutative
coalgebra \(\varnothing\to X\). Applying the contravariant lax symmetric
monoidal functor \(C(-;H\FF_2)\) gives the asserted non-unital commutative
algebra and module structures.
\end{proof}

For the compactification arrow
\(i\colon\partial_jX\to\bar X\), the relative--absolute product of
Proposition~\ref{prop:products-relative-cochain-spectra} specializes to
\[
        C(h(i);H\FF_2)\wedge C(h(\bar X);H\FF_2)
        \longrightarrow
        C(h(i);H\FF_2).
\]
Under the arithmetic boundary comparison, this induces the action
\[
        H^*_{c,\mathrm{fin}}(X;\FF_2)
        \otimes
        H^*(\bar X;\FF_2)
        \longrightarrow
        H^*_{c,\mathrm{fin}}(X;\FF_2).
\]
When \(X\) is not proper over \(\Spec(\mathbf Z)\), this exhibits only the action of the
restriction image, along
\[
        j^*\colon H^*(\bar X;\FF_2)\longrightarrow H^*(X;\FF_2),
\]
of the intrinsic action of the ordinary \'{e}tale cohomology of \(X\) on
its compactly supported cohomology. We now construct the latter at the
spectrum level and compatibly on the entire arithmetic boundary diagram.

Throughout the remainder of this subsection, coefficient spectra lie in
\(\mathcal E_2\).

For \(E\in\mathcal E_2\), let
\[
        \operatorname{res}_{X,E}\colon
        \operatorname{triv}\!\left(
                R\Gamma(X_{\acute et};\underline E_X)
        \right)
        \longrightarrow
        R\Gamma\!\left(
                (X_{\mathbf C})_{\acute et};
                \underline E_{X_{\mathbf C}}
        \right)
\]
be equivariant base change, and set
\[
        \mathscr A_X(E)
        :=
        \left(
                R\Gamma(X_{\acute et};\underline E_X),
                R\Gamma\!\left(
                        (X_{\mathbf C})_{\acute et};
                        \underline E_{X_{\mathbf C}}
                \right),
                \operatorname{res}_{X,E}
        \right)
        \in
        \mathsf{Bnd}_{G_{\mathbf R}}.
\]

For
\[
        \mathcal F_1,\mathcal F_2
        \in
        \operatorname{Shv}(X_{\acute et};\mathrm{Sp}),
\]
the projection formula for \(j\colon X\hookrightarrow\bar X\) gives a
natural pairing
\begin{equation}
\label{eq:projection-formula-open-mixed-pairing}
        j_!\mathcal F_1\otimes Rj_*\mathcal F_2
        \simeq
        j_!\bigl(\mathcal F_1\otimes j^*Rj_*\mathcal F_2\bigr)
        \longrightarrow
        j_!(\mathcal F_1\otimes\mathcal F_2),
\end{equation}
where the final arrow is induced by the counit
\(j^*Rj_*\mathcal F_2\to\mathcal F_2\).

Now let \(E_1,E_2,E_3\in\mathcal E_2\), and let
\[
        \mu\colon E_1\wedge E_2\longrightarrow E_3
\]
be a pairing of underlying spectra. Taking
\(\mathcal F_a=\underline{E_a}_X\) for \(a=1,2\), composing with the
induced morphism
\[
        \underline{E_1}_X\otimes\underline{E_2}_X
        \longrightarrow
        \underline{E_3}_X,
\]
and applying spectrum-valued derived global sections gives
\begin{equation}
\label{eq:finite-compact-support-ordinary-mixed-pairing}
        R\Gamma_{c,\mathrm{fin}}(X;\underline{E_1}_X)
        \wedge
        R\Gamma(X_{\acute et};\underline{E_2}_X)
        \longrightarrow
        R\Gamma_{c,\mathrm{fin}}(X;\underline{E_3}_X).
\end{equation}
After base change to \(\mathbf C\), the analogous pairing is
\(G_{\mathbf R}\)-equivariant. The finite and geometric pairings are
compatible with
\[
        \rho^{\mathrm{arith,sp}}_{X,E_1},
        \qquad
        \operatorname{res}_{X,E_2},
        \qquad
        \rho^{\mathrm{arith,sp}}_{X,E_3}.
\]

\begin{proposition}[Coefficient pairings on the arithmetic boundary]
\label{prop:multiplicative-arithmetic-boundary-diagram}
Let \(E_1,E_2,E_3\in\mathcal E_2\). Every pairing
\[
        \mu\colon E_1\wedge E_2\longrightarrow E_3
\]
of underlying spectra induces natural morphisms
\begin{align}
\label{eq:coefficient-internal-pairing-arithmetic-boundary}
        \mathscr B_{X,j}(E_1)\wedge\mathscr B_{X,j}(E_2)
        &\longrightarrow
        \mathscr B_{X,j}(E_3),
\\
\label{eq:coefficient-mixed-pairing-arithmetic-boundary}
        \mathscr B_{X,j}(E_1)\wedge\mathscr A_X(E_2)
        &\longrightarrow
        \mathscr B_{X,j}(E_3)
\end{align}
in \(\mathsf{Bnd}_{G_{\mathbf R}}\). They are natural in the coefficient
spectra and in compactified morphisms, compatible with suspension, and
induce compatible pairings on every term and structure morphism of the
associated arithmetic boundary diagrams.

For the multiplication of \(H\FF_2\), the datum
\(\mathscr B_{X,j}(H\FF_2)\) is a non-unital commutative algebra object and
a module over the commutative algebra object
\(\mathscr A_X(H\FF_2)\). Consequently, after restriction of scalars along
the structure maps of \(\mathscr A_X(H\FF_2)\), every term of the associated
boundary diagram is a non-unital commutative algebra and a module over
\(R\Gamma(X_{\acute et};\underline{H\FF_2}_X)\). The map
\(\rho_{X,j,H\FF_2}\) and the structure morphisms
\[
        \delta^{\mathrm{ord}}_{X,j,H\FF_2},\qquad
        \delta^{\mathrm{Tate}}_{X,j,H\FF_2},\qquad
        \delta^{\mathrm{mod}}_{X,j,H\FF_2},\qquad
        \gamma_{X,j,H\FF_2}
\]
are multiplicative and module-linear.

Under \(\eta^{\mathrm{sp}}_{H\FF_2}\), these structures correspond to the
intrinsic arithmetic cup products and projection-formula actions. They are
therefore independent of the compactification and natural under proper
morphisms of \(X\).
\end{proposition}

\begin{proof}
For each of the arrows
\[
        k=h(i),
        \qquad
        k=\widetilde h(i_{\mathbf C})=h_{\mathbf R}(i_{\mathbf R}),
\]
compose the external product of
Proposition~\ref{prop:products-relative-cochain-spectra} with the map induced
contravariantly by the relative diagonal \(k\to k\mathbin\square k\):
\[
        C(k;E_1)\wedge C(k;E_2)
        \longrightarrow
        C(k\mathbin\square k;E_3)
        \longrightarrow
        C(k;E_3).
\]
Using the equivariant construction in the second case gives the finite and
equivariant geometric components, respectively, of~\eqref{eq:coefficient-internal-pairing-arithmetic-boundary}. Naturality
with respect to \(\widetilde\rho_j\) makes them compatible with
\(\rho_{X,j,E_a}\) for \(a=1,2,3\), and hence defines a morphism in
\(\mathsf{Bnd}_{G_{\mathbf R}}\).

The projection-formula
pairing~\eqref{eq:projection-formula-open-mixed-pairing}, together with
equivariant base change to \(\mathbf C\), gives a morphism
\[
        \mathscr D^{\mathrm{sp}}_{E_1}(X)
        \wedge
        \mathscr A_X(E_2)
        \longrightarrow
        \mathscr D^{\mathrm{sp}}_{E_3}(X)
\]
in \(\mathsf{Bnd}_{G_{\mathbf R}}\). The equivalences
\(\eta^{\mathrm{sp}}_{E_1}\) and \(\eta^{\mathrm{sp}}_{E_3}\) transport it
to~\eqref{eq:coefficient-mixed-pairing-arithmetic-boundary}.

The functors
\[
        (-)^{hG_{\mathbf R}},
        \qquad
        (-)^{tG_{\mathbf R}}
\]
are lax symmetric monoidal, and
\[
        \operatorname{can}\colon
        (-)^{hG_{\mathbf R}}
        \longrightarrow
        (-)^{tG_{\mathbf R}}
\]
is a lax symmetric monoidal natural transformation;
see~\cite[Theorem~I.3.1]{NikolausScholze2018}
and~\cite[Theorem~1.8.7(3)]{RaksitTateCohomology}. Since
\(\operatorname{triv}\dashv(-)^{hG_{\mathbf R}}\) is a symmetric-monoidal
adjunction, the adjunct \(\delta^{\mathrm{ord}}\) of \(\rho\) respects both
pairings. Lax monoidality of \(\operatorname{can}\) gives the same conclusion
for \(\delta^{\mathrm{Tate}}\) and \(\delta^{\mathrm{mod}}\).

The fiber terms and their pairings may be formed in the categories of
non-unital commutative algebras and modules, whose forgetful functors create
limits. This gives the asserted structures on \(C_{c,j}\) and
\(\widehat C_{c,j}\), including compatibility with \(\gamma\).

Finally, \(\eta_{H\FF_2}^{\mathrm{sp}}\) is natural on
\((\operatorname{Comp}^{\partial}_{\mathbf Z})^{\op}\). The restriction of
this natural transformation to \(\operatorname{Comp}_{\mathbf Z}(X)^{\op}\)
identifies the preceding
algebra- and module-valued diagram with the constant intrinsic diagram.
This proves compactification independence, and proper naturality follows
from the same comparison.
\end{proof}
Throughout, we denote by \(\cdot\) the products on
cohomology induced by the spectrum-level pairings constructed here.

The common input for the Cartan formulas is the corresponding identity of
stable coefficient operations. Let
\[
        \mu\colon
        H\FF_2\wedge H\FF_2
        \longrightarrow
        H\FF_2
\]
denote multiplication. For every \(n\geq0\), the ordinary Cartan formula is
equivalent to the identity
\begin{equation}
\label{eq:stable-steenrod-coproduct-cartan}
        \operatorname{sq}^n\circ\mu
        =
        \sum_{a+b=n}
        (\Sigma^n\mu)\circ
        (\operatorname{sq}^a\wedge\operatorname{sq}^b)
\end{equation}
in
\[
        [H\FF_2\wedge H\FF_2,\Sigma^nH\FF_2]_{\mathrm{Sp}},
\]
where
\[
        \Sigma^aH\FF_2\wedge\Sigma^bH\FF_2
        \simeq
        \Sigma^n(H\FF_2\wedge H\FF_2)
\]
is understood; see~\cite[Part~III, \S12]{AdamsStableHomotopy}.

Applying~\eqref{eq:stable-steenrod-coproduct-cartan} to the absolute
arrow
\[
    \varnothing\longrightarrow h(Y)
\]
and using the multiplicative stable \'{e}tale-shape
comparison~\eqref{eq:stable-etale-shape-comparison} gives, for every
\(n\geq0\) and homogeneous \(\alpha,\beta\in H_{\acute et}^*(Y;\FF_2)\),
\begin{equation}
\label{eq:ordinary-etale-cartan}
        \Sq^n(\alpha\cdot\beta)
        =
        \sum_{a+b=n}
        \Sq^a(\alpha)\cdot\Sq^b(\beta).
\end{equation}
Here \(Y\) may be any separated scheme of finite type over $\Spec\mathbf{Z}$, and \(\Sq^r\)
denotes the ordinary Steenrod square on \'{e}tale cohomology.

\begin{theorem}[Internal and mixed Cartan formulas for the arithmetic boundary]
\label{thm:arithmetic-boundary-internal-mixed-cartan}
Let \(\mathscr H^*(X;\FF_2)\) be any of
\[
        H^*_{c,\mathrm{fin}}(X;\FF_2),\quad
        H_c^*(X_{\mathbf C};\FF_2),\quad
        H^*_{\infty}(X;\FF_2),\quad
        \widehat H^*_{\infty}(X;\FF_2),\quad
        H_c^*(X;\FF_2),\quad
        \widehat H_c^*(X;\FF_2),
\]
and let \(\Sq_{\mathscr H}^n\) denote the corresponding Steenrod operation. Then,
for \(n\geq0\) and homogeneous \(x,y\in\mathscr H^*(X;\FF_2)\),
\begin{equation}
\label{eq:arithmetic-boundary-internal-cartan}
        \Sq_{\mathscr H}^n(x\cdot y)
        =
        \sum_{a+b=n}
        \Sq_{\mathscr H}^a(x)\cdot\Sq_{\mathscr H}^b(y).
\end{equation}

Put \(X_{\mathscr H}=X_{\mathbf C}\) in the geometric case and
\(X_{\mathscr H}=X\) otherwise. For
\[
        x\in\mathscr H^q(X;\FF_2),
        \qquad
        \alpha\in H_{\acute et}^p(X_{\mathscr H};\FF_2),
\]
one has
\begin{equation}
\label{eq:arithmetic-boundary-mixed-cartan}
        \Sq_{\mathscr H}^n(x\cdot\alpha)
        =
        \sum_{a+b=n}
        \Sq_{\mathscr H}^a(x)\cdot\Sq^b(\alpha),
\end{equation}
where \(\Sq^b\) denotes the ordinary Steenrod square on \'{e}tale
cohomology.

Together with the previously constructed \(\mathcal A_2\)-actions,
equation~\eqref{eq:arithmetic-boundary-internal-cartan} makes these theories graded
non-unital \(\mathcal A_2\)-module algebras.
\end{theorem}

\begin{proof}
Apply the coefficient-natural internal and mixed
pairings~\eqref{eq:coefficient-internal-pairing-arithmetic-boundary}
and~\eqref{eq:coefficient-mixed-pairing-arithmetic-boundary} to the stable
Cartan identity~\eqref{eq:stable-steenrod-coproduct-cartan}. This gives
the corresponding identities on every term of the spectral arithmetic
boundary diagram. Passing to homotopy groups and transporting along
\(\eta^{\mathrm{sp}}_{H\FF_2}\) gives~\eqref{eq:arithmetic-boundary-internal-cartan}
and~\eqref{eq:arithmetic-boundary-mixed-cartan}. On the ordinary factor, the absolute stable-shape
comparison~\eqref{eq:stable-etale-shape-comparison} identifies the
operation induced by \(\operatorname{sq}^b\) with the ordinary
\'{e}tale Steenrod square \(\Sq^b\) appearing
in~\eqref{eq:ordinary-etale-cartan}.
\end{proof}

We now turn to the supported variant. If \(Y\) is separated and of finite
type over \(\Spec(\mathbf Z)\) and
\[
        \mathcal M\in
        \operatorname{Shv}(Y_{\acute et};\mathrm{Sp}),
\]
let \(\widehat R\Gamma_c(Y;\mathcal M)\) denote the modified
compact-support spectrum obtained from the ordinary--Tate boundary
construction with coefficient sheaf \(\mathcal M\). This construction is
exact and functorial in \(\mathcal M\).

Let \(g\colon Z\hookrightarrow T\) be a closed immersion, where \(T\) is
separated and of finite type over \(\Spec(\mathbf Z)\). For
\[
        \mathcal M\in
        \operatorname{Shv}(T_{\acute et};\mathrm{Sp}),
\]
set
\begin{equation}
\label{eq:supported-boundary-spectra}
\begin{aligned}
        R\Gamma_Z(T;\mathcal M)
        &:=
        R\Gamma\!\left(Z_{\acute et};Rg^!\mathcal M\right),
\\
        \widehat R\Gamma_{c,Z}(T;\mathcal M)
        &:=
        \widehat R\Gamma_c\!\left(Z;Rg^!\mathcal M\right).
\end{aligned}
\end{equation}
Both constructions are exact and functorial in \(\mathcal M\). If
\(u\colon T\setminus Z\hookrightarrow T\) is the complementary open
immersion, recollement gives a natural equivalence
\begin{equation}
\label{eq:supported-cochains-as-relative-fiber}
        R\Gamma_Z(T;\mathcal M)
        \simeq
        \Fib\!\left(
                R\Gamma(T_{\acute et};\mathcal M)
                \longrightarrow
                R\Gamma\!\left(
                        (T\setminus Z)_{\acute et};
                        u^*\mathcal M
                \right)
        \right).
\end{equation}

Applying these functors to the morphism of constant sheaves induced by
\(\operatorname{sq}^n\colon H\FF_2\to\Sigma^nH\FF_2\) gives natural maps
\[
\begin{aligned}
        R\Gamma_Z(T;\underline{H\FF_2}_T)
        &\longrightarrow
        \Sigma^nR\Gamma_Z(T;\underline{H\FF_2}_T),
\\
        \widehat R\Gamma_{c,Z}(T;\underline{H\FF_2}_T)
        &\longrightarrow
        \Sigma^n\widehat R\Gamma_{c,Z}
        (T;\underline{H\FF_2}_T).
\end{aligned}
\]
We denote the induced operations on homotopy groups by
\[
        \Sq_Z^n
        \qquad\text{and}\qquad
        \widehat\Sq_{c,Z}^n.
\]
Under the stable \'{e}tale-shape comparison~\eqref{eq:stable-etale-shape-comparison}
and its compatibility with
recollement, \(\Sq_Z^n\) agrees with the relative stable operation associated
with
\[
        h(T\setminus Z)\longrightarrow h(T),
\]
and hence with the ordinary stable Steenrod square on \'{e}tale cohomology
with support.

For sheaves of spectra \(\mathcal L,\mathcal M\) on \(T_{\acute et}\),
there is a natural exceptional-pullback pairing
\begin{equation}
\label{eq:exceptional-pullback-module-pairing}
        g^*\mathcal L\otimes Rg^!\mathcal M
        \longrightarrow
        Rg^!(\mathcal L\otimes\mathcal M).
\end{equation}
Applying this pairing termwise to the boundary diagram defining modified
compact support, and then taking the defining fibers, gives, for
\(E,F,K\in\mathcal E_2\) and a pairing \(E\wedge F\to K\), a
coefficient-natural morphism
\begin{equation}
\label{eq:supported-mixed-pairing-spectra}
        \widehat R\Gamma_c(Z;\underline E_Z)
        \wedge
        R\Gamma_Z(T;\underline F_T)
        \longrightarrow
        \widehat R\Gamma_{c,Z}(T;\underline K_T).
\end{equation}

Since \(g\) is proper, the counit \(Rg_*Rg^!\to\mathrm{id}\) induces a
natural forget-support morphism
\begin{equation}
\label{eq:modified-forget-support-spectrum}
        \widehat R\Gamma_{c,Z}(T;\mathcal M)
        \longrightarrow
        \widehat R\Gamma_c(T;\mathcal M).
\end{equation}
Its naturality in \(\mathcal M\) implies compatibility with the coefficient
operations above.

Write
\[
\begin{aligned}
        H_Z^q(T;\mathcal M)
        &:=
        \pi_{-q}R\Gamma_Z(T;\mathcal M),
\\
        \widehat H_{c,Z}^q(T;\mathcal M)
        &:=
        \pi_{-q}\widehat R\Gamma_{c,Z}(T;\mathcal M).
\end{aligned}
\]

\begin{theorem}[Supported mixed Cartan formula]
\label{thm:modified-supported-mixed-cartan}
For \(n\geq0\) and classes
\[
        x\in\widehat H_c^q(Z;\FF_2),
        \qquad
        \alpha\in H_Z^p(T;\FF_2),
\]
one has
\begin{equation}
\label{eq:modified-compact-support-supported-mixed-cartan}
        \widehat\Sq_{c,Z}^n(x\cdot\alpha)
        =
        \sum_{a+b=n}
        \widehat\Sq_c^a(x)\cdot\Sq_Z^b(\alpha).
\end{equation}
\end{theorem}

\begin{proof}
Apply the coefficient-natural
pairing~\eqref{eq:supported-mixed-pairing-spectra} to the stable Cartan
identity~\eqref{eq:stable-steenrod-coproduct-cartan} and pass to homotopy
groups.
\end{proof}

\subsection{Instability and the Tate boundary}
\label{subsec:instability-tate-boundary}

The operations constructed above are the ordinary stable operations
\(\Sq^r\), indexed by \(r\geq0\).  On the relative cohomology associated with
morphisms of pro-anima, they agree with the classical relative Steenrod squares
and hence satisfy the usual instability relations.  The stable construction
alone, however, imposes no such relations on the cohomology of an arbitrary
spectrum.  We now determine precisely what survives for the arithmetic
boundary spectra.

\begin{lemma}[Ordinary real boundary as relative cochains]
\label{lem:ordinary-real-boundary-relative-cochains}
For every bounded-above spectrum \(E\), regarded with trivial
\(G_{\mathbf R}\)-action, there is a canonical equivalence
\begin{equation}
\label{eq:ordinary-real-boundary-as-relative-cochains}
        C(h(i_{\mathbf R});E)
        \xrightarrow{\ \sim\ }
        C_{R,j}(X;E),
\end{equation}
natural in \(E\) and in compactified morphisms. Under this equivalence,
\(\delta^{\mathrm{ord}}_{X,j,E}\) is induced contravariantly by
\(h(i_{\mathbf R})\to h(i)\). If \(E\) is a commutative algebra spectrum
whose underlying spectrum lies in \(\mathcal E_2\), the equivalence
respects the induced non-unital commutative algebra structures. By
naturality, it also commutes with stable coefficient operations.
\end{lemma}

\begin{proof}
Choose a level presentation of
\(\widetilde h(i_{\mathbf C})\) over a small cofiltered
\(\infty\)-category \(T\) by arrows
\[
        f_t\colon A_t\longrightarrow Y_t
        \qquad(t\in T)
\]
in \(\mathcal S^{BG_{\mathbf R}}\). The relative-to-absolute comparison
and the trivial action on \(E\) give a natural comparison
\[
\begin{aligned}
        C(h(i_{\mathbf R});E)
        &\simeq
        C\bigl(\widetilde h(i_{\mathbf C})/\!/G_{\mathbf R};E\bigr)
\\
        &\simeq
        \varinjlim_{t\in T^{\op}}
        C(f_t/\!/G_{\mathbf R};E)
\\
        &\simeq
        \varinjlim_{t\in T^{\op}}
        C(f_t;E)^{hG_{\mathbf R}}
\\
        &\longrightarrow
        \left(
        \varinjlim_{t\in T^{\op}}C(f_t;E)
        \right)^{hG_{\mathbf R}}
        \simeq
        C_{R,j}(X;E).
\end{aligned}
\]
For arbitrary \(E\), the displayed arrow need not be an equivalence. For
each \(t\), put
\[
        P_t:=
        \operatorname{Cof}\!\left(
                \Sigma^\infty_+A_t
                \longrightarrow
                \Sigma^\infty_+Y_t
        \right).
\]
The spectrum \(P_t\) is connective, since suspension spectra are
connective and connective spectra are closed under cofibers. If
\(E\in\mathrm{Sp}_{\leq N}\), it follows that
\[
        C(f_t;E)=F(P_t,E)\in\mathrm{Sp}_{\leq N}
\]
uniformly in \(t\). Since \(G_{\mathbf R}\) is finite, homotopy fixed points commute
with filtered colimits of uniformly bounded-above spectra. Indeed, the
homotopy-fixed-point spectral sequence
\[
        H^s\!\left(
                G_{\mathbf R};
                \pi_uC(f_t;E)
        \right)
        \Longrightarrow
        \pi_{u-s}\!\left(C(f_t;E)^{hG_{\mathbf R}}\right)
\]
has only finitely many terms in each total degree, and finite-group
cohomology commutes with filtered colimits. Thus, the comparison is an
equivalence.

By the construction of \(\widetilde\rho_j\)
in~\eqref{eq:equivariant-boundary-base-change-arrow} and the definition of
\(\delta^{\mathrm{ord}}_{X,j,E}\) in
Definition~\ref{def:spectral-arithmetic-boundary-diagram}, the
equivalence~\eqref{eq:ordinary-real-boundary-as-relative-cochains}
identifies \(\delta^{\mathrm{ord}}_{X,j,E}\) with the composite
\[
        C(h(i);E)
        \longrightarrow
        C(h(i_{\mathbf R});E)
        \xrightarrow{\ \sim\ }
        C_{R,j}(X;E).
\]
If \(E\) is a commutative algebra spectrum whose underlying spectrum lies
in \(\mathcal E_2\), compatibility of the
homotopy-orbit--homotopy-fixed-point comparison with the colax monoidal
structure on \((-) /\!/G_{\mathbf R}\) and the lax monoidal structure on
\((-)^{hG_{\mathbf R}}\) shows that~\eqref{eq:ordinary-real-boundary-as-relative-cochains}
intertwines the
products induced by the relative diagonals. Naturality in the coefficient
spectrum gives compatibility with stable coefficient operations.
\end{proof}

\begin{lemma}[Ordinary compact support as relative cochains]
\label{lem:ordinary-compact-support-relative-cochains}
Let
\[
        \kappa_{X,j}\colon
        h(\bar X_{\mathbf R})
        \coprod_{h((\partial_jX)_{\mathbf R})}
        h(\partial_jX)
        \longrightarrow
        h(\bar X)
\]
be the corner arrow of \(h(i_{\mathbf R})\to h(i)\), where the pushout
is formed in \(\Pro(\mathcal S)\). For every bounded-above spectrum
\(E\), there is a canonical equivalence
\[
        C(\kappa_{X,j};E)
        \xrightarrow{\ \sim\ }
        C_{c,j}(X;E),
\]
natural in \(E\) and in compactified morphisms. If \(E\) is a commutative
algebra spectrum whose underlying spectrum lies in \(\mathcal E_2\), the
equivalence respects the induced non-unital commutative algebra
structures. By naturality, it also commutes with stable coefficient
operations.
\end{lemma}

\begin{proof}
More generally, let
\[
        \alpha\colon
        (f'\colon A'\to X')
        \longrightarrow
        (f\colon A\to X)
\]
be a morphism of arrows in \(\Pro(\mathcal S)\), and let
\[
        \partial\alpha\colon
        X'\coprod_{A'}A\longrightarrow X
\]
be its corner arrow. After choosing a common level presentation, the
total-cofiber identity gives cofiber sequences
\[
\begin{aligned}
 &\operatorname{Cof}\!\left(
        \Sigma^\infty_+A'_t\longrightarrow\Sigma^\infty_+X'_t
 \right)
 \longrightarrow
 \operatorname{Cof}\!\left(
        \Sigma^\infty_+A_t\longrightarrow\Sigma^\infty_+X_t
 \right)
\\
 &\hspace{35mm}\longrightarrow
 \operatorname{Cof}\!\left(
        \Sigma^\infty_+(X'_t\coprod_{A'_t}A_t)
        \longrightarrow
        \Sigma^\infty_+X_t
 \right).
\end{aligned}
\]
Applying \(F(-,E)\), passing to filtered colimits, and using their
exactness gives a natural fiber sequence
\[
        C(\partial\alpha;E)
        \longrightarrow
        C(f;E)
        \longrightarrow
        C(f';E).
\]
For \(h(i_{\mathbf R})\to h(i)\), this becomes
\[
        C(\kappa_{X,j};E)
        \longrightarrow
        C(h(i);E)
        \longrightarrow
        C(h(i_{\mathbf R});E).
\]
By Lemma~\ref{lem:ordinary-real-boundary-relative-cochains}, the final
map identifies, for bounded-above \(E\), with
\(\delta^{\mathrm{ord}}_{X,j,E}\). Its fiber is therefore
\(C_{c,j}(X;E)\).

If \(E\) is a commutative algebra spectrum whose underlying spectrum lies
in \(\mathcal E_2\), naturality of the total-cofiber construction with
respect to relative diagonals, together with the multiplicative
compatibility of
Lemma~\ref{lem:ordinary-real-boundary-relative-cochains}, shows that the
equivalence respects the induced non-unital commutative algebra
structures. Naturality in the coefficient spectrum gives compatibility
with stable coefficient operations.
\end{proof}

No analogous relative-cochain model is asserted for
\(C_{\infty,j}(X;E)\) or \(\widehat C_{c,j}(X;E)\).

\begin{proposition}[Instability for relative pro-anima cohomology]
\label{prop:instability-relative-pro-anima-terms}
Let \(f\colon A\to Y\) be a morphism in \(\Pro(\mathcal S)\). For every
homogeneous class \(x\in H^q(f;H\FF_2)\) with \(q\geq0\),
\begin{equation}
\label{eq:instability-relative-pro-anima-terms}
        \Sq_f^r(x)=0\quad(r>q),
        \qquad
        \Sq_f^q(x)=x^2.
\end{equation}
Moreover,
\[
        H^q(f;H\FF_2)=0
        \qquad(q<0).
\]
In particular, the same assertions hold for
\[
\begin{gathered}
        H^*_{c,\mathrm{fin}}(X;\FF_2),\qquad
        H_c^*(X_{\mathbf C};\FF_2),\qquad
        H_\infty^*(X;\FF_2),\qquad
        H_c^*(X;\FF_2),
\\
        H_{\acute et}^*(X;\FF_2),\qquad
        H_{\acute et}^*(X_{\mathbf C};\FF_2),\qquad
        H_Z^*(T;\FF_2),
\end{gathered}
\]
where \(Z\hookrightarrow T\) is a closed immersion and \(T\) is separated
and of finite type over \(\Spec(\mathbf Z)\).
\end{proposition}

\begin{proof}
Choose a level presentation of \(f\) over a small cofiltered
\(\infty\)-category \(I\),
\[
        f_t\colon A_t\longrightarrow Y_t
        \qquad(t\in I).
\]
Replacing each \(f_t\) by its mapping-cylinder cofibration gives natural
identifications
\[
        H^*(f_t;H\FF_2)
        \cong
        \widetilde H^*(M(f_t)/A_t;\FF_2)
        \cong
        H^*(M(f_t),A_t;\FF_2).
\]
By the defining correspondence between stable coefficient operations and
classical Steenrod operations, \(C(f_t;\operatorname{sq}^r)\) induces the
classical Steenrod square on the middle term, and hence the classical
relative square on the final term. Classical instability and the top-square
identity therefore hold at every level;
see~\cite[\S3.4, property~(4)]{FengEtaleSteenrod}. Negative-degree cohomology
vanishes.

By Proposition~\ref{prop:relative-cochain-spectrum-formalism},
\[
        H^q(f;H\FF_2)
        \cong
        \varinjlim_{t\in I^{\op}}H^q(f_t;H\FF_2),
\]
and the operations and products are induced by their levelwise
counterparts. Passing to the filtered colimit proves the assertion for
\(f\).

The finite and geometric boundary terms are obtained from \(h(i)\) and
\(h(i_{\mathbf C})\), respectively. The ordinary real-boundary and
ordinary compact-support terms are treated by
Lemmas~\ref{lem:ordinary-real-boundary-relative-cochains}
and~\ref{lem:ordinary-compact-support-relative-cochains}, respectively.
The absolute terms arise from
\[
        \varnothing\longrightarrow h(X),
        \qquad
        \varnothing\longrightarrow h(X_{\mathbf C}),
\]
and the supported term from
\[
        h(T\setminus Z)\longrightarrow h(T).
\]
The relevant comparison equivalences preserve products and stable
coefficient operations, so the preceding identities transport to all
the listed theories.
\end{proof}

Let \(M\) be a \(G_{\mathbf R}\)-spectrum.  Recall the natural cofiber sequence
\begin{equation}
\label{eq:homotopy-orbit-fixed-point-tate-sequence-instability}
        M_{hG_{\mathbf R}}
        \longrightarrow
        M^{hG_{\mathbf R}}
        \xrightarrow{\operatorname{can}_M}
        M^{tG_{\mathbf R}}.
\end{equation}
When \(M\) is one of the coefficient-functorial \(G_{\mathbf R}\)-spectra constructed
above, naturality of \(\operatorname{can}\) makes it commute with every stable coefficient
operation.  When, in addition, \(M\) carries the coefficient pairings of
Proposition~\ref{prop:multiplicative-arithmetic-boundary-diagram}, the fact that
\(\operatorname{can}\) is a lax symmetric monoidal natural transformation makes this map
multiplicative. For
\[
        M:=C_{\mathbf C,j}(X;H\FF_2),
\]
the map \(\operatorname{can}_M\) is precisely
\[
        \delta^{\mathrm{Tate}}_{X,j,H\FF_2}.
\]
Moreover, Theorem~\ref{thm:spectral-arithmetic-boundary-comparison} gives a
natural \(G_{\mathbf R}\)-equivariant equivalence
\[
        M\simeq R\Gamma_c(X_{\mathbf C},\FF_2).
\]

\begin{proposition}[Comparison and instability in the bounded range]
\label{prop:ordinary-tate-boundary-bounded-range-comparison}
Suppose that, for some integer \(N_X\geq0\),
\[
        H_c^q(X_{\mathbf C};\FF_2)=0
        \qquad(q>N_X).
\]
Then the canonical morphisms induce isomorphisms
\begin{align}
\label{eq:ordinary-tate-boundary-bounded-range-comparison}
        H^n_{\infty}(X;\FF_2)
        &\xrightarrow{\ \sim\ }
        \widehat H^n_{\infty}(X;\FF_2)
        &&(n>N_X),
\\
\label{eq:ordinary-modified-compact-support-bounded-range-comparison}
        H^n_c(X;\FF_2)
        &\xrightarrow{\ \sim\ }
        \widehat H^n_c(X;\FF_2)
        &&(n>N_X+1).
\end{align}
These maps commute with all Steenrod operations and preserve products.
Their inverses preserve products whenever the degrees of both factors and
their product lie in the indicated ranges.

Consequently, if \(q>N_X\), then every
\(x\in\widehat H_\infty^q(X;\FF_2)\) satisfies
\[
        \widehat\Sq_\infty^r(x)=0\quad(r>q),
        \qquad
        \widehat\Sq_\infty^q(x)=x^2.
\]
If \(q>N_X+1\), then every
\(x\in\widehat H_c^q(X;\FF_2)\) satisfies
\[
        \widehat\Sq_c^r(x)=0\quad(r>q),
        \qquad
        \widehat\Sq_c^q(x)=x^2.
\]
\end{proposition}

\begin{proof}
Put
\[
        M:=C_{\mathbf C,j}(X;H\FF_2).
\]
By Theorem~\ref{thm:spectral-arithmetic-boundary-comparison},
\[
        \pi_{-q}M
        \cong
        H_c^q(X_{\mathbf C};\FF_2).
\]
Proposition~\ref{prop:instability-relative-pro-anima-terms} and the
hypothesis show that \(M\) is cohomologically bounded. The strongly
convergent homotopy-orbit spectral sequence
\[
        E_2^{-p,q}
        =
        H_p\!\left(
                G_{\mathbf R},
                H_c^q(X_{\mathbf C};\FF_2)
        \right)
        \Longrightarrow
        \pi_{p-q}\!\left(M_{hG_{\mathbf R}}\right),
        \qquad p\geq0,
\]
therefore gives
\[
        \pi_{-m}\!\left(M_{hG_{\mathbf R}}\right)=0
        \qquad(m>N_X).
\]
The long exact sequence associated with
\[
        M_{hG_{\mathbf R}}
        \longrightarrow
        M^{hG_{\mathbf R}}
        \xrightarrow{\operatorname{can}_M}
        M^{tG_{\mathbf R}}
\]
proves~\eqref{eq:ordinary-tate-boundary-bounded-range-comparison}.

In a stable \(\infty\)-category, composable morphisms
\[
        A\xrightarrow{u}B\xrightarrow{v}C
\]
determine a canonical fiber sequence
\[
        \Fib(u)\longrightarrow\Fib(vu)\longrightarrow\Fib(v).
\]
Applying this to
\[
        C_{c,\mathrm{fin},j}(X;H\FF_2)
        \xrightarrow{\delta^{\mathrm{ord}}_{X,j,H\FF_2}}
        M^{hG_{\mathbf R}}
        \xrightarrow{\operatorname{can}_M}
        M^{tG_{\mathbf R}},
\]
and using
\[
        \Fib(\operatorname{can}_M)\simeq M_{hG_{\mathbf R}},
\]
gives a fiber sequence
\[
        C_{c,j}(X;H\FF_2)
        \xrightarrow{\gamma_{X,j,H\FF_2}}
        \widehat C_{c,j}(X;H\FF_2)
        \longrightarrow
        M_{hG_{\mathbf R}}.
\]
Its long exact sequence
proves~\eqref{eq:ordinary-modified-compact-support-bounded-range-comparison}.

Proposition~\ref{prop:arithmetic-boundary-coefficient-functoriality}
shows that the comparison maps commute with Steenrod operations, while
Proposition~\ref{prop:multiplicative-arithmetic-boundary-diagram} shows
that they preserve products. The final assertions follow by transporting
the instability and top-square identities for
\(H_\infty^*(X;\FF_2)\) and \(H_c^*(X;\FF_2)\) through the corresponding
isomorphisms.
\end{proof}

The modified boundary fiber sequence gives an exact sequence
\begin{equation}
\label{eq:tate-boundary-instability-defect-sequence}
        \widehat H^{q-1}_{\infty}(X;\FF_2)
        \xrightarrow{\partial_{\infty}}
        \widehat H^q_c(X;\FF_2)
        \xrightarrow{\iota_c}
        H^q_{c,\mathrm{fin}}(X;\FF_2).
\end{equation}
Both arrows commute with the stable Steenrod operations, and \(\iota_c\) is multiplicative.
Put
\begin{equation}
\label{eq:tate-boundary-instability-ideal}
        I_{\infty}(X)
        :=
        \bigoplus_{k\in\mathbf Z}
        \operatorname{im}\!\left(
        \partial_{\infty}\colon
        \widehat H^{k-1}_{\infty}(X;\FF_2)
        \longrightarrow
        \widehat H^k_c(X;\FF_2)
        \right).
\end{equation}

\begin{theorem}[Instability modulo the Tate boundary]
\label{thm:instability-modulo-tate-boundary}
Let \(x\in\widehat H_c^q(X;\FF_2)\) be homogeneous.  Then
\begin{equation}
\label{eq:upper-square-is-tate-boundary}
        \widehat\Sq_c^r(x)
        \in
        \operatorname{im}\!\left(
        \partial_{\infty}\colon
        \widehat H^{q+r-1}_{\infty}(X;\FF_2)
        \longrightarrow
        \widehat H^{q+r}_c(X;\FF_2)
        \right)
        \qquad(r\geq0,\ r>q),
\end{equation}
and, if \(q\geq0\),
\begin{equation}
\label{eq:top-square-defect-is-tate-boundary}
        \widehat\Sq_c^q(x)+x^2
        \in
        \operatorname{im}\!\left(
        \partial_{\infty}\colon
        \widehat H^{2q-1}_{\infty}(X;\FF_2)
        \longrightarrow
        \widehat H^{2q}_c(X;\FF_2)
        \right).
\end{equation}
Consequently,
\begin{equation}
\label{eq:modified-compact-support-unstable-quotient}
        \widehat H_c^*(X;\FF_2)\big/I_{\infty}(X)
\end{equation}
is an unstable non-unital \(\mathcal A_2\)-algebra.  Its negative-degree part is zero.
\end{theorem}

\begin{proof}
By exactness of~\eqref{eq:tate-boundary-instability-defect-sequence}, it is enough to show
that the images under \(\iota_c\) of the classes
in~\eqref{eq:upper-square-is-tate-boundary}
and~\eqref{eq:top-square-defect-is-tate-boundary} vanish.  Suppose first that \(q\geq0\).
Since \(\iota_c\) is
\(\mathcal A_2\)-linear and multiplicative,
Proposition~\ref{prop:instability-relative-pro-anima-terms} gives
\[
        \iota_c\!\left(\widehat\Sq_c^r(x)\right)
        =\Sq^r_{c,\mathrm{fin}}(\iota_c(x))=0
        \qquad(r>q)
\]
and
\[
        \iota_c\!\left(\widehat\Sq_c^q(x)+x^2\right)
        =\Sq^q_{c,\mathrm{fin}}(\iota_c(x))+\iota_c(x)^2=0.
\]
If \(q<0\), then \(H^q_{c,\mathrm{fin}}(X;\FF_2)=0\), so
\(\iota_c(x)=0\) and \(\iota_c(\widehat\Sq_c^r(x))=0\) for every \(r\geq0\).
Moreover, \(I_{\infty}(X)=\ker(\iota_c)\) is an
\(\mathcal A_2\)-stable ideal.  Thus, the quotient inherits the asserted algebra structure
and instability relations.  Finally,
\(H^q_{c,\mathrm{fin}}(X;\FF_2)=0\) for \(q<0\), so every negative-degree modified class
lies in \(\ker(\iota_c)\).
\end{proof}

We next make the Tate contribution explicit.  Put
\[
        G_{\mathbf R}=C_2
\]
and let
\[
        \beta_{\mathbf R}\in H^1(BC_2;\FF_2)
\]
be the standard generator.

For a graded \(\FF_2\)-algebra \(R^*=\bigoplus_qR^q\), write
\[
        R^{\wedge,+}
        :=
        \bigcup_{q_0\in\mathbf Z}
        \prod_{q\geq q_0}R^q
\]
for its forward degree completion. Multiplication extends continuously to
\(R^{\wedge,+}\), and, since \(\Sq^r\) raises degree by \(r\), the total
square
\[
        \Sq:=\sum_{r\geq0}\Sq^r
\]
is well-defined there. The Cartan formula makes \(\Sq\) multiplicative. In
particular,
\[
        \bigl(\FF_2[\beta_{\mathbf R}^{\pm1}]\bigr)^{\wedge,+}
        =
        \FF_2[[\beta_{\mathbf R}]][\beta_{\mathbf R}^{-1}].
\]

\begin{proposition}[The \(C_2\)-fixed-point calculation]
\label{prop:c2-point-tate-steenrod-calculation}
For the trivial \(C_2\)-spectrum \(H\FF_2\), one has
\begin{equation}
\label{eq:c2-point-hfp-tate-rings}
        H^*\!\left((H\FF_2)^{hC_2}\right)
        =\FF_2[\beta_{\mathbf R}],
        \qquad
        H^*\!\left((H\FF_2)^{tC_2}\right)
        =\FF_2[\beta_{\mathbf R}^{\pm1}],
\end{equation}
and the canonical map is the localization at \(\beta_{\mathbf R}\).  For every
\(m\in\mathbf Z\) and \(r\geq0\),
\begin{equation}
\label{eq:steenrod-square-beta-integral-power}
        \Sq^r(\beta_{\mathbf R}^{m})
        =\binom{m}{r}\beta_{\mathbf R}^{m+r},
\end{equation}
where the integral binomial coefficient is reduced modulo \(2\). In
particular, in the forward degree completion
\(\FF_2[[\beta_{\mathbf R}]][\beta_{\mathbf R}^{-1}]\),
\begin{equation}
\label{eq:total-steenrod-square-beta-minus-one}
        \Sq(\beta_{\mathbf R}^{-1})
        =
        \beta_{\mathbf R}^{-1}(1+\beta_{\mathbf R})^{-1}
        =
        \sum_{r\geq0}\beta_{\mathbf R}^{r-1}
        =
        \beta_{\mathbf R}^{-1}+1+\beta_{\mathbf R}
        +\beta_{\mathbf R}^2+\cdots .
\end{equation}
\end{proposition}

\begin{proof}
The standard cyclic-group cohomology and Tate cohomology calculations
give~\eqref{eq:c2-point-hfp-tate-rings}; under these identifications the canonical map is the
inclusion of the polynomial ring into its Laurent localization.  Classical instability on
\(BC_2\) gives
\[
        \Sq^0(\beta_{\mathbf R})=\beta_{\mathbf R},
        \qquad
        \Sq^1(\beta_{\mathbf R})=\beta_{\mathbf R}^2,
        \qquad
        \Sq^r(\beta_{\mathbf R})=0\quad(r>1).
\]
The coefficient identity~\eqref{eq:stable-steenrod-coproduct-cartan}
and the lax symmetric monoidality of the Tate construction imply that the
total square is multiplicative in the forward degree completion. Since
\[
        \Sq(\beta_{\mathbf R})
        =
        \beta_{\mathbf R}(1+\beta_{\mathbf R}),
        \qquad
        \Sq(1)=1,
\]
it follows that
\[
        \Sq(\beta_{\mathbf R}^{m})
        =
        \beta_{\mathbf R}^{m}(1+\beta_{\mathbf R})^{m}
\]
in
\(\FF_2[[\beta_{\mathbf R}]][\beta_{\mathbf R}^{-1}]\) for every
\(m\in\mathbf Z\), where negative powers of \(1+\beta_{\mathbf R}\) are
formed in \(\FF_2[[\beta_{\mathbf R}]]^\times\), equivalently by formal
binomial expansion. Comparing homogeneous components of degree \(m+r\)
gives~\eqref{eq:steenrod-square-beta-integral-power}. Taking \(m=-1\)
gives~\eqref{eq:total-steenrod-square-beta-minus-one}.
\end{proof}

The failure in instability is not confined to negative cohomological degrees. The following
\(C_2\)-equivariant calculation gives a degree-zero counterexample; its
arithmetic realization is recorded in
Remark~\ref{rem:scope-of-tate-instability-examples}.

For a finite-dimensional real \(C_2\)-representation \(V\), write
\[
        S^V:=V\cup\{\infty\}
\]
for its one-point compactification, based at \(\infty\). Let \(1\) and
\(\sigma\) denote the trivial and sign real representations,
respectively. Under complex conjugation,
\(\mathbf C\cong 1\oplus\sigma\) as real \(C_2\)-representations.
Consequently, the standard affine chart gives a based
\(C_2\)-equivariant identification
\[
        \bigl(\mathbf P^1(\mathbf C),\infty\bigr)
        \simeq
        S^{1\oplus\sigma}.
\]

\begin{proposition}[The \(\mathbf P^1_{\mathbf R}\) test]
\label{prop:projective-line-real-tate-instability-test}
Let
\[
        M_{\mathbf P^1}:=C\!\left(\mathbf P^1(\mathbf C);H\FF_2\right),
\]
with the \(C_2\)-action induced by complex conjugation. Let
\(\eta\in H^2(M_{\mathbf P^1}^{hC_2})\) be the Borel-equivariant Thom
class associated with the displayed representation-sphere
identification. Then
\begin{align}
\label{eq:projective-line-real-hfp-ring}
        H^*(M_{\mathbf P^1}^{hC_2})
        &\cong
        \FF_2[\beta_{\mathbf R},\eta]/(\eta^2),
\\
\label{eq:projective-line-real-tate-ring}
        H^*(M_{\mathbf P^1}^{tC_2})
        &\cong
        \FF_2[\beta_{\mathbf R}^{\pm1},\eta]/(\eta^2),
\end{align}
and the canonical map is localization at \(\beta_{\mathbf R}\).  Moreover,
\begin{equation}
\label{eq:steenrod-on-projective-line-thom-class}
        \Sq^1(\eta)=\beta_{\mathbf R}\eta,
        \qquad
        \Sq^r(\eta)=0\quad(r\geq2),
\end{equation}
and, for \(m\in\mathbf Z\) and \(r\geq0\),
\begin{equation}
\label{eq:steenrod-on-projective-line-tate-monomial}
        \Sq^r(\beta_{\mathbf R}^{m}\eta)
        =\binom{m+1}{r}\beta_{\mathbf R}^{m+r}\eta.
\end{equation}
In particular, the class
\[
        z:=\beta_{\mathbf R}^{-2}\eta
        \in H^0(M_{\mathbf P^1}^{tC_2})
\]
satisfies
\begin{equation}
\label{eq:projective-line-degree-zero-tate-counterexample}
        \Sq^r(z)=\beta_{\mathbf R}^{r-2}\eta\neq0
        \quad(r\geq0),
        \qquad
        z^2=0.
\end{equation}
Thus, the upper-vanishing and top-square identities both fail on the Tate boundary, already
for a class of degree zero.
\end{proposition}

\begin{proof}
Under the displayed representation-sphere identification, reduced Borel cohomology is the
Thom cohomology of the real vector bundle \(1\oplus\sigma\) over \(BC_2\).  Its total
Stiefel--Whitney class is
\[
        w(1\oplus\sigma)=1+\beta_{\mathbf R}.
\]
The classical Thom identity therefore gives
\[
        \Sq(\eta)=(1+\beta_{\mathbf R})\eta.
\]
See, for example,~\cite[\S8]{MILNOR_STASHEFF}.
The \'{e}tale counterpart of this Thom-class characterization
is introduced below in Definition~\ref{def:etale-sw}.
Since Borel cohomology is the ordinary cohomology of a space, its top-square identity gives
\[
        \eta^2=\Sq^2(\eta)=w_2(1\oplus\sigma)\eta=0.
\]
This proves~\eqref{eq:projective-line-real-hfp-ring}
and~\eqref{eq:steenrod-on-projective-line-thom-class}. For this bounded two-cell object, the
Tate spectral sequence is obtained from the homotopy-fixed-point spectral sequence by
inverting \(\beta_{\mathbf R}\).  This
gives~\eqref{eq:projective-line-real-tate-ring}.  The componentwise Cartan formula
and~\eqref{eq:steenrod-square-beta-integral-power} now yield
\[
\begin{split}
        \Sq^r(\beta_{\mathbf R}^{m}\eta)
        &=\binom{m}{r}\beta_{\mathbf R}^{m+r}\eta
          +\binom{m}{r-1}\beta_{\mathbf R}^{m+r}\eta
\\
        &=\binom{m+1}{r}\beta_{\mathbf R}^{m+r}\eta.
\end{split}
\]
Taking \(m=-2\) and using \(\binom{-1}{r}\equiv1\pmod2\)
proves~\eqref{eq:projective-line-degree-zero-tate-counterexample}.
\end{proof}

\begin{remark}[Arithmetic realization and scope]
\label{rem:scope-of-tate-instability-examples}
Let \(X=\mathbf P^1_{\mathbf Z}\) with the identity compactification
\(j=\mathrm{id}_X\). Since \(\partial_jX=\varnothing\), the underlying
spectrum of \(C_{\mathbf C,j}(X;H\FF_2)\) is canonically equivalent to
\[
        C\bigl(h(\mathbf P^1_{\mathbf C});H\FF_2\bigr).
\]
Friedlander's functorial comparison theorem identifies, after profinite
completion, the \'{e}tale homotopy type of
\(\mathbf P^1_{\mathbf C}\) with the singular homotopy type of
\(\mathbf P^1(\mathbf C)\);
see~\cite[Theorem~8.4]{Friedlander}. Since this comparison is an isomorphism
on mod-\(2\) cohomology in every degree, it induces an equivalence of
\(H\FF_2\)-valued cochain spectra. Its compatibility with complex
conjugation promotes this to an equivalence
\[
        C_{\mathbf C,j}(X;H\FF_2)
        \xrightarrow{\ \sim\ }
        M_{\mathbf P^1}
\]
in \(\mathrm{Sp}^{BC_2}\).
Taking Tate constructions and applying
Theorem~\ref{thm:spectral-arithmetic-boundary-comparison} gives
\[
\begin{aligned}
        \widehat R\Gamma_\infty(X,\FF_2)
        &\simeq
        C_{\infty,j}(X;H\FF_2)
\\
        &\simeq
        C_{\mathbf C,j}(X;H\FF_2)^{tC_2}
        \simeq
        M_{\mathbf P^1}^{tC_2}.
\end{aligned}
\]
Compatibility of the comparison between Betti and \'{e}tale cohomology with
diagonals, together with
Proposition~\ref{prop:multiplicative-spectrum-valued-comparison},
shows that this equivalence preserves products. Its naturality in the
coefficient spectrum shows that it commutes with stable coefficient
operations. Thus, \(z=\beta_{\mathbf R}^{-2}\eta\) determines a class in
\(\widehat H_\infty^0(X;\FF_2)\) for which both upper vanishing and the
top-square identity fail.

The preceding calculations involve only the operations \(\Sq^r\) with
\(r\geq0\): Tate inversion creates negative-degree classes, not negatively
indexed operations. They do not imply that every modified compact-support
group contains a class violating instability. In general,
Theorem~\ref{thm:instability-modulo-tate-boundary} shows that such defects
lie in the ideal \(I_\infty(X)\) contributed by the Tate boundary.
\end{remark}

\section{Arithmetic \'{e}tale characteristic classes}\label{chapter:char_classes}
In this chapter we introduce the main characteristic classes studied in this paper---the \'{e}tale Stiefel--Whitney classes and the Wu classes---and establish one of our main theorems, Theorem~\ref{thm:main-wu}, which relates them for flat projective regular schemes over bases that are either finite fields or rings of \(S\)-integers, under the standing assumption that \(2\) is invertible on the base. We begin in~\S\ref{ssec:sw-classes} by defining \'{e}tale Stiefel--Whitney classes, and then in~\S\ref{sec:absolute-relative-wu-formulas} we introduce Wu classes and prove Theorem~\ref{thm:main-wu}.

\subsection{\'{E}tale Stiefel--Whitney classes}\label{ssec:sw-classes}
We follow Benoist~\cite[\S2.4]{Benoist}; see also
Feng~\cite[\S5]{FengEtaleSteenrod} for the finite-field case. Let \(X\)
be a regular Noetherian scheme on which \(2\) is invertible. Writing \(\underline{\mu_2}_X\) for the constant sheaf of square roots
of unity on \(X\), we use the canonical isomorphism
\[
        \underline{\FF_2}_X
        \xrightarrow{\ \sim\ }
        \underline{\mu_2}_X,
        \qquad
        1\longmapsto-1,
\]
and its tensor powers to identify
\(\underline{\mu_2}_X^{\otimes a}\simeq\underline{\FF_2}_X\)
for \(a\in\mathbf Z\). Throughout this chapter, we write \(\FF_2\)
and \(\mu_2^{\otimes a}\) for these constant coefficient sheaves
whenever the ambient scheme is clear. Through
these identifications, the Steenrod operations on cohomology with
supports used below are those constructed in
\S\ref{subsec:spectral-relative-steenrod-construction}.

Let \(p\colon E\to X\) be a vector bundle of rank \(r\), and write
\(i\colon X\hookrightarrow E\) for its zero section. We denote by
\(H^*_{X,\mathrm{\acute{e}t}}(E;-)\) the \'{e}tale cohomology of \(E\)
with supports in \(i(X)\). By Gabber's absolute purity, there is a Gysin
isomorphism
\[
        \phi_E\colon
        H^m_{\mathrm{\acute{e}t}}(X;\mathbb{F}_2)
        \xrightarrow{\ \sim\ }
        H^{m+2r}_{X,\mathrm{\acute{e}t}}
        (E;\mu_2^{\otimes r}),
\]
and we set
\[
        s_{X/E}:=
        \phi_E(1)
        \in
        H^{2r}_{X,\mathrm{\acute{e}t}}
        (E;\mu_2^{\otimes r}).
\]
This is the Thom class of \(E\); see~\cite[Expos\'e~XVI, Th\'eor\`eme~3.1.1]{Riou}.

\begin{definition}[See Benoist~{\cite[\S2.4]{Benoist}}]\label{def:etale-sw}
For \(j\ge 0\), the \(j\)-th \'{e}tale Stiefel--Whitney class of \(E\) is
\[
w_j(E):=\phi_E^{-1}\bigl(\Sq^j(s_{X/E})\bigr)\in
H^j_{\mathrm{\acute{e}t}}(X;\mathbb{F}_2),
\]
where \(\Sq^j\) denotes the \(j\)-th \'{e}tale Steenrod square on $H^{2r}_{X,\mathrm{\acute{e}t}}(E;\mu_2^{\otimes r})$. We write
\[
        w(E):=
        \sum_{j\geq0}w_j(E)
        \in
        H^*_{\mathrm{\acute{e}t}}(X;\mathbb{F}_2)
\]
for the total \'{e}tale Stiefel--Whitney class.
Proposition~\ref{prop:etale-sw-basic}(i) below shows that this sum is
finite. Equivalently, \(w_j(E)\) is characterized by
\[
\Sq^j(s_{X/E})=p^*\bigl(w_j(E)\bigr)\smile s_{X/E},
\]
and, after summing over \(j\),
\[
\Sq(s_{X/E})=p^*\bigl(w(E)\bigr)\smile s_{X/E}, \qquad
\Sq:=\sum_{j\ge 0}\Sq^j.
\]
\end{definition}

\begin{proposition}[cf.~{\cite[\S2.4]{Benoist}}, see also~{\cite[\S5.4]{FengEtaleSteenrod}}]\label{prop:etale-sw-basic}
The classes \(w_j(E)\) satisfy the following properties.
\begin{enumerate}
\item[(i)] One has \(w_0(E)=1\) and \(w_j(E)=0\) for \(j>2r\).
\item[(ii)] (Naturality) If \(f\colon Y\to X\) is a morphism of regular Noetherian \(\mathbf{Z}[1/2]\)-schemes, then
\[
f^*w_j(E)=w_j(f^*E)\in H^j_{\mathrm{\acute{e}t}}(Y;\mathbb{F}_2).
\]
\item[(iii)] (Whitney product formula) For every short exact sequence
\[
0\longrightarrow E'\longrightarrow E\longrightarrow E''\longrightarrow 0
\]
of vector bundles on \(X\), one has
\[
w(E)=w(E')\smile w(E'').
\]
In particular,
\[
w(E'\oplus E'')=w(E')\smile w(E'').
\]
\end{enumerate}
\end{proposition}

\begin{proof}
For \textup{(i)}, this is immediate from Definition~\ref{def:etale-sw} and the
standard properties of Steenrod squares: one has \(\Sq^0=\mathrm{id}\), hence
\[
w_0(E)=\phi_E^{-1}\bigl(\Sq^0(s_{X/E})\bigr)=\phi_E^{-1}(s_{X/E})=1,
\]
while Proposition~\ref{prop:instability-relative-pro-anima-terms},
applied to the supported class
\[
        s_{X/E}\in
        H^{2r}_{X,\mathrm{\acute{e}t}}
        (E;\mu_2^{\otimes r}),
\]
gives
\[
        \Sq^j(s_{X/E})=0
        \qquad(j>2r).
\]
Hence, \(w_j(E)=0\) for \(j>2r\).

For \textup{(ii)}, let \(\widetilde f\colon f^*E\to E\) be the canonical morphism
of total spaces. Since the square defined by the zero sections is cartesian,
the compatibility of Gysin morphisms (equivalently, of Thom classes) with
base change implies
\[
\widetilde f^*(s_{X/E})=s_{Y/f^*E};
\]
see~\cite[Expos\'e~XVI, Proposition 2.3.2]{Riou}. On the other hand,
\'etale Steenrod squares commute with pullback;
see~\cite[\S2.1--2.2]{Benoist}. Therefore,
\[
f^*w_j(E)
=\phi_{f^*E}^{-1}\bigl(\widetilde f^*\Sq^j(s_{X/E})\bigr)
=\phi_{f^*E}^{-1}\bigl(\Sq^j(\widetilde f^*s_{X/E})\bigr)
=\phi_{f^*E}^{-1}\bigl(\Sq^j(s_{Y/f^*E})\bigr)
=w_j(f^*E).
\]

\textup{(iii)} is exactly~\cite[Lemma~2.3]{Benoist}.
\end{proof}

\begin{remark}[Virtual bundles]
\label{rem:etale-sw-virtual-bundles}
Write
\[
        H_{\acute et}^*(X;\FF_2)^{\wedge,+}
        :=
        \prod_{n\geq0}H_{\acute et}^n(X;\FF_2)
\]
for the forward-degree completion. All identities below involving total
operations or virtual total characteristic classes are interpreted degreewise
in the corresponding forward-degree completion. Since \(w_0(E)=1\), the Whitney product
formula extends the total Stiefel--Whitney class uniquely to a homomorphism
\[
        w^{\acute et}\colon
        K_0(X)
        \longrightarrow
        \left(H_{\acute et}^*(X;\FF_2)^{\wedge,+}\right)^\times,
\]
characterized by
\[
        w^{\acute et}([E]-[F])
        =
        w(E)\,w(F)^{-1}.
\]
In particular, \(w^{\acute et}([E])=w(E)\) for every vector bundle \(E\).
\end{remark}

\begin{remark}[Grothendieck's splitting principle]\label{rem:etale-sw-splitting}
Let \(\pi\colon \operatorname{Fl}(E)\to X\) be the complete flag bundle of \(E\). By repeated use of the projective bundle formula, the pullback
\[
\pi^*\colon H^*_{\mathrm{\acute{e}t}}(X;\mathbb{F}_2)
\longrightarrow
H^*_{\mathrm{\acute{e}t}}(\operatorname{Fl}(E);\mathbb{F}_2)
\]
is injective. Moreover, \(\pi^*E\) admits a filtration by subbundles
\[
0=E_0\subset E_1\subset \cdots \subset E_r=\pi^*E
\]
whose successive quotients \(L_i:=E_i/E_{i-1}\) are line bundles. Consequently, any universal polynomial identity among characteristic classes may be checked after pullback to \(\operatorname{Fl}(E)\). If \(\overline{c}_i(E)\in H^{2i}_{\mathrm{\acute{e}t}}(X;\mathbb{F}_2)\) denotes the reduction modulo \(2\) of the \'{e}tale Chern class \(c_i(E)\), and if we set
\[
\overline{c}(E):=1+\overline{c}_1(E)+\cdots+\overline{c}_r(E),
\]
then on \(\operatorname{Fl}(E)\) one may write formally
\[
\pi^*\overline{c}(E)=\prod_{i=1}^r \bigl(1+\lambda_i\bigr),
\qquad
\lambda_i:=\overline{c}_1(L_i).
\]
See Grothendieck~\cite{GrothendieckChern}; see also Riou~\cite[Expos\'e~XVI, Th\'eor\`eme~1.3, Prop.~1.4-1.5]{Riou}.
In formulas involving Chern roots, we may suppress the injective
pullback \(\pi^*\), identifying classes on \(X\) with their images
on \(\operatorname{Fl}(E)\). The individual roots \(\lambda_i\)
remain classes on the flag bundle.
\end{remark}

\begin{definition}[Bockstein class]
\label{def:bockstein_class}
For a scheme \(Z\) on which \(2\) is invertible, the Bockstein class
\(\beta_Z\in H_{\acute et}^1(Z;\FF_2)\) is the image of
\(1_\mu:=-1\in H_{\acute et}^0(Z;\mu_2)\) under the connecting
morphism of
\[
        0\longrightarrow\mu_2\longrightarrow\mu_4
        \xrightarrow{(\cdot)^2}\mu_2\longrightarrow0,
\]
using the standing coefficient identifications.
\end{definition}

\begin{remark}
\label{lem:bockstein-equals-kummer}
Write \(\kappa_Z(u)\in H_{\acute et}^1(Z;\mu_2)\) for the Kummer
class of a unit \(u\in\Gamma(Z,\mathcal O_Z^\times)\);
see~\cite[Tag~03PK]{STACKS-PROJECT}. Comparing the defining sequence
with the Kummer sequence gives, by naturality of connecting morphisms,
\[
        \beta_Z=\kappa_Z(-1).
\]
For a field \(F\), we abbreviate \(\kappa_{\Spec F}\) to \(\kappa_F\).
\end{remark}

Combining Benoist's rank-one calculation with the Whitney product formula
and the splitting principle gives the following description;
compare~\cite[Theorem~5.10]{FengEtaleSteenrod} in the finite-field case.

\begin{theorem}[Benoist~{\cite[Lemma~2.4]{Benoist}}]
\label{thm:etale-sw-chern}
Let \(\pi\colon\operatorname{Fl}(E)\to X\) be the complete flag bundle of
Remark~\ref{rem:etale-sw-splitting}, and let \(L_1,\dots,L_r\) be the
successive line-bundle quotients of \(\pi^*E\). Put
\(\lambda_i:=\overline{c}_1(L_i)\). Then, in
\(H^*_{\mathrm{\acute{e}t}}(\operatorname{Fl}(E);\mathbb{F}_2)\),
\[
\begin{aligned}
        \pi^*w(E)
        &=\prod_{i=1}^r(1+\pi^*\beta_X+\lambda_i),\\
        \pi^*\overline{c}(E)
        &=\prod_{i=1}^r(1+\lambda_i).
\end{aligned}
\]
In particular, if \(\beta_X=0\) (for instance, if \(-1\) has a square root on \(X\)), then
\[
w(E)=\overline{c}(E).
\]
\end{theorem}

\begin{remark}\label{rem:etale-sw-finite-field}
Theorem~\ref{thm:etale-sw-chern} shows that $w(E) = \sum_{k=0}^r(1+\beta_X)^{r-k}\overline{c}_k(E)$. When $\beta_X^2 = 0$ so that $(1+\beta_X)^2 = 1$, e.g.~when the base is a finite field, one recovers Feng's formula~\cite[Theorem~5.10]{FengEtaleSteenrod}.
\end{remark}

\subsection{Wu classes and Wu formulas}
\label{sec:absolute-relative-wu-formulas}
We define total absolute Wu classes and derive the absolute Wu formula,
Theorem~\ref{thm:absolute-wu-formula}, from the relative Wu formulas
discussed in~\S\ref{subsubsec:rel_wu}. We compute the Wu class of the
arithmetic base in~\S\ref{subsec:absolute-arithmetic-vt}, making the
arithmetic formula explicit.

Throughout this subsection, \(B\) is either \(\Spec(k)\), with \(k\)
a finite field of odd characteristic, or \(\Spec(\mathcal O_{K,S})\),
where \(\mathcal O_{K,S}\) is the ring of \(S\)-integers in a number
field \(K\) and \(2\in\mathcal O_{K,S}^{\times}\).
Put \(\epsilon_B:=\dim(B)\in\{0,1\}\). Let \(X\) be a regular
separated \(B\)-scheme of finite type, flat and pure of relative
dimension \(d\), and set
\[
        D_X:=d+\epsilon_B=\dim(X),
        \qquad
        N_X:=2D_X+1.
\]

We write \(\widehat H_c^*(X;-)\) for ordinary compactly supported
cohomology in the finite-field case and for the modified compactly
supported cohomology of Definition~\ref{def:compact-support-over-X}
in the \(S\)-integer case. The Steenrod operations on these groups
are induced by the stable coefficient maps of
Definition~\ref{def:spectral-relative-steenrod-operations}, using
the intrinsic construction of
Definition~\ref{def:arithmetic-boundary-steenrod-operations}
in the \(S\)-integer case.

For a coefficient ring \(\Lambda\), we write
\(\underline{\mathrm{Hom}}_{X,\Lambda}\) for internal Hom in the
category of sheaves of \(\Lambda\)-modules on \(X_{\acute et}\),
and \(\Ext^r_{X,\Lambda}\) for global Ext in this category.
With these conventions, duality takes the following uniform form.

\begin{theorem}[Poincar\'e--Artin--Verdier duality]
\label{thm:milne-7-6}
Let \(F\) be a constructible sheaf of \(\FF_2\)-modules on
\(X_{\acute et}\). For every \(r\in\mathbf Z\), the Yoneda pairing
followed by the trace map induces a perfect pairing of finite groups
\begin{equation}
\label{eq:av}
\Ext^r_{X,\FF_2}
        \bigl(F,\mu_2^{\otimes D_X}\bigr)
\times
\widehat H_c^{N_X-r}(X;F)
\longrightarrow
\widehat H_c^{N_X}
        \bigl(X;\mu_2^{\otimes D_X}\bigr)
\xrightarrow{\ \int_X\ }
\FF_2.
\end{equation}
If \(F\) is locally constant and constructible, this becomes
\[
        H_{\acute et}^r\bigl(X;F^\vee(D_X)\bigr)
        \times
        \widehat H_c^{N_X-r}(X;F)
        \longrightarrow
        \FF_2,
\]
where
\[
        F^\vee:=
        \underline{\mathrm{Hom}}_{X,\FF_2}(F,\FF_2),
        \qquad
        F^\vee(D_X):=
        F^\vee\otimes_{\FF_2}\mu_2^{\otimes D_X}.
\]
\end{theorem}

In the \(S\)-integer case, this is generalized Artin--Verdier
duality; see~\cite[Thm.~4.6, Lem.~3.7, Thm.~3.9 and Lem.~3.11]{GeisserSchmidt}.
In the finite-field case, it is the usual Poincar\'e duality over a
finite field.

\begin{definition}[Absolute Wu class]
\label{def:absolute-wu-class}
For \(n\geq0\), the \(n\)-th \emph{absolute Wu class} is the unique
\(v_{X,n}\in H_{\acute et}^n(X;\FF_2)\) satisfying
\begin{equation}
\label{eq:wu_formula}
        \int_X\widehat\Sq_c^n(x)
        =
        \int_X x\cdot v_{X,n}
        \qquad
        \text{for all }
        x\in\widehat H_c^{N_X-n}
        \bigl(X;\mu_2^{\otimes D_X}\bigr).
\end{equation}
The \emph{total absolute Wu class} is
\[
        v_X:=\sum_{n\geq0}v_{X,n}
        \in H_{\acute et}^*(X;\FF_2)^{\wedge,+}.
\]
\end{definition}

Existence and uniqueness follow componentwise from
Theorem~\ref{thm:milne-7-6}: taking
\(F=\mu_2^{\otimes D_X}\) in~\eqref{eq:av} identifies the factor dual
to
\[
        \widehat H_c^{N_X-n}
        \bigl(X;\mu_2^{\otimes D_X}\bigr)
\]
with \(H_{\acute et}^n(X;\FF_2)\).

We require the definition of the virtual relative tangent and normal bundles for smoothable lci morphisms.
\begin{definition}[Virtual tangent and normal bundles, cf.~{\cite[Example~2.3.10]{DegliseJinKhanFundamental}}, see also~{\cite[Tag 08SH]{STACKS-PROJECT}}]
\label{def:virtual-normal-bundle}
Let
\[
f\colon Y\to X
\]
be a smoothable local complete intersection morphism, i.e.~a morphism admitting a factorization as
\[
Y \xhookrightarrow{i} P \xrightarrow{p} X
\]
with \(i\) a regular closed immersion and \(p\) smooth. The \emph{virtual tangent bundle} of \(f\) is the class:
\[
\tau_f=[i^*T_{P/X}]-[N_{Y/P}] \in K_0(Y),
\]
where $T_{P/X}$ and $N_{Y/P}$ are the honest tangent and normal bundles associated to the maps $i$ and $p$, respectively. We define the \emph{virtual normal bundle} of \(f\) to be the class
\[
\nu_f:=-\tau_f\in K_0(Y).
\]
In particular:
\begin{enumerate}
\item if \(f\) is smooth, then
\[
\tau_f = [T_{Y/X}],\qquad \nu_f = -[T_{Y/X}].
\]
\item If \(f=i\) is a regular closed immersion, then
\[
\tau_f = -[N_{Y/X}],\qquad \nu_f=[N_{Y/X}].
\]
\end{enumerate}
\end{definition}

For the morphisms considered below, the virtual tangent bundle is always
defined. Indeed, if \(f\colon X\to B\) is projective with \(X\) regular,
choose a factorization
\[
        X\xhookrightarrow{i}\mathbb P_B^N\xrightarrow{p}B.
\]
Since \(B\), and hence \(\mathbb P_B^N\), is regular, \(i\) is a regular
closed immersion by~\cite[Tag~0E9J]{STACKS-PROJECT}. 
Thus, under our standing flatness and dimension hypotheses, \(f\) is smoothable lci of virtual relative dimension \(d\).

The total stable Steenrod square on ordinary \'{e}tale cohomology,
\[
        \Sq:=\sum_{n\geq0}\Sq^n,
\]
extends degreewise to the forward-degree completion. More precisely, if
\[
        a=\sum_{j\geq0}a_j
        \in H_{\acute et}^*(X;\FF_2)^{\wedge,+},
\]
then
\[
        \bigl(\Sq(a)\bigr)_m
        =
        \sum_{i+j=m}\Sq^i(a_j).
\]
The Cartan formula~\eqref{eq:ordinary-etale-cartan} makes \(\Sq\) a
continuous ring endomorphism. Since \(\Sq^0=\mathrm{id}\) and every
other component strictly raises forward degree, it is a continuous
ring automorphism. We denote its inverse by \(\Sq^{-1}\). Thus,
\(\Sq^{-1}\) always refers to the inverse of the total stable square
on ordinary \'{e}tale cohomology.

\begin{theorem}[Absolute Wu formula]
\label{thm:absolute-wu-formula}
Let \(X\) and \(B\) be as above, and suppose that the structural
morphism \(f\colon X\to B\) is projective. Writing \(v_B\) for the
total absolute Wu class of \(B\), one has
\[
        v_X
        =
        \Sq^{-1}\!\bigl(w^{\acute et}(\tau_f)\bigr)\cdot f^*v_B
        \in H_{\acute et}^*(X;\FF_2)^{\wedge,+}.
\]
\end{theorem}

For later reference, if \(B=\Spec(k)\) is a finite field, then
\[
        v_B=1.
\]
Indeed, \(\Sq^0=\mathrm{id}\) gives \(v_{B,0}=1\), while
\(\Sq^1(1)=0\) gives \(v_{B,1}=0\). There are no further components
because
\[
        H_{\acute et}^n(B;\FF_2)=0
        \qquad(n>1).
\]

For geometrically connected \(X\) in the finite-field case,
Theorem~\ref{thm:absolute-wu-formula} therefore recovers
Feng's Wu formula~\cite[Theorem~6.5]{FengEtaleSteenrod}.
\begin{proof}[Proof of Theorem~\ref{thm:absolute-wu-formula}]
Put
\[
        t_f:=
        \Sq^{-1}\!\bigl(w^{\acute et}(\tau_f)\bigr).
\]
Let
\[
        x\in
        \widehat H_c^q
        \bigl(X;\mu_2^{\otimes D_X}\bigr)
\]
be homogeneous. Products with \(t_f\), total Steenrod operations, and
Gysin pushforwards are interpreted degreewise on the resulting
bounded-below forward series. Thus, only finitely many summands
contribute in each fixed degree.

Write \(\widehat f_{*,c}\) for the ordinary Gysin pushforward \(f_*\)
in the finite-field case and for the modified compact-support
pushforward of Definition~\ref{def:compact-support-pushforwards}
in the \(S\)-integer case.
Since \(f\) has relative dimension \(d\), its Gysin pushforward lowers
cohomological degree by \(2d\) and Tate twist by \(d\). In particular,
it carries the dualizing twist
\[
        D_X=d+\epsilon_B
\]
of \(X\) to the dualizing twist \(\epsilon_B\) of \(B\).

The projection formula, compatibility with the duality traces
in~\eqref{eq:modified-gysin-trace-compatibility} and its ordinary
finite-field analog, and the defining property of \(v_B\), applied
degreewise, give
\[
\begin{aligned}
        \int_X x\cdot t_f\cdot f^*v_B
        &=
        \int_B
        \widehat f_{*,c}(x\cdot t_f)\cdot v_B
        \\
        &=
        \int_B
        \widehat\Sq_c\!\left(
                \widehat f_{*,c}(x\cdot t_f)
        \right).
\end{aligned}
\]
Here and below, each integral means that the component in the
appropriate duality degree is taken before applying the trace.

By Theorem~\ref{thm:benoist-ordinary-wu-modified-section}
in the finite-field case and
Theorem~\ref{thm:modified-compactly-supported-relative-wu-declared}
in the \(S\)-integer case, we have
\[
\begin{aligned}
        \widehat\Sq_c\!\left(
                \widehat f_{*,c}(x\cdot t_f)
        \right)
        &=
        \widehat f_{*,c}\!\left(
                \widehat\Sq_c(x\cdot t_f)\cdot
                w^{\acute et}(\nu_f)
        \right)
        \\
        &=
        \widehat f_{*,c}\!\left(
                \widehat\Sq_c(x)\cdot
                \Sq(t_f)\cdot
                w^{\acute et}(\nu_f)
        \right).
\end{aligned}
\]
The second equality is the stable mixed Cartan formula; in the
arithmetic case, see
Theorem~\ref{thm:arithmetic-boundary-internal-mixed-cartan}. By the
definition of \(t_f\),
\[
        \Sq(t_f)=w^{\acute et}(\tau_f),
\]
and, since \(\nu_f=-\tau_f\),
\[
        w^{\acute et}(\tau_f)\cdot
        w^{\acute et}(\nu_f)=1.
\]
Consequently,
\[
        \widehat\Sq_c\!\left(
                \widehat f_{*,c}(x\cdot t_f)
        \right)
        =
        \widehat f_{*,c}\!\left(\widehat\Sq_c(x)\right).
\]
Trace compatibility now yields
\[
        \int_X x\cdot t_f\cdot f^*v_B
        =
        \int_B
        \widehat f_{*,c}\!\left(\widehat\Sq_c(x)\right)
        =
        \int_X\widehat\Sq_c(x).
\]
Thus, \(t_f\cdot f^*v_B\) satisfies the componentwise identities of
Definition~\ref{def:absolute-wu-class}. Their uniqueness proves the
formula.
\end{proof}

\begin{theorem}[Absolute arithmetic Wu class]
\label{thm:absolute-arithmetic-wu-class}
Let
\[
        B=\Spec\mathcal O_{K,S},
        \qquad
        2\in\mathcal O_{K,S}^{\times}.
\]
Then, in the forward-degree completion,
\[
        v_B
        =
        \sum_{n\geq0}
        \binom{2-n}{n}\beta_B^n
        =
        1+\beta_B+\beta_B^3+\beta_B^4+\beta_B^5
        +\beta_B^8+\beta_B^9+\cdots ,
\]
where the generalized integral binomial coefficients are reduced
modulo \(2\).
\end{theorem}

Theorem~\ref{thm:absolute-arithmetic-wu-class} is proved in
\S\ref{subsec:absolute-arithmetic-vt}.

\subsubsection{The modified compact-support relative Wu formula}
\label{subsubsec:rel_wu}

Throughout this subsection, let
\[
        B=\Spec\mathcal O_{K,S},
        \qquad
        2\in\mathcal O_{K,S}^{\times}.
\]
Unless otherwise stated, schemes are regular, separated, and of finite
type over \(B\), and admit ample invertible sheaves. Modified compact
support is understood intrinsically, independently of a chosen
compactification.

We first recall the ordinary formula used in the projective-bundle
calculation.

\begin{theorem}[Benoist's relative Wu formula
{\cite[Theorem~2.5]{Benoist}}]
\label{thm:benoist-ordinary-wu-modified-section}
Let \(f\colon Y\to X\) be a proper morphism of regular Noetherian
\(\mathbf Z[1/2]\)-schemes admitting ample invertible sheaves, pure of
virtual relative dimension \(-c\), with virtual normal bundle \(\nu_f\).
For \(q,r\in\mathbf Z\) and \(x\in H_{\acute et}^q(Y;\mu_2^{\otimes r})\),
one has, in the forward-degree completion,
\[
        \Sq(f_*x)
        =
        f_*\!\left(\Sq(x)\cdot w^{\acute et}(\nu_f)\right).
\]
\end{theorem}

We next construct the pushforward appearing in the modified formula.

\begin{definition}[Riou's Gysin morphism]
\label{def:riou-gysin-morphism}
Let \(f\colon Y\to X\) be a smoothable local complete intersection
morphism of pure virtual relative dimension \(-c\). For \(r\in\mathbf Z\),
Riou's canonical Gysin class is
\[
        \operatorname{Cl}_f^r\colon
        \mu_{2,Y}^{\otimes r}
        \longrightarrow
        Rf^!\mu_{2,X}^{\otimes(r+c)}[2c],
\]
where
\(\mu_{2,Y}^{\otimes r}\simeq f^*\mu_{2,X}^{\otimes r}\); see~\cite[Expos\'e~XVI,
D\'efinition~2.5.11]{Riou}. If \(f\) is proper, its
adjoint is
\begin{equation}
\label{eq:riou_gysin}
        \operatorname{gys}_f^r\colon
        Rf_!\mu_{2,Y}^{\otimes r}
        \longrightarrow
        \mu_{2,X}^{\otimes(r+c)}[2c];
\end{equation}
compare~\cite[Expos\'e~XVI, Remarque~2.5.14]{Riou}. It induces the
ordinary Gysin pushforward
\[
        f_*\colon
        H_{\acute et}^q(Y;\mu_2^{\otimes r})
        \longrightarrow
        H_{\acute et}^{q+2c}
        (X;\mu_2^{\otimes(r+c)}).
\]
For a regular closed immersion, \(\operatorname{Cl}_f^r\) is the
absolute-purity equivalence;
see~\cite[Expos\'e~XVI, Th\'eor\`eme~3.1.1]{Riou}.
\end{definition}

\begin{definition}[Arithmetic Gysin pushforwards]
\label{def:compact-support-pushforwards}
Let \(f\colon Y\to X\) be proper and as in
Definition~\ref{def:riou-gysin-morphism}. Applying~\eqref{eq:riou_gysin} to the
finite and archimedean boundary
constructions gives a natural morphism of fiber sequences
\[
\begin{tikzcd}[column sep=large,row sep=large]
\widehat R\Gamma_c(Y;\mu_2^{\otimes r})
  \arrow[r]\arrow[d]
&
R\Gamma_{c,\mathrm{fin}}(Y;\mu_2^{\otimes r})
  \arrow[r]\arrow[d]
&
\widehat R\Gamma_\infty(Y;\mu_2^{\otimes r})
  \arrow[d]
\\
\widehat R\Gamma_c(X;\mu_2^{\otimes(r+c)})[2c]
  \arrow[r]
&
R\Gamma_{c,\mathrm{fin}}(X;\mu_2^{\otimes(r+c)})[2c]
  \arrow[r]
&
\widehat R\Gamma_\infty(X;\mu_2^{\otimes(r+c)})[2c].
\end{tikzcd}
\]
On cohomology, denote the three vertical maps by
\[
\begin{aligned}
\widehat f_{*,c}&\colon
\widehat H_c^q(Y;\mu_2^{\otimes r})
\longrightarrow
\widehat H_c^{q+2c}(X;\mu_2^{\otimes(r+c)}),
\\
f_{*,c,\mathrm{fin}}&\colon
H_{c,\mathrm{fin}}^q(Y;\mu_2^{\otimes r})
\longrightarrow
H_{c,\mathrm{fin}}^{q+2c}(X;\mu_2^{\otimes(r+c)}),
\\
\widehat f_{*,\infty}&\colon
\widehat H_\infty^q(Y;\mu_2^{\otimes r})
\longrightarrow
\widehat H_\infty^{q+2c}(X;\mu_2^{\otimes(r+c)}).
\end{aligned}
\]
\end{definition}

\begin{proposition}[Formal properties of arithmetic Gysin pushforwards]
\label{prop:modified-gysin-formalism}
Let
\[
        Z\xrightarrow{\,g\,}Y\xrightarrow{\,f\,}X
\]
be composable proper morphisms as above, of pure virtual relative
dimensions \(-d\) and \(-c\), respectively. Then
\[
        \widehat{(f\circ g)}_{*,c}(x)
        =
        \widehat f_{*,c}\!\left(\widehat g_{*,c}(x)\right).
\]
Moreover, for
\[
\begin{aligned}
x&\in\widehat H_c^q(Y;\mu_2^{\otimes r}),
&
\alpha&\in H_{\acute et}^p(X;\mu_2^{\otimes s}),
\\
y&\in\widehat H_c^q(X;\mu_2^{\otimes r}),
&
\beta&\in H_{\acute et}^p(Y;\mu_2^{\otimes s}),
\end{aligned}
\]
one has
\begin{align}
\label{eq:modified-gysin-projection-formula}
        \widehat f_{*,c}(x\cdot f^*\alpha)
        &=
        \widehat f_{*,c}(x)\cdot\alpha,
\\
\label{eq:modified-gysin-mixed-projection-formula}
        \widehat f_{*,c}(\widehat f_c^*y\cdot\beta)
        &=
        y\cdot f_*(\beta),
\end{align}
where \(\widehat f_c^*\) denotes proper pullback on modified
compact-support cohomology.

Finally, suppose that \(X\) and \(Y\) are equidimensional and lie in
the duality setting of Theorem~\ref{thm:milne-7-6}. Write
\[
        D_X:=\dim(X),
        \qquad
        D_Y:=\dim(Y).
\]
Since \(f\) has pure virtual relative dimension \(-c\), one has
\(D_X=D_Y+c\). The top-degree pushforward is compatible with the
Artin--Verdier trace maps: for
\[
        z\in
        \widehat H_c^{2D_Y+1}
        \bigl(Y;\mu_2^{\otimes D_Y}\bigr),
\]
one has
\begin{equation}
\label{eq:modified-gysin-trace-compatibility}
        \int_X\widehat f_{*,c}(z)=\int_Y z.
\end{equation}
\end{proposition}

\begin{proof}
Compatibility with composition follows
from~\cite[Expos\'e~XVI, Th\'eor\`eme~2.5.12]{Riou}. The two projection
formulas follow from the \(f^*(-)\)-module compatibility of Riou's Gysin
class and the projection formula for \(Rf_{!}\). These identities persist
under archimedean base change and passage to the defining boundary
fiber sequences.

Trace compatibility follows from the construction of generalized
Artin--Verdier duality via \(Rf_{!}\);
see~\cite[Proof of Theorem~4.6, especially (4.8)--(4.9)]{GeisserSchmidt}.
Under the purity identifications used above, its transitivity agrees
with Riou's compatibility of Gysin classes with composition.
\end{proof}

\begin{theorem}[Modified compact-support relative Wu formula]
\label{thm:modified-compactly-supported-relative-wu-declared}
Let \(f\colon Y\to X\) be a proper smoothable local complete intersection
morphism of pure virtual relative dimension \(-c\) under the standing
hypotheses above, and let
\(\nu_f\in K_0(Y)\) be its virtual normal bundle. For
\(q,r\in\mathbf Z\), \(n\geq0\), and
\[
        x\in\widehat H_c^q(Y;\mu_2^{\otimes r}),
\]
one has
\begin{equation}
\label{eq:modified-compact-support-relative-wu-componentwise}
        \widehat\Sq_c^n\!\left(\widehat f_{*,c}x\right)
        =
        \sum_{a+b=n}
        \widehat f_{*,c}\!\left(
                \widehat\Sq_c^a(x)\cdot
                w_b^{\acute et}(\nu_f)
        \right).
\end{equation}
Equivalently, in the forward-degree completion,
\[
        \widehat\Sq_c\!\left(\widehat f_{*,c}x\right)
        =
        \widehat f_{*,c}\!\left(
                \widehat\Sq_c(x)\cdot
                w^{\acute et}(\nu_f)
        \right).
\]
\end{theorem}

\begin{proof}
We prove the componentwise
identity~\eqref{eq:modified-compact-support-relative-wu-componentwise}.

We first observe that the formula is stable under composition. Suppose that
\[
        Y\xrightarrow{\,g\,}P\xrightarrow{\,p\,}X
\]
are morphisms for which~\eqref{eq:modified-compact-support-relative-wu-componentwise} holds, and put \(f=p\circ g\). By
Proposition~\ref{prop:modified-gysin-formalism}, the Gysin maps are
compatible with composition and satisfy the projection formula. Hence,
\[
\begin{aligned}
\widehat\Sq_c^n\!\left(\widehat f_{*,c}x\right)
&=
\sum_{a+b=n}
\widehat p_{*,c}\!\left(
        \widehat\Sq_c^a(\widehat g_{*,c}x)
        \cdot w_b^{\acute et}(\nu_p)
\right)
\\
&=
\sum_{u+v+b=n}
\widehat p_{*,c}\!\left(
        \widehat g_{*,c}\!\left(
                \widehat\Sq_c^u(x)\cdot
                w_v^{\acute et}(\nu_g)
        \right)
        \cdot w_b^{\acute et}(\nu_p)
\right)
\\
&=
\sum_{u+v+b=n}
\widehat f_{*,c}\!\left(
        \widehat\Sq_c^u(x)\cdot
        w_v^{\acute et}(\nu_g)\cdot
        g^*w_b^{\acute et}(\nu_p)
\right).
\end{aligned}
\]
Since
\[
        \nu_f=\nu_g+g^*\nu_p
\]
and \(w^{\acute et}\) is multiplicative, the last expression is
\[
        \sum_{u+k=n}
        \widehat f_{*,c}\!\left(
                \widehat\Sq_c^u(x)\cdot
                w_k^{\acute et}(\nu_f)
        \right).
\]
By the standing properness and ampleness hypotheses, \(f\) is projective
and hence admits a factorization
\[
        Y\xhookrightarrow{\,g\,}\mathbb P_X^N
        \xrightarrow{\,p\,}X.
\]
Since \(f\) is lci and \(p\) is smooth, \(g\) is a regular closed
immersion. It therefore suffices to prove the formula for regular closed
immersions and projective-space bundles.

\medskip
\noindent
\emph{Regular closed immersions.}
Let \(g\colon Z\hookrightarrow T\) be a regular closed immersion of
codimension \(d\), with normal bundle \(N_g\). Absolute purity gives Thom
isomorphisms
\[
\begin{aligned}
\phi_g^r\colon
H_{\acute et}^q(Z;\mu_2^{\otimes r})
&\xrightarrow{\ \sim\ }
H_Z^{q+2d}(T;\mu_2^{\otimes(r+d)}),
\\
\widehat\phi_{g,c}^r\colon
\widehat H_c^q(Z;\mu_2^{\otimes r})
&\xrightarrow{\ \sim\ }
\widehat H_{c,Z}^{q+2d}(T;\mu_2^{\otimes(r+d)}).
\end{aligned}
\]
Write
\[
        s_g:=\phi_g^0(1)
        \in H_Z^{2d}(T;\mu_2^{\otimes d})
\]
for the Thom class. Compatibility of Riou's purity classes with tensor
products identifies the Thom maps with multiplication by \(s_g\); in
particular,
\[
        \widehat\phi_{g,c}^r(x)=x\cdot s_g.
\]
The Thom-class characterization of the Stiefel--Whitney classes gives
\[
        \Sq_Z^b(s_g)
        =
        w_b^{\acute et}(N_g)\cdot s_g;
\]
compare~\cite[\S2.4, (2.11), and the proof of Theorem~2.5]{Benoist}. Therefore,
Theorem~\ref{thm:modified-supported-mixed-cartan} gives
\[
\begin{aligned}
\widehat\Sq_{c,Z}^n\!\left(\widehat\phi_{g,c}^r(x)\right)
&=
\sum_{a+b=n}
\widehat\Sq_c^a(x)\cdot\Sq_Z^b(s_g)
\\
&=
\sum_{a+b=n}
\widehat\Sq_c^a(x)\cdot
w_b^{\acute et}(N_g)\cdot s_g
\\
&=
\sum_{a+b=n}
\widehat\phi_{g,c}^r\!\left(
        \widehat\Sq_c^a(x)\cdot
        w_b^{\acute et}(N_g)
\right).
\end{aligned}
\]
By construction, \(\widehat g_{*,c}\) is the composite of
\(\widehat\phi_{g,c}^r\) with the forget-support
morphism~\eqref{eq:modified-forget-support-spectrum}. Since the forget-support morphism commutes
with stable coefficient operations,
\[
        \widehat\Sq_c^n\!\left(\widehat g_{*,c}x\right)
        =
        \sum_{a+b=n}
        \widehat g_{*,c}\!\left(
                \widehat\Sq_c^a(x)\cdot
                w_b^{\acute et}(N_g)
        \right).
\]
Since \(\nu_g=[N_g]\), this proves the regular-immersion case.

\medskip
\noindent
\emph{Projective-space bundles.}
Let
\[
        p\colon P:=\mathbb P_X^N\longrightarrow X
\]
be the projection, and put
\[
        \lambda:=
        c_1^{\acute et}\!\left(\mathcal O_P(1)\right)
        \in H_{\acute et}^2(P;\mu_2).
\]
Riou's projective-bundle formula gives
\[
        \bigoplus_{m=0}^N
        \mu_2^{\otimes(r-m)}[-2m]
        \xrightarrow{\ \sim\ }
        Rp_*\mu_2^{\otimes r},
\]
whose \(m\)-th component is multiplication by \(\lambda^m\);
see~\cite[Th\'eor\`eme~1.2]{Riou}. Applying
\(\widehat R\Gamma_c(X;-)\) gives
\[
        \bigoplus_{m=0}^N
        \widehat H_c^{q-2m}
        (X;\mu_2^{\otimes(r-m)})
        \xrightarrow{\ \sim\ }
        \widehat H_c^q(P;\mu_2^{\otimes r}),
        \qquad
        (y_m)_m\longmapsto
        \sum_{m=0}^N
        \widehat p_c^*(y_m)\cdot\lambda^m.
\]
It is therefore enough to consider a class
\[
        x=\widehat p_c^*(y)\cdot\lambda^k,
        \qquad 0\leq k\leq N.
\]

Let \(p_*\) denote the ordinary Gysin pushforward. Its trace
normalization gives
\[
        \epsilon_k:=p_*(\lambda^k)
        =
        \begin{cases}
        0,&0\leq k<N,\\
        1,&k=N;
        \end{cases}
\]
see~\cite[D\'efinition~2.4.1]{Riou}. The mixed
projection formula of
Proposition~\ref{prop:modified-gysin-formalism} gives
\[
        \widehat p_{*,c}(x)=y\cdot\epsilon_k.
\]

For \(n\geq0\), put
\[
        R_n:=
        \sum_{a+b=n}
        \widehat p_{*,c}\!\left(
                \widehat\Sq_c^a(x)\cdot
                w_b^{\acute et}(\nu_p)
        \right).
\]
Naturality under pullback along \(p\), together with
Theorem~\ref{thm:arithmetic-boundary-internal-mixed-cartan}, gives
\[
        \widehat\Sq_c^a(x)
        =
        \sum_{u+v=a}
        \widehat p_c^*\!\left(\widehat\Sq_c^u(y)\right)
        \cdot\Sq^v(\lambda^k).
\]
Applying the second projection formula of
Proposition~\ref{prop:modified-gysin-formalism}, we obtain
\[
        R_n
        =
        \sum_{u+v+b=n}
        \widehat\Sq_c^u(y)\cdot
        p_*\!\left(
                \Sq^v(\lambda^k)\cdot
                w_b^{\acute et}(\nu_p)
        \right).
\]
The componentwise form of Benoist's relative Wu formula,
Theorem~\ref{thm:benoist-ordinary-wu-modified-section}, says that
\[
        \Sq^\ell(\epsilon_k)
        =
        \sum_{v+b=\ell}
        p_*\!\left(
                \Sq^v(\lambda^k)\cdot
                w_b^{\acute et}(\nu_p)
        \right).
\]
Consequently,
\[
\begin{aligned}
        R_n
        &=
        \sum_{u+\ell=n}
        \widehat\Sq_c^u(y)\cdot\Sq^\ell(\epsilon_k)
\\
        &=
        \widehat\Sq_c^n(y\cdot\epsilon_k)
\\
        &=
        \widehat\Sq_c^n\!\left(\widehat p_{*,c}x\right).
\end{aligned}
\]
This proves the projective-bundle case, and together with the
composition reduction, completes the proof.
\end{proof}

\subsubsection{The absolute Wu class of
\texorpdfstring{$\mathcal O_{K,S}$}{O(K,S)}}
\label{subsec:absolute-arithmetic-vt}

We now take \(B=\Spec\mathcal O_{K,S}\) and compute its total absolute
Wu class in terms of \(\beta_B\); see
Proposition~\ref{prop:absolute-arithmetic-wu-formula}.

For a scheme \(Z\) on which \(2\) is invertible and \(r\in\mathbf Z\),
consider the coefficient sequence
\begin{equation}
\label{eq:absolute-bockstein-sequence}
0\longrightarrow\mu_2^{\otimes r}
\longrightarrow\mu_4^{\otimes r}
\longrightarrow\mu_2^{\otimes r}
\longrightarrow0,
\qquad
\mu_n^{\otimes0}:=\mathbf Z/n.
\end{equation}
Denote its degree-one connecting operator on ordinary \'etale
cohomology by \(\delta_{Z,r}\). For \(Z=B\), write
\(\delta_r:=\delta_{B,r}\) and denote the corresponding operator on
\(\widehat H_c^*(B;\mu_2^{\otimes r})\), in all integer degrees, by
\(\widehat\delta_r\). In this notation,
\(\beta_Z=\delta_{Z,1}(1_\mu)\).

We shall also use Benoist's comparison between twisted and untwisted
Bocksteins; his class \(\varpi\) is the class denoted \(\beta_Z\) here.

\begin{lemma}[{\cite[Lemma~2.2]{Benoist}}]
\label{lem:benoist-bockstein}
For every scheme \(Z\) on which \(2\) is invertible,
\(r\in\mathbf Z\), and
\(x\in H_{\acute et}^q(Z;\mu_2^{\otimes r})\), one has
\[
        \delta_{Z,0}(x)
        =
        \delta_{Z,r}(x)-r\,\beta_Z\cdot x,
\]
where \(r\) is reduced modulo \(2\) and the coefficient systems are
identified as above.
\end{lemma}

\begin{lemma}
\label{lem:delta-adjoint}
Let \((-)^\vee:=\mathrm{Hom}_{\FF_2}(-,\FF_2)\).
Under the Artin--Verdier duality isomorphisms
\[
H^1(B;\mathbf Z/2)\simeq \widehat H_c^2(B;\mu_2)^\vee
\qquad\text{and}\qquad
H^0(B;\mu_2)\simeq \widehat H_c^3(B;\mathbf Z/2)^\vee,
\]
the map
\[
\delta_1 \colon H^0(B;\mu_2)\longrightarrow H^1(B;\mu_2)\simeq H^1(B;\mathbf Z/2)
\]
is the transpose of
\[
\widehat{\delta}_0 \colon \widehat H_c^2(B;\mu_2)\longrightarrow \widehat H_c^3(B;\mathbf Z/2).
\]
Equivalently, for all \(a\in H^0(B;\mu_2)\) and \(x\in \widehat H_c^2(B;\mu_2)\),
\[
\langle \delta_1(a),x\rangle
=
\langle a,\widehat{\delta}_0(x)\rangle.
\]
\end{lemma}

\begin{proof}
Set \(D(F):=\underline{\mathrm{Hom}}_{B,\mathbf Z}(F,\mathbb G_m)\).
Since \(2\) is invertible on \(B\), for \(\ell\in\{2,4\}\) we have
canonical identifications
\[
D(\mu_\ell)\simeq \mathbf Z/\ell,
\qquad
D(\mathbf Z/\ell)\simeq \mu_\ell,
\]
and the short exact sequence
\[
0\longrightarrow \mu_2\longrightarrow \mu_4\longrightarrow \mu_2\longrightarrow 0
\]
is Cartier dual to
\[
0\longrightarrow \mathbf Z/2\longrightarrow \mathbf Z/4\longrightarrow \mathbf Z/2\longrightarrow 0.
\]

For every finite locally constant sheaf \(F\) of
\(\mathbf Z/4\)-modules on \(B_{\acute et}\), Artin--Verdier duality
gives a functorial perfect pairing
\[
H^i(B;D(F))\times \widehat H_c^{3-i}(B;F)\longrightarrow \mathbf Q/\mathbf Z.
\]
Applying Artin--Verdier duality to this pair of exact sequences gives a
diagram that commutes up to sign; see~\cite[Lemma~3.4]{AhlqvistCarlsonPunctured}
and, more generally,~\cite[Lemma~4.3]{DemarcheHarari}:
\[
\begin{tikzcd}
H^0(B;\mu_2) \arrow[r,"\delta_1"] \arrow[d,"\sim"'] &
H^1(B;\mu_2)\arrow[d,"\sim"]\\
\widehat H_c^3(B;\mathbf Z/2)^\vee \arrow[r,"(\widehat\delta_0)^\vee"'] &
\widehat H_c^2(B;\mathbf{Z}/2)^\vee.
\end{tikzcd}
\]
Since we are working modulo $2$, there is no sign ambiguity, hence the diagram commutes.
\end{proof}

\begin{lemma}[Real-place boundary and duality]
\label{lem:s-integer-real-boundary-duality}
Let \(\Omega_{\mathbf R}(K)\) denote the set of real places of \(K\),
and let
\[
        \beta_\nu\in H_{\acute et}^1(\mathbf R;\FF_2)
\]
be the standard generator. There are natural identifications
\begin{equation}
\label{eq:s-integer-tate-boundary-ring}
        \widehat H_\infty^*(B;\FF_2)
        \cong
        \prod_{\nu\in\Omega_{\mathbf R}(K)}
        \FF_2[\beta_\nu^{\pm1}]
\end{equation}
and, for \(n\geq3\),
\begin{equation}
\label{eq:s-integer-high-degree-real-restriction}
        H_{\acute et}^n(B;\FF_2)
        \xrightarrow{\ \sim\ }
        \prod_{\nu\in\Omega_{\mathbf R}(K)}
        \FF_2\beta_\nu^n.
\end{equation}
Moreover, the Tate-boundary morphism induces an isomorphism
\begin{equation}
\label{eq:s-integer-tate-boundary-isomorphism}
        \partial_\infty\colon
        \widehat H_\infty^{2-n}(B;\mu_2)
        \xrightarrow{\ \sim\ }
        \widehat H_c^{3-n}(B;\mu_2)
        \qquad(n\geq3).
\end{equation}
For
\[
        a\in H_{\acute et}^n(B;\FF_2),
        \qquad
        y=(y_\nu)_\nu
        \in\widehat H_\infty^{2-n}(B;\mu_2),
\]
one has
\begin{equation}
\label{eq:s-integer-local-global-duality}
        \int_B a\cdot\partial_\infty(y)
        =
        \sum_{\nu\in\Omega_{\mathbf R}(K)}
        \operatorname{tr}_\nu\!\left(
                \operatorname{res}_\nu(a)\cdot y_\nu
        \right),
\end{equation}
where \(\operatorname{res}_\nu\) denotes restriction along the morphism
\(\Spec\mathbf R\to B\) corresponding to the real place \(\nu\), and
\[
        \operatorname{tr}_\nu\colon
        \widehat H_{\acute et}^2(\mathbf R;\mu_2)
        \longrightarrow\FF_2
\]
is normalized by \(\operatorname{tr}_\nu(\beta_\nu^2)=1\).
The morphism \(\partial_\infty\) commutes with the stable Steenrod
operations and is \(H_{\acute et}^*(B;\FF_2)\)-linear via restriction.
\end{lemma}

\begin{proof}
The geometric fiber \(B_{\mathbf C}\) is the finite set of complex
embeddings of \(K\). Complex conjugation fixes the real embeddings and
permutes the remaining embeddings in free two-element orbits. Tate
cohomology vanishes on the corresponding induced spectra;
see~\cite[Lemma~I.3.8(i)]{NikolausScholze2018}.
Thus,~\eqref{eq:s-integer-tate-boundary-ring} follows from
Proposition~\ref{prop:c2-point-tate-steenrod-calculation}.

Taking \(S_f\) in the notation of Ahlqvist--Carlson to be the set of
finite places inverted in \(\mathcal O_{K,S}\),
the description~\eqref{eq:s-integer-high-degree-real-restriction} follows
from~\cite[Theorem~2.5 and the final paragraph of \S3]{AhlqvistCarlsonPunctured}.

Compatibility of the compact-support localization sequence with the
global and local trace pairings identifies \(\partial_\infty\) with the
transpose of the restriction map
in~\eqref{eq:s-integer-high-degree-real-restriction};
see~\cite[Chapter II, \S3, p.176 and proof of Corollary 3.4]{MILNE_ARITHMETIC}
and~\cite[Theorem~4.3, especially (4.1)]{GeisserSchmidt}. Since the global
and local pairings are perfect, this
proves~\eqref{eq:s-integer-tate-boundary-isomorphism}
and~\eqref{eq:s-integer-local-global-duality}. Under
\[
        \tfrac12\mathbf Z/\mathbf Z\simeq\FF_2,
\]
the local invariant of \(\beta_\nu^2\) is \(1\), giving the stated
normalization.

Compatibility with Steenrod operations and module structures follows
from Theorem~\ref{thm:canonical-arithmetic-boundary-operations} and
Proposition~\ref{prop:multiplicative-arithmetic-boundary-diagram}.
\end{proof}

Write
\[
        v_B=\sum_{n\geq0}v_{B,n},
        \qquad
        v_{B,n}\in H_{\acute et}^n(B;\FF_2).
\]

\begin{proposition}[Absolute arithmetic Wu class]
\label{prop:absolute-arithmetic-wu-formula}
For every \(n\geq0\),
\begin{equation}
\label{eq:absolute-arithmetic-wu-components}
        v_{B,n}
        =
        \binom{2-n}{n}\beta_B^n,
\end{equation}
where the generalized integral binomial coefficient is reduced modulo
\(2\). Equivalently, in the forward-degree completion,
\[
        v_B
        =
        \sum_{n\geq0}
        \binom{2-n}{n}\beta_B^n
        =
        1+\beta_B+\beta_B^3+\beta_B^4+\beta_B^5
        +\beta_B^8+\beta_B^9+\cdots .
\]
In particular,
\[
        v_{B,0}=1,\qquad
        v_{B,1}=\beta_B,\qquad
        v_{B,2}=0.
\]
\end{proposition}

\begin{proof}
Write
\[
        \langle a,x\rangle_B:=\int_B a\cdot x.
\]
By Definition~\ref{def:absolute-wu-class} and the perfectness of
Artin--Verdier duality, \(v_{B,n}\) is characterized by
\begin{equation}
\label{eq:absolute-arithmetic-wu-component-characterization}
        \langle v_{B,n},x\rangle_B
        =
        \int_B\widehat\Sq_c^n(x),
        \qquad
        x\in\widehat H_c^{3-n}(B;\mu_2).
\end{equation}

For \(n=0\), the identity \(\widehat\Sq_c^0=\mathrm{id}\) gives
\(v_{B,0}=1\). For \(n=1\), the equality
\(\widehat\Sq_c^1=\widehat\delta_0\) and
Lemma~\ref{lem:delta-adjoint} give
\[
\begin{aligned}
        \int_B\widehat\Sq_c^1(x)
        &=
        \left\langle1_\mu,\widehat\delta_0(x)\right\rangle_B
\\
        &=
        \left\langle\delta_1(1_\mu),x\right\rangle_B
        =
        \left\langle\beta_B,x\right\rangle_B.
\end{aligned}
\]
Thus, \(v_{B,1}=\beta_B\).

Let \(n=2\). Put
\[
        M_B:=R\Gamma_c(B_{\mathbf C};\FF_2),
\]
regarded as a \(G_{\mathbf R}\)-spectrum. The boundary diagram and the
cofiber sequence defining the Tate construction give a fiber sequence
\[
        R\Gamma_c(B;\FF_2)
        \longrightarrow
        \widehat R\Gamma_c(B;\FF_2)
        \longrightarrow
        (M_B)_{hG_{\mathbf R}}.
\]
Since \(B_{\mathbf C}\) is a finite disjoint union of points, \(M_B\),
and hence \((M_B)_{hG_{\mathbf R}}\), is connective. It follows that
\[
        H_c^1(B;\FF_2)
        \longrightarrow
        \widehat H_c^1(B;\FF_2)
\]
is surjective. Therefore, every
\(x\in\widehat H_c^1(B;\FF_2)\) is the image of some
\(\widetilde x\in H_c^1(B;\FF_2)\). Compatibility with coefficient
operations and ordinary instability give
\[
        \widehat\Sq_c^2(x)
        =
        \gamma_{B,\FF_2}\!\left(\Sq_c^2(\widetilde x)\right)
        =
        0.
\]
Equation~\eqref{eq:absolute-arithmetic-wu-component-characterization}
and perfectness therefore imply \(v_{B,2}=0\).

Finally, suppose that \(n\geq3\). For each
\(\nu\in\Omega_{\mathbf R}(K)\), put
\[
        e_{\nu,n}
        :=
        \partial_\infty
        \bigl(0,\ldots,\beta_\nu^{\,2-n},\ldots,0\bigr)
        \in\widehat H_c^{3-n}(B;\mu_2).
\]
By Lemma~\ref{lem:s-integer-real-boundary-duality}, these classes form
the basis dual to the real-place basis
in~\eqref{eq:s-integer-high-degree-real-restriction}. Since
\(\partial_\infty\) commutes with Steenrod operations,
Proposition~\ref{prop:c2-point-tate-steenrod-calculation} gives
\[
\begin{aligned}
        \int_B\widehat\Sq_c^n(e_{\nu,n})
        &=
        \operatorname{tr}_\nu
        \bigl(\Sq^n(\beta_\nu^{\,2-n})\bigr)
\\
        &=
        \binom{2-n}{n}
        \operatorname{tr}_\nu(\beta_\nu^2)
        =
        \binom{2-n}{n}.
\end{aligned}
\]
Naturality of the Bockstein gives
\[
        \operatorname{res}_\nu(\beta_B)=\beta_\nu,
\]
and hence
\[
        \langle\beta_B^n,e_{\nu,n}\rangle_B=1.
\]
Evaluating~\eqref{eq:absolute-arithmetic-wu-component-characterization} on the
dual basis \(\{e_{\nu,n}\}_\nu\) proves
\[
        v_{B,n}
        =
        \binom{2-n}{n}\beta_B^n.
\]
When \(K\) has no real places, both sides vanish for \(n\geq3\)
by~\eqref{eq:s-integer-high-degree-real-restriction}.
\end{proof}

\begin{remark}[Dependence on \(S\)]
\label{rem:dependence-on-S}
The Kummer sequence gives an exact sequence
\[
0\longrightarrow
\mathcal O_{K,S}^{\times}/(\mathcal O_{K,S}^{\times})^2
\longrightarrow
H_{\acute et}^1(B;\mu_2)
\longrightarrow
\operatorname{Pic}(B)[2]
\longrightarrow0;
\]
see~\cite[Tag~03PK]{STACKS-PROJECT}. Thus, \(\beta_B\) belongs to the
unit part and is represented there by \(-1\).

If \(S\subseteq S'\), put
\[
        B':=\Spec\mathcal O_{K,S'}
\]
and let \(j\colon B'\hookrightarrow B\) be the induced open immersion.
Naturality gives
\[
        j^*\beta_B=\beta_{B'}.
\]
Consequently, Proposition~\ref{prop:absolute-arithmetic-wu-formula}
gives
\[
        j^*v_B=v_{B'}.
\]
Thus enlarging \(S\) changes the ambient cohomology ring but not the
universal expression for the absolute Wu class. In particular,
\(\beta_B=0\) if and only if \(-1\) is a square in \(K\), equivalently
if and only if \(i\in K\).
\end{remark}

\begin{remark}[Comparison with Feng's local class]
\label{rem:comparison-with-feng}
Let \(\mathfrak p\notin S\) be a finite prime of \(K\), let
\(D_{\mathfrak p}\) be a decomposition group, and let
\(I_{\mathfrak p}\subset D_{\mathfrak p}\) be its inertia subgroup. Since
\(\beta_B\in H^1(B;\mu_2)\), its restriction to \(D_{\mathfrak p}\) is
unramified and hence factors through
\[
D_{\mathfrak p}/I_{\mathfrak p}\simeq
\mathrm{Gal}(\overline{\kappa(\mathfrak p)}/\kappa(\mathfrak p)).
\]
The resulting class in
\[
H^1(\kappa(\mathfrak p);\mu_2)
\]
is precisely Feng's class \(\alpha_{\kappa(\mathfrak p)}\), since both are
obtained by applying the connecting morphism for
\[
0\to \mu_2\to \mu_4\to \mu_2\to 0
\]
to the section \(1_\mu=-1\). By~\cite[Remark~5.9]{FengEtaleSteenrod}, this class vanishes if
and only if
\[
\#\kappa(\mathfrak p)\equiv 1 \pmod 4.
\]
When \(\kappa(\mathfrak p)=\mathbb{F}_p\), this is equivalent to
\[
\left(\frac{-1}{p}\right)=1,
\]
that is, to \(p\equiv 1 \pmod 4\); equivalently, the local class is nonzero
exactly when \(p\equiv 3 \pmod 4\).
\end{remark}

\paragraph{Arithmetic Bockstein powers}
Although \(v_{B,2}=0\), the class \(\beta_B^2\) need not vanish. We
record its arithmetic interpretation and then determine when the higher
powers of \(\beta_B\) are nonzero.

Let
\[
        j\colon\eta=\Spec K\hookrightarrow B
\]
be the generic point, and write
\[
        j^*(\beta_B^2)
        \in H_{\acute et}^2(K;\FF_2)
\]
for the restriction of \(\beta_B^2\).

\begin{proposition}[The second Bockstein power;
see~{\cite[Proposition~4.7.1 and \S2.5]{GilleSzamuely}}]
\label{prop:beta-square-generic-point}
Under the coefficient identifications fixed above and the Kummer
isomorphism,
\[
        H_{\acute et}^2(K;\FF_2)
        \simeq
        H_{\acute et}^2(K;\mu_2^{\otimes2})
        \simeq
        \Br(K)[2],
\]
one has
\[
        j^*(\beta_B^2)=[(-1,-1)],
\]
where \((-1,-1)\) is the quaternion algebra
\[
        K\langle u,v\rangle/
        (u^2+1,\ v^2+1,\ uv+vu).
\]
\end{proposition}

\begin{proof}
By Remark~\ref{lem:bockstein-equals-kummer},
\(j^*\beta_B=\kappa_K(-1)\). Hence, \(j^*(\beta_B^2)\) is the
Kummer symbol
\[
        \kappa_K(-1)\cdot\kappa_K(-1),
\]
which corresponds to the Brauer class of \((-1,-1)\).
\end{proof}

\begin{corollary}
\label{cor:beta-square-generic-point-splitting}
With notation as above,
\[
(\beta_B^2)|_\eta=0
\iff
(-1,-1)\ \text{is split over }K
\iff
-1\in N_{K(i)/K}\bigl(K(i)^\times\bigr).
\]
\end{corollary}

\begin{proof}
By Proposition~\ref{prop:beta-square-generic-point},
\((\beta_B^2)|_\eta\) is the Brauer class of \(({-1},{-1})\), so it vanishes
if and only if \(({-1},{-1})\) is split. By~\cite[Prop.~1.1.7]{GilleSzamuely},
a quaternion algebra \((a,b)\) is split if
and only if \(b\) is a norm from \(K(\sqrt a)\). Applying this with
\(a=b=-1\) gives the last equivalence.
\end{proof}

\begin{remark}
\label{rem:beta-square-Q}
For \(B=\Spec\mathbf Z[1/2]\), one has
\[
        j^*(\beta_B^2)\neq0,
\]
and hence \(\beta_B^2\neq0\). Indeed, its generic restriction is the
class of \((-1,-1)\), which remains nonsplit over \(\mathbf R\), since
\[
        -1\notin
        N_{\mathbf C/\mathbf R}(\mathbf C^\times)
        =
        \mathbf R_{>0};
\]
see~\cite[Proposition~1.1.7]{GilleSzamuely}.
\end{remark}

\begin{proposition}[Higher Bockstein powers and real places]
\label{prop:beta-cube-real-place}
For every \(k\geq3\),
\[
        \beta_B^k\neq0
        \qquad\Longleftrightarrow\qquad
        K\text{ admits a real embedding}.
\]
Consequently,
\[
        v_{B,k}\neq0
        \qquad\Longleftrightarrow\qquad
        \binom{2-k}{k}\equiv1\pmod2
        \quad\text{and}\quad
        K\text{ admits a real embedding}.
\]
\end{proposition}

\begin{proof}
By~\eqref{eq:s-integer-high-degree-real-restriction},
\[
        H_{\acute et}^k(B;\FF_2)
        \xrightarrow{\ \sim\ }
        \prod_{\nu\in\Omega_{\mathbf R}(K)}
        H_{\acute et}^k(\mathbf R;\FF_2).
\]
Naturality of the Bockstein identifies the restriction of \(\beta_B\)
at \(\nu\) with the generator
\[
        \beta_\nu\in H_{\acute et}^1(\mathbf R;\FF_2).
\]
Since
\[
        H_{\acute et}^*(\mathbf R;\FF_2)
        \cong\FF_2[\beta_\nu],
\]
the image of \(\beta_B^k\) is
\[
        (\beta_\nu^k)_{\nu\in\Omega_{\mathbf R}(K)}.
\]
This tuple is nonzero precisely when \(\Omega_{\mathbf R}(K)\) is
nonempty. The assertion concerning \(v_{B,k}\) then follows from
Proposition~\ref{prop:absolute-arithmetic-wu-formula}.
\end{proof}

\section{A generalization of Hecke's theorem}
\label{chap:generalised-hecke}
In~\S\ref{sec:generalized-hecke-theorem}, we establish the
generalized Hecke theorem and describe its archimedean defects.
In~\S\ref{sec:applications_and_examples}, we discuss
specializations, examples, and related results.

\subsection{The generalized Hecke theorem}
\label{sec:generalized-hecke-theorem}
The main result is the generalized Hecke theorem,
Theorem~\ref{thm:gen_hecke}, which gives congruences for the
universal Wu polynomials modulo the archimedean comparison ideal.
We construct these polynomials in~\S\ref{subsubsec:etale-2td-polynomials},
then introduce the ideal and prove the theorem
in~\S\ref{subsubsec:hecke-comparison}.
In~\S\ref{subsubsec:real-place-wu}, we describe the high-degree
arithmetic Wu classes through the real loci and establish
nonvanishing results.
\subsubsection{\'Etale \texorpdfstring{$2$}{2}-Todd classes
and universal polynomials}
\label{subsubsec:etale-2td-polynomials}

Let \(B\) be one of the bases considered in
\S\ref{sec:absolute-relative-wu-formulas}. Define
\[
        V_B(t):=
        \begin{cases}
        1,
        & B=\Spec(k),\quad k\text{ finite},\\[2pt]
        \displaystyle\sum_{n\geq0}\binom{2-n}{n}t^n,
        & B=\Spec(\mathcal O_{K,S}),
        \end{cases}
        \qquad V_B(t)\in\FF_2[[t]],
\]
where the generalized integral binomial coefficients are reduced
modulo \(2\). For a flat projective morphism \(f\colon X\to B\), pure of
relative dimension \(d\), with \(X\) regular,
Theorems~\ref{thm:absolute-wu-formula}
and~\ref{thm:absolute-arithmetic-wu-class} give
\begin{equation}
\label{eq:explicit_wu}
        v_X
        =
        \Sq^{-1}\!\bigl(w^{\acute et}(\tau_f)\bigr)
        \cdot V_B(\beta_X).
\end{equation}

Following Hirzebruch, we introduce the \'etale \(2\)-Todd
class to express the first factor in~\eqref{eq:explicit_wu}
in terms of Chern classes. Its generating series also accounts
for the base factor; see
Corollary~\ref{cor:2td-universal-polynomials}.

\begin{definition}\label{def:2td-class}
Let $X$ be a regular noetherian $B$-scheme and let $\xi\in K^0(X)$. The total
\emph{\'{e}tale $2$-Todd class} of $\xi$ is
\[
2td(\xi):=\Sq^{-1}\bigl(w^{\mathrm{\acute et}}(\xi)\bigr)
\in \left(H_{\acute et}^*(X;\FF_2)^{\wedge,+}\right)^\times.
\]
We write
\[
2td(\xi)=\sum_{m\ge 0} 2td(\xi)_m,
\qquad
2td(\xi)_m\in H^m_{\mathrm{\acute et}}(X;\mathbb F_2).
\]
\end{definition}

\begin{proposition}\label{prop:2td-basic}
For every morphism $g\colon Y\to X$ of regular noetherian $B$-schemes and every
$\xi,\eta\in K^0(X)$, one has
\[
g^*2td(\xi)=2td(g^*\xi),
\qquad
2td(\xi+\eta)=2td(\xi)\cdot 2td(\eta).
\]
In particular, if
\[
0\longrightarrow \xi'\longrightarrow \xi\longrightarrow \xi''\longrightarrow 0
\]
is a short exact sequence of vector bundles on $X$, then
\[
2td(\xi)=2td(\xi')\cdot 2td(\xi'').
\]
\end{proposition}

\begin{proof}
This is immediate from the corresponding properties of
$w^{\mathrm{\acute et}}$ and the fact that $\Sq^{-1}$ is a continuous ring
automorphism of $H_{\acute et}^*(X;\FF_2)^{\wedge,+}$.
\end{proof}

To make the class $2td$ explicit, we introduce the associated generating
series. Let
\[
Q(z):=\frac{z}{1-e^{-z}}
=\sum_{m\ge 0}\frac{B_m}{m!}z^m
\in \mathbf Q[[z]]
\]
be the Todd power series, with the convention \(B_1=1/2\).
The following lemma is elementary; see~\cite[Remark following Lemma~1.1]{HIRZEBRUCH}.

\begin{lemma}\label{lem:T-series}
The power series $Q(2z)$ is $2$-integral, and its reduction modulo $2$,
\begin{equation}\label{eq:T-series}
T(z):=Q(2z)\bmod 2
\in \mathbb{F}_2[[z]],
\end{equation}
has the form
\begin{equation}\label{eq:T-expansion}
T(z)=1+\sum_{r\ge 0} z^{2^r}.
\end{equation}
\end{lemma}

\begin{proof}
Write
\[
Q(2z)=\sum_{m\ge 0} a_m z^m,
\qquad
a_m=\frac{2^m B_m}{m!}.
\]
We have $a_0=1$ and $a_1=2B_1=1$. If $m>1$ is odd, then $B_m=0$, so
$a_m=0$. It therefore remains to consider the coefficients $a_{2n}$ with
$n\ge 1$.

By the von Staudt--Clausen theorem (see, for example,~\cite[Ch.~15]{IrelandRosen}),
one has $v_2(B_{2n})=-1$. Hence,
\[
v_2(a_{2n})=2n-1-v_2((2n)!).
\]
Let $s_2(N)$ denote the sum of the binary digits of $N$. By Legendre's
formula,
\[
v_2((2n)!)=2n-s_2(2n)=2n-s_2(n),
\]
and therefore
\[
v_2(a_{2n})=s_2(n)-1\ge 0.
\]
This proves that $Q(2z)$ is $2$-integral.

Moreover, \(a_{2n}\) is odd precisely when \(s_2(n)=1\), or
equivalently when \(n\) is a power of \(2\). Thus, the reduction
of \(Q(2z)\) modulo \(2\) is
\[
1+\sum_{r\ge 0} z^{2^r},
\]
as claimed.
\end{proof}

\begin{lemma}
\label{lem:T-power-coefficients}
For \(r\in\mathbf Z\) and \(n\geq0\), one has
\[
        [t^n]T(t)^r
        =
        \binom{r-n-1}{n}
        \quad\text{in }\FF_2,
\]
where \([t^n]\) denotes coefficient extraction and generalized
integral binomial coefficients are reduced modulo \(2\).
In particular, recalling \(\epsilon_B=\dim B\in\{0,1\}\), one has
\begin{equation}
\label{eq:base-wu-T-power}
        V_B(t)=T(t)^{3\epsilon_B}.
\end{equation}
\end{lemma}

\begin{proof}
By~\eqref{eq:T-expansion}, \(T(u+u^2)=1+u\).
The substitution \(t=u+u^2\), with \(dt=du\), gives
\[
        [t^n]T(t)^r
        =
        \operatorname{Res}_{u=0}
        \frac{(1+u)^{r-n-1}}{u^{n+1}}\,du
        =
        \binom{r-n-1}{n}.
\]
The last assertion follows by taking \(r=3\) in the arithmetic
case; in the finite-field case both sides are \(1\).
\end{proof}

\begin{remark}
In the complex topological setting, Hirzebruch proved that the multiplicative
class determined by the power series $T(z)$ is the inverse image under
$\Sq$ of the Stiefel--Whitney class; see~\cite[Thm.~5.1]{HIRZEBRUCH}.
The theorem below is the corresponding
arithmetic statement.
\end{remark}

\begin{lemma}\label{lem:2td-line-bundle}
Let $L$ be a line bundle on a regular noetherian $B$-scheme $X$, and put
\[
\bar x:=\overline{c}_1(L)\in H^2_{\mathrm{\acute et}}(X;\mathbb F_2).
\]
Define
\[
\Theta_\beta(x):=T(\beta)\cdot T\!\bigl(T(\beta)^{-2}x\bigr)
\in \mathbb{F}_2[[\beta,x]],
\]
where $T(z)=1+\sum_{r\ge 0} z^{2^r}$ and $T(\beta)^{-1}$ denotes the
inverse of $T(\beta)$ in $\mathbb{F}_2[[\beta]]$.
Then
\[
2td(L)=\Theta_{\beta_X}(\bar x).
\]
\end{lemma}

\begin{proof}
All formal power series below are evaluated degreewise in the completed graded ring $\prod_{m\ge 0} H^m_{\mathrm{\acute et}}(X;\mathbb F_2)$. Write
\[
\beta:=\beta_X,\qquad x:=\bar x.
\]
Since $T(z)$ has constant term $1$, the class $T(\beta)$ is invertible.

First note that, by~\eqref{eq:T-expansion}, one has the formal identity
\[
T(z+z^2)
=
1+\sum_{r\ge 0}(z+z^2)^{2^r}
=
1+\sum_{r\ge 0}\bigl(z^{2^r}+z^{2^{r+1}}\bigr)
=
1+z.
\]
Since $\deg(\beta)=1$, one has $\Sq(\beta)=\beta+\beta^2$, hence
\[
\Sq\bigl(T(\beta)\bigr)=T(\Sq(\beta))=T(\beta+\beta^2)=1+\beta.
\]

Next, because $\deg(x)=2$, the instability relations give
\[
\Sq(x)=x+\Sq^1(x)+x^2.
\]
By Lemma~\ref{lem:benoist-bockstein},
\[
\Sq^1(a)=\delta_1(a)+\beta\cdot a,
\]
where $\delta_1$ is the Bockstein attached to the exact sequence
\[
0\longrightarrow \mu_2\longrightarrow \mu_4\longrightarrow \mu_2\longrightarrow 0.
\]
Now $x$ is the reduction of
$c_1(L)\in H^2_{\mathrm{\acute et}}(X;\mathbf Z_2(1))$, so its reduction
modulo $2$ lifts to
$H^2_{\mathrm{\acute et}}(X;\mathbf Z/4(1))
\cong H^2_{\mathrm{\acute et}}(X;\mu_4)$.
Therefore, $\delta_1(x)=0$, and so
\[
\Sq(x)=x+\beta x+x^2=x(1+\beta+x).
\]

Since
\[
2td(L)=\Sq^{-1}\bigl(w^{\mathrm{\acute et}}(L)\bigr),
\]
it is enough to prove that
\[
\Sq\bigl(\Theta_\beta(x)\bigr)=w^{\mathrm{\acute et}}(L).
\]
By Theorem~\ref{thm:etale-sw-chern},
\[
w^{\mathrm{\acute et}}(L)=1+\beta+x.
\]

Now
\begin{align*}
\Sq\bigl(\Theta_\beta(x)\bigr)
&=
\Sq\bigl(T(\beta)\bigr)\cdot \Sq\!\left(T\!\bigl(T(\beta)^{-2}x\bigr)\right) \\
&=
(1+\beta)\cdot T\!\left(\Sq\bigl(T(\beta)\bigr)^{-2}\Sq(x)\right) \\
&=
(1+\beta)\cdot
T\!\left(\frac{x(1+\beta+x)}{(1+\beta)^2}\right).
\end{align*}
Since $1+\beta$ is invertible in $H_{\acute et}^*(X;\FF_2)^{\wedge,+}$,
we may set
\[
u:=\frac{x}{1+\beta}.
\]
Then
\[
\frac{x(1+\beta+x)}{(1+\beta)^2}=u+u^2.
\]
Using the identity $T(u+u^2)=1+u$, we obtain
\[
\Sq\bigl(\Theta_\beta(x)\bigr)
=
(1+\beta)\bigl(1+u\bigr)
=
1+\beta+x.
\]
This is exactly $w^{\mathrm{\acute et}}(L)$, and the result follows.
\end{proof}

\begin{remark}[The untwisted case]\label{rem:untwisted-case}
Assume that $\beta_X=0$. This happens, for instance, when $B=\Spec \mathbb{F}_q$ with $q\equiv 1\pmod 4$, or when $B=\Spec \mathcal O_{K,S}$ and $i\in K$. In this case
\[
\Theta_{\beta_X}(x)=T(x)=Q(2x)\bmod 2,
\]
and there is no difference between our arithmetic $2td$-generating function, and that of Hirzebruch.
\end{remark}

\begin{theorem}[cf.~Hirzebruch~{\cite[Thm.~5.1]{HIRZEBRUCH}}]
\label{thm:etale-2td-splitting}
Let $X$ be a regular noetherian $B$-scheme and let $\xi$ be a vector bundle of rank $r$
on $X$. Let
\[
\pi\colon \mathrm{Fl}(\xi)\longrightarrow X
\]
be the complete flag bundle, and let $L_1,\dots,L_r$ be the successive
quotient line bundles of $\pi^*\xi$. Put
\[
x_i:=\overline{c}_1(L_i)\in
H^2_{\mathrm{\acute et}}(\mathrm{Fl}(\xi);\mathbb F_2).
\]
Then, in $H_{\acute et}^*(\mathrm{Fl}(\xi);\FF_2)^{\wedge,+}$, one has
\begin{equation}\label{eq:2td-splitting}
\pi^*2td(\xi)=\prod_{i=1}^r \Theta_{\pi^*\beta_X}(x_i).
\end{equation}
\end{theorem}

\begin{proof}
By Proposition~\ref{prop:2td-basic} and Grothendieck's splitting principle (see Remark~\ref{rem:etale-sw-splitting}),
\[
\pi^*2td(\xi)=2td(\pi^*\xi)=\prod_{i=1}^r 2td(L_i).
\]
Lemma~\ref{lem:2td-line-bundle} identifies each factor with
$\Theta_{\pi^*\beta_X}(x_i)$, which proves~\eqref{eq:2td-splitting}.
\end{proof}

\begin{corollary}
\label{cor:2td-universal-polynomials}
For \(m\geq0\) and \(r\in\mathbf Z\), there are canonical polynomials
\[
        P_{m,r}\in
        \FF_2[\beta,c_1,\dots,c_{\lfloor m/2\rfloor}],
\]
homogeneous of weighted degree \(m\) for
\(\deg(\beta)=1\) and \(\deg(c_i)=2i\), such that
\[
        2td(\xi)_m
        =
        P_{m,r}\bigl(\beta_X,\overline c_1(\xi),\dots\bigr)
\]
for every virtual bundle \(\xi\in K^0(X)\) of constant rank \(r\)
on a regular Noetherian \(B\)-scheme \(X\).

If, moreover, \(f\colon X\to B\) is flat and projective of pure
relative dimension \(d\), then
\begin{align}
\label{eq:wu-augmented-tangent}
        v_X
        &=
        2td\!\left(\tau_f+\mathcal O_X^{\oplus3\epsilon_B}\right),
        \\
        v_{X,m}
        &=
        P_{m,d+3\epsilon_B}
        \bigl(\beta_X,\overline c_1(\tau_f),\dots\bigr).
        \notag
\end{align}
\end{corollary}

\begin{proof}
Normalizing the product in~\eqref{eq:2td-splitting} gives
\[
        \frac{\prod_{i=1}^{s}\Theta_\beta(x_i)}{T(\beta)^s}
        =
        \prod_{i=1}^{s}T\!\bigl(T(\beta)^{-2}x_i\bigr).
\]
By symmetry, its homogeneous components are polynomials in
\(\beta\) and the Chern classes. Since \(T(0)=1\), adjoining
a zero root leaves this product unchanged.
The resulting Chern-class expressions are therefore compatible
across ranks and extend multiplicatively to virtual bundles by
division, using
\(\overline c(E-F)=\overline c(E)/\overline c(F)\).

By Proposition~\ref{prop:2td-basic}, this normalized expression
evaluates on a virtual bundle \(\xi\) of rank \(r\) as
\(T(\beta_X)^{-r}2td(\xi)\).
Multiplying by \(T(\beta)^r\) and taking weighted degree \(m\)
defines \(P_{m,r}\); only \(c_i\) with \(2i\leq m\) occur.

Lemma~\ref{lem:2td-line-bundle} gives
\(2td(\mathcal O_X)=T(\beta_X)\).
Thus~\eqref{eq:explicit_wu}, \eqref{eq:base-wu-T-power},
and multiplicativity give~\eqref{eq:wu-augmented-tangent}.
The augmented virtual bundle has rank \(d+3\epsilon_B\) and
the same Chern classes as \(\tau_f\), proving the final assertion.
\end{proof}

\subsubsection{The archimedean comparison ideal and the generalized Hecke theorem}
\label{subsubsec:hecke-comparison}
Fix a flat projective morphism \(f\colon X\to B\), pure of relative
dimension \(d\), with \(X\) regular, and retain the notation
\(D_X,N_X\) of~\S\ref{sec:absolute-relative-wu-formulas}.
Over finite fields and rings of \(S\)-integers in totally imaginary
number fields, instability and duality give
\[
        v_{X,m}=0\qquad(2m>N_X).
\]
In the latter case, the Tate boundary vanishes, so
Theorem~\ref{thm:instability-modulo-tate-boundary} supplies
the required instability.
Through the universal Wu polynomials, these vanishings give
generalized Hecke relations.
Theorem~\ref{thm:gen_hecke}, stated for both choices of \(B\),
extends these relations to arbitrary number fields as congruences
modulo the archimedean comparison ideal.

For the constructions and auxiliary results below, assume
\(B=\Spec(\mathcal O_{K,S})\).
For \(q\in\mathbf Z\), let
\[
        \gamma_X^q\colon
        H_c^q(X;\mu_2^{\otimes D_X})
        \longrightarrow
        \widehat H_c^q(X;\mu_2^{\otimes D_X})
\]
denote the ordinary-to-modified comparison.

\begin{definition}[Archimedean comparison ideal]
\label{def:archimedean-comparison-ideal}
For \(m\geq0\), set
\[
        J_\infty^m(X)
        :=
        \bigl(\operatorname{im}\gamma_X^{N_X-m}\bigr)^\perp
        \subset H_{\acute et}^m(X;\FF_2),
        \qquad
        J_\infty(X):=\bigoplus_{m\geq0}J_\infty^m(X),
\]
where orthogonality is taken with respect to the
Artin--Verdier pairing~\eqref{eq:av}.
\end{definition}

\begin{proposition}[Archimedean description]
\label{prop:archimedean-comparison-image}
For each real place \(\nu\in\Omega_{\mathbf R}(K)\), put
\[
        X_\nu:=X\times_B\Spec(\mathbf R),
        \qquad G_\nu:=\Gal(\mathbf C/\mathbf R).
\]
For every \(m\geq0\), there is a canonical map of degree \(3\)
\[
        a_\infty\colon
        \bigoplus_{\nu\in\Omega_{\mathbf R}(K)}
        H_{G_\nu}^{m-3}(X_\nu(\mathbf C);\FF_2)
        \longrightarrow H_{\acute et}^m(X;\FF_2)
\]
whose image is \(J_\infty^m(X)\).
Here \(H_{G_\nu}^*\) denotes ordinary Borel equivariant cohomology
for complex conjugation.
\end{proposition}

\begin{proof}
Put \(M=R\Gamma_c(X_{\mathbf C};\FF_2)\).
The comparison fiber sequence from the proof of
Proposition~\ref{prop:ordinary-tate-boundary-bounded-range-comparison}
is
\[
        R\Gamma_c(X;\FF_2)
        \longrightarrow
        \widehat R\Gamma_c(X;\FF_2)
        \xrightarrow{\ \rho_X\ }
        M_{hG_{\mathbf R}}.
\]
Equivariant Poincar\'e duality on the smooth proper complex fibers gives
\[
        H^q(M_{hG_{\mathbf R}})^\vee
        \simeq H_\infty^{2d-q}(X;\FF_2),
        \qquad
        (-)^\vee:=\operatorname{Hom}_{\FF_2}(-,\FF_2);
\]
see~\cite[\S\S1.1.4--1.1.5]{BenoistWittenbergRealHodgeI}.
Transposing \(\rho_X^{N_X-m}\) under this identification and
Artin--Verdier duality, and using \(N_X=2d+3\), gives a map
\[
        H_\infty^{m-3}(X;\FF_2)
        \longrightarrow H_{\acute et}^m(X;\FF_2)
\]
whose image is \(J_\infty^m(X)\) by exactness.

Its source decomposes as
\[
\begin{aligned}
        H_\infty^{m-3}(X;\FF_2)
        &\simeq
        \bigoplus_{\nu\in\Omega_{\mathbf R}(K)}
        H_{G_\nu}^{m-3}(X_\nu(\mathbf C);\FF_2)
        \\
        &\qquad\oplus
        \bigoplus_{w\in\Omega_{\mathbf C}(K)}
        H^{m-3}(X_w(\mathbf C);\FF_2),
\end{aligned}
\]
where \(\Omega_{\mathbf C}(K)\) denotes the complex places and
\(X_w:=X\times_B\Spec(\mathbf C)\), using one representative
embedding for each \(w\).

The transposed map vanishes on the complex-place summands.
Indeed, the norm is an equivalence on the corresponding induced
spectra, so each complex-place component of \(\rho_X\) factors
through geometric restriction from finite compact supports.
Writing \(a\colon B\to\Spec(\mathbf Z)\), the degree-\(q\)
restriction factors through the Leray edge map to
\[
        H^0\bigl(\Spec(\mathbf Z);
                a_!R^qf_*\FF_2\bigr).
\]
Here \(a_!\) is exact because \(a\) is separated and quasi-finite.
These sections have proper support away from \(2\), hence vanish
at the geometric generic point. Restricting the transposed map
to the real-place summands therefore defines \(a_\infty\),
with image \(J_\infty^m(X)\).
\end{proof}

\begin{lemma}
\label{lem:archimedean-comparison-ideal}
The subgroup \(J_\infty(X)\) is a graded ideal satisfying
\[
        J_\infty^m(X)=
        \begin{cases}
        0,&0\leq m\leq2,\\
        H_{\acute et}^m(X;\FF_2),&m\geq N_X.
        \end{cases}
\]
Moreover, \(J_\infty(X)=0\) if \(K\) is totally imaginary or
\(\beta_X=0\).
\end{lemma}

\begin{proof}
The comparison maps are \(H_{\acute et}^*(X;\FF_2)\)-linear by
Proposition~\ref{prop:multiplicative-arithmetic-boundary-diagram}.
Their images form a graded submodule, whose orthogonal complement
under the trace pairing is a graded ideal.

The low-degree vanishing follows from
Proposition~\ref{prop:archimedean-comparison-image}.
For \(m>N_X\), the defining ordinary compact-support group has
negative degree and is zero. At \(m=N_X\), every connected
component of \(X\) dominates \(B\), so
\(H_{c,\mathrm{fin}}^0(X;\FF_2)=0\); the ordinary boundary sequence
then gives \(H_c^0(X;\FF_2)=0\).

For totally imaginary \(K\), the source of \(a_\infty\) is zero.
If \(\beta_X=0\), Kummer theory factors \(X\) through the finite
\'etale base \(B[i]\), whose connected components are totally
imaginary arithmetic bases. The induced morphism remains flat
and projective. Since the comparison and duality are defined
using \(X\to\Spec(\mathbf Z)\), the preceding case applies.
\end{proof}

\begin{theorem}[Generalized Hecke theorem]
\label{thm:gen_hecke}
For either choice of \(B\), with \(J_\infty(X):=0\) in the
finite-field case, one has, for every \(m\geq0\) with \(2m>N_X\),
\[
        P_{m,d+3\epsilon_B}\bigl(
                \beta_X,\overline c_1(\tau_f),\dots,
                \overline c_{\lfloor m/2\rfloor}(\tau_f)
        \bigr)
        \equiv0\pmod{J_\infty(X)}.
\]
These classes vanish in the finite-field case, and also in the
arithmetic case whenever
\(\widehat R\Gamma_\infty(X;\FF_2)\simeq0\).
\end{theorem}

\begin{proof}
By Corollary~\ref{cor:2td-universal-polynomials}, the displayed
polynomial represents \(v_{X,m}\). Put \(q:=N_X-m\).
Over a finite field, ordinary instability
(Proposition~\ref{prop:instability-relative-pro-anima-terms})
and duality give \(v_{X,m}=0\), since \(m>q\).

Assume now that \(B=\Spec(\mathcal O_{K,S})\).
For \(x\in H_c^q(X;\mu_2^{\otimes D_X})\), naturality of the
stable operations and the defining Wu identity give
\[
        \int_X\gamma_X(x)\cdot v_{X,m}
        =
        \int_X\widehat\Sq_c^m(\gamma_X(x))
        =
        \int_X\gamma_X(\Sq_c^m(x))
        =0.
\]
Indeed, \(m>q\), so
Proposition~\ref{prop:instability-relative-pro-anima-terms}
applies; if \(q<0\), the source group is zero.
Hence, \(v_{X,m}\in J_\infty^m(X)\).

If the arithmetic Tate boundary vanishes,
Theorem~\ref{thm:instability-modulo-tate-boundary} gives
\(\widehat\Sq_c^m=0\) on
\(\widehat H_c^{N_X-m}(X;\mu_2^{\otimes D_X})\);
duality again yields \(v_{X,m}=0\).
\end{proof}

\subsubsection{Real-place Wu classes}
\label{subsubsec:real-place-wu}
In the arithmetic case,
Lemma~\ref{lem:archimedean-comparison-ideal} shows that the
congruence imposes a potentially nontrivial restriction only for
\[
        d+2\leq m\leq2d+2.
\]
In higher degrees, the Wu classes are determined by the real loci.

\begin{proposition}[Real-place description of the Wu classes]
\label{prop:real-place-wu-classes}
Suppose that \(B=\Spec(\mathcal O_{K,S})\).
For each real place \(\nu\), put
\(M_\nu:=X_\nu(\mathbf R)\), and let
\[
        \operatorname{res}_\nu\colon
        H_{\acute et}^*(X;\FF_2)
        \longrightarrow H^*(M_\nu;\FF_2)[\beta_\nu]
\]
be \'etale--equivariant comparison followed by restriction to the
real locus. Here \(\beta_\nu\) is the generator of
\(H^1(G_\nu;\FF_2)\), so that
\(\operatorname{res}_\nu(\beta_X)=\beta_\nu\).

For \(m\geq N_X\), these maps induce an isomorphism
\begin{equation}
\label{eq:real-place-high-degree}
        H_{\acute et}^m(X;\FF_2)
        \xrightarrow{\ \sim\ }
        \bigoplus_{\nu\in\Omega_{\mathbf R}(K)}
        \bigoplus_{i=0}^{d}
        H^i(M_\nu;\FF_2)\,\beta_\nu^{m-i}.
\end{equation}
Write
\[
        u_\nu:=\Sq^{-1}\bigl(w(TM_\nu)\bigr)
\]
for the classical total Wu class of \(M_\nu\). Then
\begin{equation}
\label{eq:real-place-wu-class}
        \operatorname{res}_\nu(v_X)
        =
        u_\nu
        \sum_{j=0}^{d}
        T(\beta_\nu)^{d+3-j}
        \Sq^{-1}\bigl(w_j(TM_\nu)\bigr)
\end{equation}
in the forward-degree completion of
\(H^*(M_\nu;\FF_2)[\beta_\nu]\), where \(\Sq\) on \(M_\nu\)
is the classical total Steenrod square.
\end{proposition}

\begin{proof}
The generic fiber of \(f\) is smooth, so each nonempty \(M_\nu\)
is a compact smooth manifold of dimension \(d\).
By~\cite[Theorem~B, (0.4)]{GeisserSchmidt}, restriction identifies
\(H_{\acute et}^m(X;\FF_2)\), for \(m\geq N_X\), with the sum of
the real-place Tate cohomology groups. Ordinary and Tate
equivariant cohomology agree above degree \(2d\), by
Proposition~\ref{prop:ordinary-tate-boundary-bounded-range-comparison}.
Restriction to the fixed loci and the equivariant K\"unneth
decomposition therefore give~\eqref{eq:real-place-high-degree};
see~\cite[\S\S1.1.3, 1.2.1, and especially (1.26)]
{BenoistWittenbergRealHodgeI}.

Under comparison, the Thom definition identifies the \'etale
Stiefel--Whitney classes with their equivariant counterparts;
see also~\cite[Proposition~2.2]{BenoistWittenbergWuRelations}.
On the real locus, the complex tangent bundle, regarded as an
equivariant real bundle, restricts to
\[
        TM_\nu\oplus(TM_\nu\otimes\mathrm{sign}).
\]
Writing \(a_1,\dots,a_d\) for formal Stiefel--Whitney roots of
\(TM_\nu\), equations~\eqref{eq:explicit_wu}
and~\eqref{eq:base-wu-T-power} give
\[
        \operatorname{res}_\nu(v_X)
        =
        T(\beta_\nu)^3
        \prod_{i=1}^{d}T(a_i)T(a_i+\beta_\nu).
\]
Since
\[
        T(a_i+\beta_\nu)
        =
        T(\beta_\nu)+\Sq^{-1}(a_i),
        \qquad
        \prod_{i=1}^{d}T(a_i)=u_\nu,
\]
expanding the product proves~\eqref{eq:real-place-wu-class}.
\end{proof}
We work out this description for relative curves and the
projective plane in~\S\ref{sssec:real-place-defects}.

\begin{corollary}
\label{cor:real-place-wu-nonvanishing}
If \(X_\nu(\mathbf R)\neq\varnothing\) for some real place \(\nu\),
then
\[
        v_{X,2^s}\neq0
        \qquad\text{whenever }2^s\geq d+3.
\]
In particular, the absolute Wu classes are supported in infinitely many positive degrees.
\end{corollary}

\begin{proof}
At a real point, the total Wu class restricts to
\(T(\beta_\nu)^{d+3}\).
Let \(n=2^s\geq d+3\). Using
\[
        \binom{-k}{n}=(-1)^n\binom{k+n-1}{n}
        \qquad(k\geq1)
\]
with \(k=n-d-2\), Lemma~\ref{lem:T-power-coefficients} gives
\[
        [t^n]T(t)^{d+3}
        =
        \binom{d+2-n}{n}
        =
        \binom{2n-d-3}{n}
        \quad\text{in }\FF_2.
\]
Since \(n\) is a power of \(2\),
\[
        (1+t)^{2n-d-3}
        =
        (1+t^n)(1+t)^{n-d-3}
        \quad\text{in }\FF_2[t].
\]
As \(0\leq n-d-3<n\), the coefficient of \(t^n\) is \(1\).
Thus, \(v_{X,n}\) restricts to \(\beta_\nu^n\neq0\).
\end{proof}

\subsection{Specializations and examples}
\label{sec:applications_and_examples}
We apply the generalized Hecke theorem to the different and
theta characteristics, then discuss related results of Serre
and Shusterman--Sawin
(\S\S\ref{subsubsec:hecke}--\ref{subsubsec:ss}).
We next give higher-dimensional relations and examples of
real-place defects
(\S\S\ref{subsubsec:untwisted-d-folds}--\ref{sssec:real-place-defects}),
and conclude with the function-field semicharacteristic formula
(\S\ref{sssec:function-field-semicharacteristic-defect}).

\subsubsection{Hecke's theorem for the different, away from \texorpdfstring{$2$}{2}}
\label{subsubsec:hecke}
Let $B=\Spec \mathcal O_{K,S}$, and let
\[
f\colon X=\Spec \mathcal O_{L,S_L}\longrightarrow B
\]
be the finite flat morphism attached to a finite number field extension $L/K$, with $S_L$ the set of places of $L$ lying above $S$.

A well known theorem of Hecke, reads as follows.
\begin{theorem}[Hecke's theorem on the different; see
{\cite[\S~63, Satz~176]{Hecke1923}} (or, in English translation, {\cite[Theorem~176]{Hecke1981}}).]
Let \(\mathfrak D_{X/B}\) denote the relative different of the extension $f: X\to B$. Then the class of \(\mathfrak D_{X/B}\) in
\(\operatorname{Pic}(X)\) is a square.
\end{theorem}

\begin{remark}
Hecke's approach is very different from ours, and he reduces the problem to the study of certain Gauss sums with square odd
denominator, which he evaluates explicitly.
\end{remark}

We shall recover Hecke's theorem away from the prime $2$ as a corollary to the arithmetic (absolute) Wu formula/our generalized Hecke theorem.
\begin{theorem}
\label{thm:hecke-wu}
Notations as above, assume furthermore that $2$ is invertible in $B$, then the class of \(\mathfrak D_{X/B}\) in
\(\operatorname{Pic}(X)\) is a square.
\end{theorem}

\begin{proof}
The scheme \(X\) satisfies Artin--Verdier duality in total degree
\[
N_X=3.
\]

Let \(\tau_f\in K^0(X)\) be the virtual relative tangent class.
By~\eqref{eq:explicit_wu} and~\eqref{eq:base-wu-T-power},
\[
        v_X=T(\beta_X)^3\,2td(\tau_f).
\]
For a virtual bundle \(\xi\) of rank \(r\),
Lemma~\ref{lem:2td-line-bundle}, multiplicativity, and the
splitting principle give
\[
        2td(\xi)_1=r\beta_X,
        \qquad
        2td(\xi)_2
        =
        \binom{r+1}{2}\beta_X^2+\overline c_1(\xi).
\]
Since \(\rk(\tau_f)=0\) and \([t^2]T(t)^3=0\), it follows that
\[
        v_{X,2}=2td(\tau_f)_2=\overline c_1(\tau_f).
\]

Now let $\mathfrak D_{X/B}\subset \mathcal O_X$ be the different ideal, and write
\[
\mathfrak D_{X/B}=\mathcal O_X(-R)
\]
for the associated effective divisor $R$ on $X$. Thus, $R$ is the different divisor of $f$.

By~\cite[Tags 0BW2-5]{STACKS-PROJECT}, the different ideal sheaf is the inverse of the relative dualizing sheaf, i.e.
\[
\omega_{X/B}\cong \mathcal O_X(R).
\]
On the other hand,
\[
\det(\tau_f)\cong \omega_{X/B}^{-1},
\]
hence
\[
\overline c_1(\tau_f)
=
\overline c_1(\det \tau_f)
=
\overline c_1(\omega_{X/B}^{-1})
=
\overline c_1(\mathcal O_X(R)),
\]
the last equality because signs disappear modulo $2$.  Therefore,
\[
v_{X,2}=\overline c_1(\mathcal O_X(R)).
\]

Since \(2\cdot2>N_X=3\), Theorem~\ref{thm:gen_hecke} and
Lemma~\ref{lem:archimedean-comparison-ideal} give
\[
        v_{X,2}\in J_\infty^2(X)=0.
\]
Consequently,
\[
\overline c_1(\mathcal O_X(R))=0
\qquad\text{in }H^2_{\acute et}(X;\mathbb F_2).
\]
By the Kummer exact sequence, this is equivalent to the existence of a line bundle $M$ on $X$ such that
\[
\mathcal O_X(R)\cong M^{\otimes 2}.
\]
Equivalently, the class of the different divisor $R$ is even in $\Pic(X)$.
For the Dedekind scheme $X$, this is the usual statement that the ideal class of
the different is a square.
\end{proof}

\subsubsection{Atiyah's Theorem for Curves over finite fields}
\label{subsubsec:curve}
A classical theorem of Atiyah states that the canonical bundle on a Riemann surface admits a theta characteristic; see~\cite[\S~5, Rem.~2]{AtiyahSpin}. Atiyah's statement may be viewed as an absolute (as opposed to relative) version of Hecke's theorem in the category of Riemann surfaces, and has essentially been proven in the introduction. We now consider its finite field analog.

\begin{theorem}[cf.~{\cite[\S~5, Rem.~2]{AtiyahSpin}}]\label{thm:atiyah-spin}
Let $k=\FF_q$ with $q$ odd, and let $X/k$ be a smooth projective curve. Then the class of
$\omega_X$ in $\Pic(X)$ is a square. Equivalently, there exists a line bundle $\vartheta$ on $X$
such that
\[
\vartheta^{\otimes 2}\cong \omega_X.
\]
\end{theorem}

\begin{proof}
Since $X$ satisfies Artin--Verdier duality in total degree $N_X=3$,
Theorem~\ref{thm:gen_hecke} gives
\[
v_{X,2}=0.
\]

The finite-field absolute Wu formula gives \(v_X=2td(T_X)\).
Moreover, \(\beta_X^2=0\), since \(\beta_X\) is pulled back from
\(k\) and \(H_{\acute et}^2(k;\FF_2)=0\).
Lemma~\ref{lem:2td-line-bundle} therefore gives
\[
        v_{X,2}
        =
        \beta_X^2+\overline c_1(T_X)
        =
        \overline c_1(T_X).
\]
Hence,
\[
0=v_{X,2}=\overline c_1(T_X)=\overline c_1(\omega_X),
\]
because $T_X\cong \omega_X^{-1}$ and signs disappear modulo $2$. As $2\in k^\times$, the Kummer exact sequence shows that
$\omega_X\in 2\Pic(X)$, equivalently that $\omega_X\cong \vartheta^{\otimes 2}$ for some line bundle $\vartheta$ on $X$.
\end{proof}

\subsubsection{Serre's Riemann--Hurwitz theorem for spin bundles}
\label{subsubsec:hurwitz-spin}
Let
\[
f\colon X\longrightarrow Y
\]
be a generically \'etale morphism of smooth proper curves over a finite field $k$ with $2\in k^\times$. Let $R_f$ be the effective Cartier divisor
cut out by the different ideal of $f$; equivalently, $R_f$ is the vanishing divisor of
\[
df\colon f^*\Omega_{Y/k}\longrightarrow \Omega_{X/k}.
\]
Then the Riemann--Hurwitz formula gives an isomorphism
\[
\omega_X\cong f^*\omega_Y\otimes \mathcal O_X(R_f).
\]
A classical theorem of Serre, stated for Riemann surfaces (so that tameness is automatic)
in~\cite[\S~6, formulas~(11)--(12)]{SerreTheta}, reads as follows.

\begin{theorem}[Serre's Riemann--Hurwitz theorem for spin bundles]
\label{thm:serre-hurwitz-spin}
Further assume that $f$ is tamely ramified and that every ramification index $e_x$ is odd.
Then the different divisor is even:
\[
R_f=2D_f,
\qquad
D_f=\sum_{x\in X}\frac{e_x-1}{2}[x].
\]
Hence,
\[
\omega_X\cong f^*\omega_Y\otimes \mathcal O_X(2D_f).
\]
Consequently, for every theta characteristic $\vartheta_Y$ on $Y$, the line bundle
\[
\vartheta_X:=f^*\vartheta_Y\otimes \mathcal O_X(D_f)
\]
is a theta characteristic on $X$.
Equivalently, the assignment
\[
\vartheta_Y\longmapsto f^*\vartheta_Y\otimes \mathcal O_X(D_f)
\]
defines a canonical lift of theta characteristics from $Y$ to $X$.
\end{theorem}
\begin{remark}
Serre moreover proves a parity formula for the induced theta characteristic; see~\cite[\S~6, Théorème~3, Formula~14]{SerreTheta}. We shall not consider that refinement here. Also related, albeit in a somewhat different direction, Esnault--Kahn--Viehweg~\cite{EsnaultKahnViehwegOddRam} compute the Hasse--Witt classes
of the associated pushforward bundle, $E=f_*\mathcal{O}_X(D_f)$, for a general tame cover of Dedekind schemes.
\end{remark}
\begin{proof}
Tameness and the odd ramification indices give
\[
        R_f=\sum_{x\in X}(e_x-1)[x]=2D_f;
\]
see~\cite[Tag~0C1F]{STACKS-PROJECT}.
For a theta characteristic \(\vartheta_Y\), put
\(\vartheta_X:=f^*\vartheta_Y\otimes\mathcal O_X(D_f)\).
Riemann--Hurwitz then gives
\[
        \vartheta_X^{\otimes2}
        \cong
        f^*\omega_Y\otimes\mathcal O_X(2D_f)
        \cong
        \omega_X.
\]
\end{proof}

\subsubsection{The Shusterman--Sawin topological Hecke theorem}
\label{subsubsec:ss}
Sawin and the second author prove a topological analog of Theorem~\ref{thm:hecke-wu} for closed topological $3$-folds.
To state their analog, we require the notion of smooth branched covers, in the sense of Viro.

\begin{definition}[Smooth branched covering, cf.~{\cite[\S~1.1]{ViroBranchedCovering}}]
\label{def:viro-smooth-branched-cover}
Let \(X\) and \(Y\) be smooth \(n\)-folds, and let
\[
P\colon Y\to X
\]
be a surjective smooth map. We say that \(P\) is a \emph{smooth branched covering} if every point \(x\in X\) admits an open neighborhood \(U\subseteq X\) such that
\[
P^{-1}(U)=\coprod_{\alpha} V_\alpha
\]
is a disjoint union of open subsets \(V_\alpha\subseteq Y\), and for each \(\alpha\) the restricted map
\[
P|_{V_\alpha}\colon V_\alpha\to U
\]
is either a diffeomorphism, or there exist local coordinates identifying \(U\) and \(V_\alpha\) with open subsets of \(\mathbb R^{n-2}\times \mathbb C\) in which
\(P|_{V_\alpha}\) is given by
\[
(u,z)\longmapsto (u,z^m)
\]
for some integer \(m\ge 2\).
\end{definition}

\begin{theorem}[Generalized Hecke for smooth branched covers of $3$-folds~{\cite[Thm.~1.1]{SawinShusterman}}]
\label{thm:sawin-shusterman}
Let
\[
f\colon M\longrightarrow N
\]
be a smooth branched cover of closed smooth \(3\)-folds (in the sense of Viro), branched over a link \(L\subset N\), then the associated branch divisor is trivial in
\(H_1(M;\mathbb F_2)\).
\end{theorem}

More precisely, let \(\widetilde L:=f^{-1}(L)\), and for each connected
component \(\widetilde K\subset \widetilde L\) let \(e_{\widetilde K}\geq 1\)
denote the ramification index of \(f\) along \(\widetilde K\).  Since we work
modulo \(2\), no choice of orientations is needed, and one may define the
branch divisor class
\[
D_f:=\sum_{\widetilde K\subset \widetilde L}
(e_{\widetilde K}-1)[\widetilde K]
\in H_1(M;\mathbb F_2).
\]
Then~\cite[Thm.~1.1]{SawinShusterman} asserts that
\[
D_f=0\qquad\text{in }H_1(M;\mathbb F_2).
\]

\begin{remark}
The proof of Theorem~\ref{thm:sawin-shusterman} is completely different from the one below. 
It pairs \(D_f\) with arbitrary classes in \(H^1(M;\mathbb F_2)\), 
passes to the corresponding double covers, and then analyzes the resulting monodromy by means 
of a central extension of the hyperoctahedral group; see~\cite[\S\S 2--3]{SawinShusterman}.
\end{remark}

We now demonstrate how Theorem~\ref{thm:sawin-shusterman} follows from the topological Wu formula.
\begin{proof}
Write \(\widetilde L:=f^{-1}(L)\) and \(U:=M\setminus \widetilde L\). For each connected
component \(\widetilde K\subset \widetilde L\), let \(e_{\widetilde K}\geq 1\) denote the
ramification index of \(f\) along \(\widetilde K\). Since \(f\) is a smooth branched cover
in the sense of Definition~\ref{def:viro-smooth-branched-cover}, the ramification index is
locally constant on the branch locus, hence constant on each connected component
\(\widetilde K\).

All cohomology groups below are taken with coefficients in \(\mathbb F_2\), unless
explicitly indicated otherwise.

Let
\[
\tau_f:=[T_M]-f^*[T_N]\in KO^0(M)
\]
be the virtual relative tangent bundle. By definition,
\[
w(\tau_f)=w(T_M)\,w(f^*T_N)^{-1}.
\]

We first prove that \(w_2(\tau_f)=0\). For a closed smooth
\(3\)-manifold \(X\), instability gives \(v_{X,2}=0\).
The Wu formula
\[
w(T_X)=\Sq(v_X)
\]
therefore yields
\[
w_1(T_X)=v_{X,1},
\qquad
w_2(T_X)=v_{X,2}+\Sq^1(v_{X,1})=\Sq^1(v_{X,1})=v_{X,1}^2=w_1(T_X)^2.
\]
Thus, for every closed smooth \(3\)-manifold \(X\),
\begin{equation}\label{eq:w2-w1-square}
w_2(T_X)=w_1(T_X)^2.
\end{equation}

Next we claim that
\[
w_1(T_M)=f^*w_1(T_N)\qquad\text{in }H^1(M).
\]
Indeed, on \(U=M\setminus \widetilde L\) the differential is an isomorphism
\[
df\colon T_M|_U \xrightarrow{\sim} f^*T_N|_U,
\]
so the two classes agree after restriction to \(U\).

Let \(\nu(\widetilde L)\) be a tubular neighborhood of \(\widetilde L\) in \(M\). By
excision and the Thom isomorphism for the normal \(2\)-plane bundle of
\(\widetilde L\subset M\),
\[
H^1(M;U)
\cong
H^1\!\bigl(\nu(\widetilde L);\nu(\widetilde L)\setminus \widetilde L\bigr)
\cong
H^{-1}(\widetilde L)
=0.
\]
Hence, the restriction map
\[
H^1(M)\longrightarrow H^1(U)
\]
is injective, and the claim follows. In particular,
\[
w_1(\tau_f)=w_1(T_M)+f^*w_1(T_N)=0.
\]

Now set
\[
x:=f^*w_1(T_N),\qquad y:=f^*w_2(T_N).
\]
Up to degree \(2\), one has
\[
w(f^*T_N)=1+x+y,
\qquad
w(f^*T_N)^{-1}=1+x+x^2+y,
\]
since
\[
(1+x+y)(1+x+x^2+y)=1
\]
modulo terms of degree \(\ge 3\). Therefore,
\[
w_2(\tau_f)=w_2(T_M)+w_1(T_M)x+x^2+y.
\]
Using \(w_1(T_M)=x\) and~\eqref{eq:w2-w1-square} for \(X=M\) and \(X=N\), we obtain
\[
w_2(\tau_f)=x^2+x^2+x^2+x^2=0.
\]

We now compute the same class \(w_2(\tau_f)\) by localizing it along the ramification
link \(\widetilde L\).

Since \(df\) is an isomorphism over \(U\), it gives a stable trivialization of
\(\tau_f|_U\). Let
\[
w_2^{\mathrm{rel}}(\tau_f,df)\in H^2(M;U)
\]
denote the corresponding relative second Stiefel--Whitney class. By construction, its
image under the natural map
\[
H^2(M;U)\longrightarrow H^2(M)
\]
is the absolute class \(w_2(\tau_f)\).

Choose pairwise disjoint tubular neighborhoods \(T_{\widetilde K}\) of the connected
components \(\widetilde K\subset \widetilde L\). Excision gives
\[
H^2(M;U)
\cong
\bigoplus_{\widetilde K\subset \widetilde L}
H^2(T_{\widetilde K};T_{\widetilde K}\setminus \widetilde K).
\]
For each \(\widetilde K\), let
\[
u_{\widetilde K}\in H^2(T_{\widetilde K};T_{\widetilde K}\setminus \widetilde K)
\]
be the Thom class of the normal bundle of \(\widetilde K\subset M\). Since
\(\widetilde K\) is connected, the Thom isomorphism identifies
\[
H^2(T_{\widetilde K};T_{\widetilde K}\setminus \widetilde K)
\cong
H^0(\widetilde K)\cong \mathbb F_2,
\]
with generator \(u_{\widetilde K}\). Hence, there are unique coefficients
\(a_{\widetilde K}\in \mathbb F_2\) such that
\[
w_2^{\mathrm{rel}}(\tau_f,df)
=
\sum_{\widetilde K\subset \widetilde L}
a_{\widetilde K}\,u_{\widetilde K}.
\]
Under the map \(H^2(M;U)\to H^2(M)\), the class \(u_{\widetilde K}\) maps to
\(\mathrm{PD}([\widetilde K])\). Consequently,
\begin{equation}\label{eq:w2-sum-ak}
w_2(\tau_f)
=
\sum_{\widetilde K\subset \widetilde L}
a_{\widetilde K}\,\mathrm{PD}([\widetilde K]).
\end{equation}

It remains to determine \(a_{\widetilde K}\) for a fixed component
\(\widetilde K\subset \widetilde L\).

Fix a point \(p\in \widetilde K\), and let
\[
\iota_p\colon (\Delta,\partial\Delta)\hookrightarrow
(T_{\widetilde K},T_{\widetilde K}\setminus \widetilde K)
\]
be the inclusion of a meridional disk, i.e.\ of a fiber disk of the normal bundle over
\(p\). By the defining property of the Thom class,
\[
\iota_p^*(u_{\widetilde K})
\]
is the generator of
\[
H^2(\Delta;\partial\Delta)\cong \mathbb F_2.
\]
Since both source and target are \(1\)-dimensional over \(\mathbb F_2\), the map
\[
\iota_p^*\colon
H^2(T_{\widetilde K};T_{\widetilde K}\setminus \widetilde K)
\longrightarrow
H^2(\Delta;\partial\Delta)
\]
is an isomorphism. Therefore, \(a_{\widetilde K}\) is exactly the image of
\(w_2^{\mathrm{rel}}(\tau_f,df)\) in \(H^2(\Delta;\partial\Delta)\).

By Definition~\ref{def:viro-smooth-branched-cover}, after shrinking neighborhoods of
\(p\) and \(f(p)\), there exist local coordinates
\[
(-\varepsilon,\varepsilon)\times D \subset \mathbf R\times \mathbf C
\]
centered at \(p\) and \(f(p)\) in which \(\widetilde K\) and \(L\) are given by
\((-\varepsilon,\varepsilon)\times\{0\}\), and \(f\) is given by
\[
f(u,z)=(u,z^{e_{\widetilde K}}).
\]
Take \(\Delta=\{0\}\times D\). Over \(\Delta\), the tangent bundles split as
\[
T_M|_\Delta \cong \varepsilon^1\oplus T_\Delta,
\qquad
f^*T_N|_\Delta \cong \varepsilon^1\oplus (f|_\Delta)^*T_\Delta,
\]
where the first summand is the tangent line in the \(u\)-direction. On
\(\Delta^\times:=\Delta\setminus\{0\}\), the isomorphism \(df\) is the identity on the
first summand and the differential of \(z\mapsto z^{e_{\widetilde K}}\) on the second.
By stability of relative Stiefel--Whitney classes, adjoining
the common trivial line \(\varepsilon^1\) does not affect the
degree-\(2\) class. Thus, the restriction of
\(w_2^{\mathrm{rel}}(\tau_f,df)\) to \(H^2(\Delta;\partial\Delta)\) is the relative
second Stiefel--Whitney class of the rank-\(2\) datum
\[
\bigl(T_\Delta,(f|_\Delta)^*T_\Delta,df|_{\partial\Delta}\bigr).
\]

Collapsing \(\partial\Delta\) to a point identifies \(H^2(\Delta;\partial\Delta)\) with
\(H^2(S^2)\). Under this identification, the preceding relative class becomes the ordinary
\(w_2\) of the oriented real \(2\)-plane bundle \(\xi_{\widetilde K}\to S^2\) obtained
by clutching the two trivial bundles \(T_\Delta\) and \((f|_\Delta)^*T_\Delta\) along
\(\partial\Delta=S^1\) via the boundary isomorphism \(df\).

Using the complex coordinate \(z\) to orient and trivialize both bundles, the clutching
map is
\[
S^1\longrightarrow \mathbf C^\times,
\qquad
\zeta\longmapsto e_{\widetilde K}\zeta^{\,e_{\widetilde K}-1},
\]
because
\[
d(z^{e_{\widetilde K}})=e_{\widetilde K}z^{e_{\widetilde K}-1}dz.
\]
Since multiplication by the positive real number \(e_{\widetilde K}\) is homotopic to the
identity in \(\mathbf C^\times\), this loop is homotopic to
\[
\zeta\longmapsto \zeta^{\,e_{\widetilde K}-1}.
\]
Hence, \(\xi_{\widetilde K}\) is the underlying real bundle of the complex line bundle on
\(S^2\cong \mathbf{CP}^1\) with clutching function
\(\zeta\mapsto \zeta^{\,e_{\widetilde K}-1}\). Its first Chern class is therefore
\((e_{\widetilde K}-1)\) times the positive generator of \(H^2(S^2;\mathbf Z)\), and so
\[
w_2(\xi_{\widetilde K})
\equiv e_{\widetilde K}-1 \pmod 2
\]
in \(H^2(S^2)\cong \mathbb F_2\). It follows that
\[
a_{\widetilde K}=e_{\widetilde K}-1\in \mathbb F_2.
\]

Substituting this into~\eqref{eq:w2-sum-ak}, we obtain
\[
w_2(\tau_f)
=
\sum_{\widetilde K\subset \widetilde L}
(e_{\widetilde K}-1)\,\mathrm{PD}([\widetilde K])
=
\mathrm{PD}\!\left(
\sum_{\widetilde K\subset \widetilde L}
(e_{\widetilde K}-1)[\widetilde K]
\right)
=
\mathrm{PD}(D_f).
\]
Since we already proved that \(w_2(\tau_f)=0\), we conclude that
\[
\mathrm{PD}(D_f)=0.
\]
By Poincar\'e duality over \(\mathbb F_2\), this implies
\[
D_f=0\qquad\text{in }H_1(M;\mathbb F_2).
\]
\end{proof}

Thus, the theorem of Sawin--Shusterman is the smooth-\(3\)-manifold analog of Hecke's theorem.

\subsubsection{Untwisted \texorpdfstring{$d$}{d}-folds}
\label{subsubsec:untwisted-d-folds}
Let $f\colon X\to B$ be a smooth projective morphism pure of relative dimension \(d\), where \(B\) is either the spectrum of a finite field or the spectrum of a ring of \(S\)-integers \(\mathcal O_{K,S}\). Assume that $2\in \Gamma(B,\mathcal O_B^\times)$, and that $\beta_X = 0$, which holds, e.g.~when $\mu_4 \subset \Gamma(B,\mathcal O_B^\times)$. Our aim is to make explicit the resulting congruences between the mod-\(2\) Chern classes of \(T_{X/B}\), generalizing the low dimensional examples explored in the previous subsections to higher dimensions.

By Remark~\ref{rem:untwisted-case} and Theorem~\ref{thm:etale-2td-splitting}, if
\[
x_1,\dots,x_d
\]
denote the Chern roots of the class \([T_{X/B}]\), then
\[
2td([T_{X/B}])=\prod_{i=1}^d T(x_i),
\qquad
T(z)=1+\sum_{k\ge 0} z^{2^k}.
\]
Write
\[
\overline c_i:=c_i(T_{X/B}) \pmod 2,
\qquad
2td_{B,m}(X):=2td([T_{X/B}])_m.
\]
Expanding \(\prod_{i=1}^d T(x_i)\) and rewriting the result in terms of the mod-\(2\) Chern classes, one finds, up to cohomological degree \(16\),
\begin{align*}
2td_{B,6}(X)
&=
\overline c_1\,\overline c_2,\\
2td_{B,8}(X)
&=
\overline c_1^4+\overline c_1\,\overline c_3
+\overline c_2^2+\overline c_4,\\
2td_{B,10}(X)
&=
\overline c_1^3\,\overline c_2+\overline c_1^2\,\overline c_3
+\overline c_1\,\overline c_2^2+\overline c_1\,\overline c_4,\\
2td_{B,12}(X)
&=
\overline c_1^2\,\overline c_2^2+\overline c_1^3\,\overline c_3
+\overline c_1\,\overline c_2\,\overline c_3+\overline c_3^2
+\overline c_1^2\,\overline c_4+\overline c_2\,\overline c_4,\\
2td_{B,14}(X)
&=
\overline c_1^2\,\overline c_2\,\overline c_3
+\overline c_1\,\overline c_3^2
+\overline c_1\,\overline c_2\,\overline c_4,\\
2td_{B,16}(X)
&=
\overline c_1^8+\overline c_2^4+\overline c_1^2\,\overline c_3^2
+\overline c_1^2\,\overline c_2\,\overline c_4
+\overline c_1\,\overline c_3\,\overline c_4+\overline c_4^2\\
&\qquad
+\overline c_1^3\,\overline c_5+\overline c_3\,\overline c_5
+\overline c_1^2\,\overline c_6+\overline c_2\,\overline c_6
+\overline c_1\,\overline c_7+\overline c_8.
\end{align*}

Since \(\beta_X=0\),
Corollary~\ref{cor:2td-universal-polynomials} gives
\(v_X=2td(T_{X/B})\), and
Lemma~\ref{lem:archimedean-comparison-ideal} gives
\(J_\infty(X)=0\) in the arithmetic case.
Theorem~\ref{thm:gen_hecke} therefore implies vanishing from
cohomological degree \(d+1\) onward over a finite field, and
from degree \(d+2\) onward over \(\Spec(\mathcal O_{K,S})\).
Only even degrees occur, since \(\deg(x_i)=2\).
We record the resulting relations for \(4\leq d\leq8\).

\begin{example}[Untwisted fourfolds]
Assume \(d=4\). Then the vanishing range starts in cohomological degree \(5\) when \(B\) is the spectrum of a finite field, and in degree \(6\) when \(B=\Spec \mathcal O_{K,S}\). Hence, in either case
\[
\overline c_1\,\overline c_2=0,
\qquad
\overline c_1^4+\overline c_1\,\overline c_3
+\overline c_2^2+\overline c_4=0.
\]
\end{example}

\begin{example}[Untwisted fivefolds]
Assume \(d=5\). Over a finite field one obtains
\begin{align*}
\overline c_1\,\overline c_2&=0,\\
\overline c_1^4+\overline c_1\,\overline c_3
+\overline c_2^2+\overline c_4&=0,\\
\overline c_1^3\,\overline c_2+\overline c_1^2\,\overline c_3
+\overline c_1\,\overline c_2^2+\overline c_1\,\overline c_4&=0.
\end{align*}
Over \(B=\Spec \mathcal O_{K,S}\) one obtains
\begin{align*}
\overline c_1^4+\overline c_1\,\overline c_3
+\overline c_2^2+\overline c_4&=0,\\
\overline c_1^3\,\overline c_2+\overline c_1^2\,\overline c_3
+\overline c_1\,\overline c_2^2+\overline c_1\,\overline c_4&=0.
\end{align*}
\end{example}

\begin{example}[Untwisted sixfolds]
Assume \(d=6\). Then the vanishing range starts in cohomological degree \(7\) when \(B\) is the spectrum of a finite field, and in degree \(8\) when \(B=\Spec \mathcal O_{K,S}\). Since only even degrees occur, in either case one obtains
\begin{align*}
\overline c_1^4+\overline c_1\,\overline c_3
+\overline c_2^2+\overline c_4&=0,\\
\overline c_1^3\,\overline c_2+\overline c_1^2\,\overline c_3
+\overline c_1\,\overline c_2^2+\overline c_1\,\overline c_4&=0,\\
\overline c_1^2\,\overline c_2^2+\overline c_1^3\,\overline c_3
+\overline c_1\,\overline c_2\,\overline c_3+\overline c_3^2
+\overline c_1^2\,\overline c_4+\overline c_2\,\overline c_4&=0.
\end{align*}
\end{example}

\begin{example}[Untwisted sevenfolds]
Assume \(d=7\). Over a finite field one obtains
\begin{align*}
\overline c_1^4+\overline c_1\,\overline c_3
+\overline c_2^2+\overline c_4&=0,\\
\overline c_1^3\,\overline c_2+\overline c_1^2\,\overline c_3
+\overline c_1\,\overline c_2^2+\overline c_1\,\overline c_4&=0,\\
\overline c_1^2\,\overline c_2^2+\overline c_1^3\,\overline c_3
+\overline c_1\,\overline c_2\,\overline c_3+\overline c_3^2
+\overline c_1^2\,\overline c_4+\overline c_2\,\overline c_4&=0,\\
\overline c_1^2\,\overline c_2\,\overline c_3
+\overline c_1\,\overline c_3^2
+\overline c_1\,\overline c_2\,\overline c_4&=0.
\end{align*}
Over \(B=\Spec \mathcal O_{K,S}\) one obtains
\begin{align*}
\overline c_1^3\,\overline c_2+\overline c_1^2\,\overline c_3
+\overline c_1\,\overline c_2^2+\overline c_1\,\overline c_4&=0,\\
\overline c_1^2\,\overline c_2^2+\overline c_1^3\,\overline c_3
+\overline c_1\,\overline c_2\,\overline c_3+\overline c_3^2
+\overline c_1^2\,\overline c_4+\overline c_2\,\overline c_4&=0,\\
\overline c_1^2\,\overline c_2\,\overline c_3
+\overline c_1\,\overline c_3^2
+\overline c_1\,\overline c_2\,\overline c_4&=0.
\end{align*}
\end{example}

\begin{example}[Untwisted eightfolds]
Assume \(d=8\). Then the vanishing range starts in cohomological degree \(9\) when \(B\) is the spectrum of a finite field, and in degree \(10\) when \(B=\Spec \mathcal O_{K,S}\). Since only even degrees occur, in either case one obtains
\begin{align*}
\overline c_1^3\,\overline c_2+\overline c_1^2\,\overline c_3
+\overline c_1\,\overline c_2^2+\overline c_1\,\overline c_4&=0,\\
\overline c_1^2\,\overline c_2^2+\overline c_1^3\,\overline c_3
+\overline c_1\,\overline c_2\,\overline c_3+\overline c_3^2
+\overline c_1^2\,\overline c_4+\overline c_2\,\overline c_4&=0,\\
\overline c_1^2\,\overline c_2\,\overline c_3
+\overline c_1\,\overline c_3^2
+\overline c_1\,\overline c_2\,\overline c_4&=0,\\
\overline c_1^8+\overline c_2^4+\overline c_1^2\,\overline c_3^2
+\overline c_1^2\,\overline c_2\,\overline c_4
+\overline c_1\,\overline c_3\,\overline c_4+\overline c_4^2&\\
\qquad
+\overline c_1^3\,\overline c_5+\overline c_3\,\overline c_5
+\overline c_1^2\,\overline c_6+\overline c_2\,\overline c_6
+\overline c_1\,\overline c_7+\overline c_8&=0.
\end{align*}
\end{example}

\subsubsection{Real-place defects}
\label{sssec:real-place-defects}
We illustrate Proposition~\ref{prop:real-place-wu-classes}
over \(B=\Spec(\mathcal O_{K,S})\).

\begin{example}[Relative curves]
Suppose that \(f\colon X\to B\) has relative dimension \(1\).
Each nonempty real locus \(M_\nu\) is a disjoint union of circles,
so its tangent bundle is trivial. Hence
\[
        \operatorname{res}_\nu(v_X)
        =
        T(\beta_\nu)^4
        =
        1+\sum_{s\geq2}\beta_\nu^{2^s}.
\]
For \(m\geq5\), these restrictions determine \(v_{X,m}\)
uniquely by~\eqref{eq:real-place-high-degree}.
\end{example}

\begin{example}[The projective plane]
Let \(X=\mathbf P_B^2\), and suppose that \(K\) has a real place
\(\nu\). Then \(M_\nu=\mathbf{RP}^2\), with
\[
        H^*(M_\nu;\FF_2)=\FF_2[a]/(a^3),
        \qquad
        w(TM_\nu)=1+a+a^2,
        \qquad u_\nu=1+a,
\]
where \(\deg(a)=1\).
Since \(\Sq(a)=a+a^2\) and \(\Sq(a^2)=a^2\), one has
\[
        \Sq^{-1}(a)=a+a^2,
        \qquad
        \Sq^{-1}(a^2)=a^2.
\]
Thus,~\eqref{eq:real-place-wu-class} gives
\begin{align*}
        \operatorname{res}_\nu(v_X)
        &=
        (1+a)\bigl(
                T(\beta_\nu)^5
                +(a+a^2)T(\beta_\nu)^4
                +a^2T(\beta_\nu)^3
        \bigr)
        \\
        &=
        T(\beta_\nu)^5
        +a\bigl(T(\beta_\nu)^5+T(\beta_\nu)^4\bigr)
        +a^2T(\beta_\nu)^3
        \\
        &=
        T(\beta_\nu)^5
        +(\beta_\nu a+a^2)T(\beta_\nu)^3,
\end{align*}
using \(a^3=0\) and \(T(t)^2+T(t)=t\).
Lemma~\ref{lem:T-power-coefficients} gives, in \(\FF_2\),
\[
        [t^{11}]T(t)^5=\binom{17}{11}=0,
        \qquad
        [t^9]T(t)^3=\binom{15}{9}=1.
\]
Consequently,
\[
        \operatorname{res}_\nu(v_{X,11})
        =
        \beta_\nu^{10}a+\beta_\nu^9a^2
        \neq0.
\]
This class restricts to zero at every individual real point,
but is detected by the positive-degree cohomology of the real locus.
\end{example}

\subsubsection{The function-field semicharacteristic defect}
\label{sssec:function-field-semicharacteristic-defect}
We conclude with a function-field analog of the semicharacteristic
formula of Lusztig--Milnor--Peterson~\cite{LusztigMilnorPetersonSemiCharacteristics}. Classically, if \(M^{4d+1}\) is an
orientable closed topological manifold, Lusztig, Milnor, and Peterson
show that the defect between the mod-\(2\) and rational semicharacteristics
is computed via the Wu characteristic number
\[
   \chi_{1/2}(M;\mathbb F_2)-\chi_{1/2}(M;\mathbf Q)
   =
   \bigl\langle v_{2d}(M)\,\Sq^1v_{2d}(M),[M]\bigr\rangle.
\]

Throughout, let $k=\mathbb F_q$, $q$ odd, and let \(X/k\) be smooth, 
projective, and geometrically connected of dimension
\(2d\). All cohomology groups are ordinary \'etale cohomology groups.  We write
\[
        H^i(X;\mathbf Z_2(d))
        :=
        \varprojlim_n H^i(X;\mathbf Z/2^n(d)),
        \qquad
        H^i(X;\mathbf Q_2(d))
        :=
        H^i(X;\mathbf Z_2(d))\otimes_{\mathbf Z_2}\mathbf Q_2.
\]
We also use
\[
        H^i(X;\mathbf Q_2/\mathbf Z_2(d))
        :=
        \varinjlim_n H^i(X;\mathbf Z/2^n(d)).
\]
Let
\[
        \int_X:
        H^{4d+1}(X;\mathbf Q_2/\mathbf Z_2(2d))
        \longrightarrow
        \mathbf Q_2/\mathbf Z_2
\]
denote the trace map furnished by Poincar\'e duality.
We use the same notation for the finite-coefficient trace
\[
        \int_X\colon
        H^{4d+1}(X;\FF_2)\longrightarrow\FF_2,
\]
using the standing coefficient identifications.
These traces are compatible with the coefficient inclusion
\(\FF_2\hookrightarrow\mathbf Q_2/\mathbf Z_2\),
\(1\mapsto\tfrac12\).

\begin{definition}[The semicharacteristics]
\label{def:function-field-semicharacteristics}
For a field \(F\in \{\mathbb F_2, \mathbf Q_2\}\), define the semicharacteristic by
\[
        \chi_{1/2}(X;F)
        :=
        \sum_{i=0}^{2d}
        \dim_F H^i_{\mathrm{\acute et}}(X;F(d))
        \pmod 2.
\]
\end{definition}

\begin{definition}[The semicharacteristic defect]
\label{def:function-field-semicharacteristic-defect}
The \(2\)-adic semicharacteristic defect is
\[
        \Delta_2(X)
        :=
        \chi_{1/2}(X;\mathbb F_2)
        -
        \chi_{1/2}(X;\mathbf Q_2)
        \in \mathbb F_2.
\]
\end{definition}

We are now ready to state our function-field analog of the Lusztig--Milnor--Peterson formula.
\begin{theorem}[The function-field Lusztig--Milnor--Peterson formula]\label{thm:ff-lmp-formula}
With \(\delta_d\) the twisted Bockstein operator of
Lemma~\ref{lem:benoist-bockstein}, and \(v_{2d}(X)\) the middle
Wu class of \(X\), one has
\[
        \Delta_2(X)
        =
        \int_X v_{2d}(X)\,\delta_d(v_{2d}(X)).
\]
\end{theorem}
The rest of this subsubsection is dedicated to the proof of this result (see Theorem~\ref{thm:function-field-semicharacteristic-defect-formula}).

\begin{definition}
\label{def:middle-linking-parity}
Put
\[
   \mathcal T_X
   :=
   H^{2d+1}(X;\mathbb Z_2(d))_{\mathrm{tors}},
\]
The \(2\)-primary middle linking parity of \(X\) is
\[
\epsilon_2(X):=
\dim_{\mathbb{F}_2}\bigl(\mathcal T_X/2\mathcal T_X\bigr)\pmod 2.
\]
\end{definition}

\begin{proposition}[Universal-coefficient reduction to middle torsion]
\label{prop:function-field-defect-middle-torsion}
With notation as above,
\[
        \Delta_2(X)
        =
        \epsilon_2(X).
\]
\end{proposition}

\begin{proof}
For each \(i\), set
\[
        A^i:=H^i(X;\mathbf Z_2(d)).
\]
By the finiteness theorem for smooth proper \'etale cohomology over finite
fields, each \(A^i\) is a finitely generated \(\mathbf Z_2\)-module.  Write
\[
        A^i\simeq \mathbf Z_2^{b_i}\oplus T^i,
\]
where \(T^i\) is finite \(2\)-primary, and put
\[
        \tau_i:=\dim_{\mathbb F_2}(T^i/2T^i).
\]
The long exact sequence in cohomology associated to the coefficient short exact sequence
\[
0\to \mathbf Z_2(d)\xrightarrow{\cdot 2} \mathbf Z_2(d)\longrightarrow \mathbb F_2\to 0,
\]
gives a natural short exact sequence
\[
        0
        \longrightarrow
        A^i\otimes_{\mathbf Z_2}\mathbb F_2
        \longrightarrow
        H^i(X;\mathbb F_2)
        \longrightarrow
        A^{i+1}[2]
        \longrightarrow
        0.
\]
Hence,
\[
        \dim_{\mathbb F_2}H^i(X;\mathbb F_2)
        =
        b_i+\tau_i+\tau_{i+1}.
\]
On the other hand,
\[
        \dim_{\mathbf Q_2}H^i(X;\mathbf Q_2(d))=b_i.
\]
Therefore,
\[
\begin{aligned}
        \Delta_2(X)
        &=
        \sum_{i=0}^{2d}(\tau_i+\tau_{i+1})
        \pmod 2                                               \\
        &=
        \tau_0+\tau_{2d+1}
        \pmod 2.
\end{aligned}
\]
The group
\[
        A^0=H^0(X;\mathbf Z_2(d))
\]
is torsion-free, possibly zero, so \(\tau_0=0\). Thus,
\[
        \Delta_2(X)
        =
        \tau_{2d+1}
        =
        \dim_{\mathbb F_2}
        \bigl(H^{2d+1}(X;\mathbf Z_2(d))_{\mathrm{tors}}/2\bigr)
        =
        \epsilon_2(X).
\]
\end{proof}

\begin{definition}[The \(2\)-primary Artin--Tate linking form, cf.~{\cite{TateBourbaki1995},~\cite{JahnHigherBrauer2015},~\cite{FengEtaleSteenrod}}]
\label{def:function-field-artin-tate-linking}
Let
\[
        \widetilde\delta:
        H^{2d}(X;\mathbf Q_2/\mathbf Z_2(d))_{\mathrm{nd}}
        \longrightarrow
        H^{2d+1}(X;\mathbf Z_2(d))_{\mathrm{tors}}
        =
        \mathcal T_X
\]
be the boundary map associated with
\[
        0
        \longrightarrow
        \mathbf Z_2(d)
        \longrightarrow
        \mathbf Q_2(d)
        \longrightarrow
        \mathbf Q_2/\mathbf Z_2(d)
        \longrightarrow
        0.
\]
It is an isomorphism. Following the Artin--Tate--Jahn--Feng
construction, define first
\[
        \langle x,y\rangle_{\mathrm{AT}}
        :=
        \int_X x\cup \widetilde\delta(y),
        \qquad
        x,y\in
        H^{2d}(X;\mathbf Q_2/\mathbf Z_2(d))_{\mathrm{nd}},
\]
where the subscript \(\mathrm{nd}\) denotes quotient by the maximal divisible subgroup.
Transporting this form along \(\widetilde\delta\), we obtain a pairing
\[
        \lambda_X:\mathcal T_X\times \mathcal T_X\longrightarrow \mathbf Q_2/\mathbf Z_2.
\]
Explicitly, for \(a,b\in \mathcal T_X\), choose
\[
        \widetilde a
        \in
        H^{2d}(X;\mathbf Q_2/\mathbf Z_2(d))_{\mathrm{nd}}
\]
with \(\widetilde\delta(\widetilde a)=a\), and set
\[
        \lambda_X(a,b)
        :=
        \int_X \widetilde a\cup b.
\]
This is the \(2\)-primary Artin--Tate linking form on the middle torsion group
\(\mathcal T_X\).
\end{definition}

\begin{remark}
The form \(\lambda_X\) is independent of the choice of \(\widetilde a\) and is
nonsingular by Poincar\'e duality. Since the total duality degree is
\(4d+1\), it is skew-symmetric in the linking-form sense:
\[
        \lambda_X(a,b)=-\lambda_X(b,a).
\]
Feng's alternation theorem strengthens this in the finite-field setting: if
\(X/k\) is smooth, projective, geometrically connected, and of even dimension
over a finite field of odd characteristic, then the higher Artin--Tate pairing
is alternating; see~\cite[Theorem~1.3]{FengEtaleSteenrod}.
\end{remark}

\begin{lemma}[Elementary divisors of \(2\)-primary linking forms]
\label{lem:linking-elementary-divisors}
Let \(T\) be a finite \(2\)-primary abelian group equipped with a nonsingular
skew-symmetric form
\[
        \lambda:T\times T\longrightarrow \mathbf Q_2/\mathbf Z_2.
\]
Write
\[
        T\simeq
        \bigoplus_{e\ge 1}
        (\mathbf Z/2^e\mathbf Z)^{m_e}.
\]
Then \(m_e\) is even for every \(e\ge 2\). In particular,
\[
        \dim_{\mathbb F_2}(T/2T)\equiv m_1\pmod 2.
\]
If, moreover, \(\lambda\) is alternating, then \(m_e\) is even for every \(e\ge 1\), and
\[
        \dim_{\mathbb F_2}(T/2T)\equiv 0\pmod 2.
\]
\end{lemma}

\begin{proof}
For \(e\ge 1\), set
\[
        V_e
        :=
        T[2^e]\Big/\bigl(T[2^{e-1}]+2T[2^{e+1}]\bigr).
\]
A direct check from the elementary divisor decomposition shows that
\[
        \dim_{\mathbb F_2}V_e=m_e.
\]
The form induces a bilinear pairing
\[
        \lambda_e:V_e\times V_e
        \longrightarrow
        \tfrac12\mathbf Z_2/\mathbf Z_2\simeq \mathbb F_2
\]
by the formula
\[
        \lambda_e(\overline x,\overline y)
        :=
        2^{e-1}\lambda(x,y).
\]
This is well-defined.  If \(x\) is changed by an element of \(T[2^{e-1}]\), then
\(2^{e-1}\lambda(x,y)=0\).  If \(x\) is changed by an element \(2z\), with
\(z\in T[2^{e+1}]\), then
\[
        2^{e-1}\lambda(2z,y)
        =
        2^e\lambda(z,y)
        =
        \lambda(z,2^ey)
        =
        0,
\]
because \(y\in T[2^e]\).  The same argument applies in the second variable.

The form \(\lambda_e\) is nonsingular.  Suppose that
\(\overline x\in V_e\) pairs trivially with every element of \(V_e\).  Then
\[
        2^{e-1}\lambda(x,y)=0
        \qquad
        \text{for all } y\in T[2^e],
\]
or equivalently
\[
        \lambda(2^{e-1}x,y)=0
        \qquad
        \text{for all } y\in T[2^e].
\]
For a nonsingular form, the annihilator of \(T[2^e]\) is \(2^eT\).
Thus, \(2^{e-1}x\in 2^eT\).  Choose \(z\in T\) such that
\[
        2^{e-1}x=2^ez.
\]
Since \(x\in T[2^e]\), multiplying by \(2\) gives \(z\in T[2^{e+1}]\).  Hence,
\[
        x-2z\in T[2^{e-1}],
\]
and therefore
\[
        x\in T[2^{e-1}]+2T[2^{e+1}].
\]
Thus, \(\overline x=0\), proving nonsingularity.

Assume first only that \(\lambda\) is skew-symmetric.  Then
\[
        2\lambda(x,x)=0
        \qquad
        \text{for all }x\in T.
\]
If \(e\ge 2\), it follows that
\[
        \lambda_e(\overline x,\overline x)
        =
        2^{e-1}\lambda(x,x)
        =
        2^{e-2}\bigl(2\lambda(x,x)\bigr)
        =
        0.
\]
Thus, \(\lambda_e\) is nonsingular and alternating over \(\mathbb F_2\), so
\(V_e\) is even-dimensional.  Hence, \(m_e\) is even for every \(e\ge 2\).

If \(\lambda\) is alternating, then \(\lambda(x,x)=0\) for all \(x\), and the
same argument applies also for \(e=1\). Since
\[
        \dim_{\mathbb F_2}(T/2T)=\sum_{e\ge 1}m_e,
\]
the assertion follows.
\end{proof}

For \(r\in\mathbf Z\), let
\[
        \delta_r:
        H^m(X;\mathbb F_2)
        \longrightarrow
        H^{m+1}(X;\mathbb F_2)
\]
denote the twisted Bockstein $\delta_r$ of Lemma~\ref{lem:benoist-bockstein}. 
This is the connecting morphism 
associated with the short exact sequence
\[
        0
        \longrightarrow
        \mathbf Z/2(r)
        \longrightarrow
        \mathbf Z/4(r)
        \longrightarrow
        \mathbf Z/2(r)
        \longrightarrow
        0.
\]
By Lemma~\ref{lem:benoist-bockstein}, for every \(x\in H^m(X;\mathbb F_2)\), one has
\begin{equation}\label{eq:twisted_bockstein}
\delta_r(x)
=
\Sq^1(x)+r\,\beta_Xx.
\end{equation}

\begin{definition}[The twisted middle Bockstein form]
\label{def:function-field-middle-bockstein-form}
Define
\[
        B_X^{(d)}:
        H^{2d}(X;\mathbb F_2)\times H^{2d}(X;\mathbb F_2)
        \longrightarrow
        \mathbb F_2
\]
by
\[
        B_X^{(d)}(x,y)
        :=
        \int_X x\,\delta_d(y).
\]
We write
\[
        \operatorname{rank}_{\mathbb F_2}B_X^{(d)}
\]
for the rank of the adjoint linear map
\[
        H^{2d}(X;\mathbb F_2)
        \longrightarrow
        H^{2d}(X;\mathbb F_2)^\vee.
\]
\end{definition}

\begin{proposition}[Twisted de Rham--Milnor parity formula]
\label{prop:function-field-derham-milnor-parity}
One has
\[
        \operatorname{rank}_{\mathbb F_2}B_X^{(d)}
        \equiv
        \epsilon_2(X)
        \pmod 2.
\]
\end{proposition}

\begin{proof}
By Poincar\'e duality, the adjoint of \(B_X^{(d)}\) is precisely
\[
        \delta_d:
        H^{2d}(X;\mathbb F_2)
        \longrightarrow
        H^{2d+1}(X;\mathbb F_2).
\]
Therefore,
\[
        \operatorname{rank}B_X^{(d)}
        =
        \operatorname{rank}\delta_d.
\]

Write
\[
        H^{2d+1}(X;\mathbf Z_2(d))_{\mathrm{tors}}
        \simeq
        \bigoplus_{e\ge 1}
        (\mathbf Z/2^e\mathbf Z)^{m_e}.
\]
We first compute the rank of \(\delta_d\).  Let
\[
        A^i:=H^i(X;\mathbf Z_2(d)).
\]
For \(n=2,4\), the universal coefficient sequence gives
\[
        0
        \longrightarrow
        A^{2d}/n
        \longrightarrow
        H^{2d}(X;\mathbf Z/n(d))
        \longrightarrow
        A^{2d+1}[n]
        \longrightarrow
        0.
\]
The reduction map from \(n=4\) to \(n=2\) gives a morphism between these short
exact sequences, and the left vertical map
\[
        A^{2d}/4\longrightarrow A^{2d}/2
\]
is surjective. Hence, the cokernel of
\[
        H^{2d}(X;\mathbf Z/4(d))
        \longrightarrow
        H^{2d}(X;\mathbf Z/2(d))
\]
is naturally identified with the cokernel of
\[
        A^{2d+1}[4]\xrightarrow{\ \cdot2\ }A^{2d+1}[2].
\]
This cokernel is naturally isomorphic to the image of \(\delta_d\).
A free \(\mathbf Z_2\)-summand contributes nothing.
On a summand \(\mathbf Z/2^e\mathbf Z\), the displayed map is
zero for \(e=1\) and surjective for \(e\geq2\). Thus,
\[
        \operatorname{rank}_{\mathbb F_2}\delta_d=m_1.
\]

On the other hand, 
\[
        \epsilon_2(X)
        =
        \dim_{\mathbb F_2}(\mathcal T_X/2\mathcal T_X)
        \equiv
        \sum_{e\ge 1}m_e
        \pmod 2.
\]
The Artin--Tate linking form \(\lambda_X\) is nonsingular and skew-symmetric.
By Lemma~\ref{lem:linking-elementary-divisors}, \(m_e\) is even for every
\(e\ge 2\).  Hence,
\[
        \epsilon_2(X)
        \equiv
        m_1
        =
        \operatorname{rank}_{\mathbb F_2}B_X^{(d)}
        \pmod 2.
\]
\end{proof}

We now introduce the Wu-theoretic expression for the same parity.  Let
\[
        v_X=1+v_1(X)+v_2(X)+\cdots
        \in H^\ast(X;\mathbb F_2)
\]
be the absolute \'etale Wu class of \(X\), characterized by
\[
        \int_X \Sq(z)
        =
        \int_X z\,v_X
\]
for all \(z\in H^\ast(X;\mathbb F_2)\), where both sides mean the component of
total degree \(4d+1\).  Instability gives the top-half vanishing
\[
        v_i(X)=0
        \qquad
        (i>2d).
\]
Write $T_{X/k}$ for the tangent bundle of \(X\) relative to \(\Spec\,k\).
By the finite-field absolute Wu formula,
\[
        v_X
        =
        \Sq^{-1}\bigl(w^{\mathrm{\acute et}}(T_{X/k})\bigr).
\]
In particular, because \(T_{X/k}\) has even rank \(2d\), one has (see Theorem~\ref{thm:etale-sw-chern}):
\[
        w^{\mathrm{\acute et}}_1(T_{X/k})=0,
        \qquad
        v_1(X)=0.
\]
We shall call
\[
        v_{2d}(X)
        =
        \Bigl(\Sq^{-1}w^{\mathrm{\acute et}}(T_{X/k})\Bigr)_{2d}
        \in H^{2d}(X;\mathbb F_2)
\]
the middle Wu class of \(X\).

\begin{lemma}[Rank parity and characteristic vectors]
\label{lem:function-field-rank-characteristic-vector}
Let \(V\) be a finite-dimensional \(\mathbb F_2\)-vector space equipped with a
symmetric bilinear form \(B\).  Suppose that \(c\in V\) is characteristic, i.e.
\[
        B(x,x)=B(c,x)
        \qquad
        \text{for every }x\in V.
\]
Then
\[
        \operatorname{rank}(B)\equiv B(c,c)\pmod 2.
\]
\end{lemma}

\begin{proof}
Quotienting by the radical does not change the rank.  The image of \(c\) is
again characteristic for the induced nondegenerate symmetric form.  Thus, we may
assume that \(B\) is nondegenerate.

If \(B(c,c)=1\), then the line \(\mathbb F_2c\) is nondegenerate, and its
orthogonal complement has characteristic vector \(0\).  Hence, the orthogonal
complement is alternating and even-dimensional.  Therefore, \(\operatorname{rank}(B)\)
is odd.

If \(B(c,c)=0\), then either \(c=0\), in which case \(B\) is alternating and
therefore has even rank, or \(c\neq 0\).  In the latter case choose \(x\) with
\(B(c,x)=1\). The span of \(c\) and \(x\) is a nondegenerate plane, and its
orthogonal complement has characteristic vector \(0\).  Again the total rank is
even.  Thus, in all cases
\[
        \operatorname{rank}(B)\equiv B(c,c)\pmod 2.
\]
\end{proof}

\begin{proposition}[The middle Wu class is characteristic, cf.~{\cite[Theorem~4.4, Lemma~3.12]{FengEtaleSteenrod}}]
\label{prop:function-field-middle-wu-characteristic}
The class
\[
        v_{2d}(X)\in H^{2d}(X;\mathbb F_2)
\]
is characteristic for the twisted Bockstein form \(B_X^{(d)}\). Consequently,
\[
        \operatorname{rank}_{\mathbb F_2}B_X^{(d)}
        \equiv
        B_X^{(d)}(v_{2d}(X),v_{2d}(X))
        \pmod 2.
\]
\end{proposition}

\begin{proof}
For \(x,y\in H^{2d}(X;\FF_2)\),
\eqref{eq:twisted_bockstein}, the Cartan formula, and \(v_1(X)=0\)
give
\[
        B_X^{(d)}(x,y)+B_X^{(d)}(y,x)
        =
        \int_X\Sq^1(xy)
        =
        \int_Xxy\,v_1(X)
        =0.
\]
Thus \(B_X^{(d)}\) is symmetric.
By~\cite[Theorem~4.4 and Lemma~3.12]{FengEtaleSteenrod},
with coefficients \(\mathbf Z/2\), and the defining Wu identity,
\[
        B_X^{(d)}(x,x)
        =
        \int_X\Sq^{2d}\!\bigl(\delta_d(x)\bigr)
        =
        \int_Xv_{2d}(X)\,\delta_d(x)
        =
        B_X^{(d)}(v_{2d}(X),x).
\]
The rank congruence follows from
Lemma~\ref{lem:function-field-rank-characteristic-vector}.
\end{proof}

\begin{remark}
This may be regarded as the finite-field version of Browder's characteristic-vector identity, 
which states that the middle Wu class of a closed smooth manifold of dimension $4d+1$ is characteristic 
for the middle Steenrod square; see~\cite[Lemma~5]{BrowderPoincareDuality}.
\end{remark}

\begin{theorem}[Function-field semicharacteristic-defect formula]
\label{thm:function-field-semicharacteristic-defect-formula}
Let \(X/\mathbb F_q\) be smooth, projective, and geometrically connected of
dimension \(2d\), with \(q\) odd.  Then
\[
        \Delta_2(X)
        =
        \epsilon_2(X)
        =
        \bigl(\operatorname{rank}_{\FF_2}B_X^{(d)}\bmod2\bigr)
        =
        \int_X v_{2d}(X)\,\delta_d(v_{2d}(X)).
\]
\end{theorem}

\begin{proof}
By Proposition~\ref{prop:function-field-defect-middle-torsion},
\[
        \Delta_2(X)=\epsilon_2(X).
\]
By Proposition~\ref{prop:function-field-derham-milnor-parity},
\[
        \epsilon_2(X)
        \equiv
        \operatorname{rank}_{\mathbb F_2}B_X^{(d)}
        \pmod 2.
\]
By Proposition~\ref{prop:function-field-middle-wu-characteristic},
\[
        \operatorname{rank}_{\mathbb F_2}B_X^{(d)}
        \equiv
        B_X^{(d)}(v_{2d}(X),v_{2d}(X))
        \pmod 2.
\]
By definition of \(B_X^{(d)}\), this last expression is
\[
        \int_X v_{2d}(X)\,\delta_d(v_{2d}(X)).
\]
This completes the proof.
\end{proof}

\begin{remark}
In particular, Theorem~\ref{thm:function-field-semicharacteristic-defect-formula},
together with Feng's alternation theorem for the Artin--Tate linking
form~\cite[Theorem~1.3]{FengEtaleSteenrod}, shows that
\[
        \Delta_2(X)=0.
\]
\end{remark}

\appendix
\section{There is no formal lift of relative Wu formulas}
\label{par:feng-not-formal}
We now address Remark~\ref{rem:feng_vs_us}, and discuss what formal implications the finite field relative Wu formula has on the arithmetic one. Let
\[
B=\Spec \mathcal O_{K,S},
\qquad
f\colon X\to B
\]
be a smooth projective morphism with geometrically connected fibers; as usual, assume that $2\in \Gamma(B,\mathcal O_B)^\times$.

Absolute Wu classes need not commute with restriction.
For a prime \(p\equiv3\pmod4\), consider
\[
        i_p\colon\Spec\FF_p\hookrightarrow\Spec\mathbf Z[1/2].
\]
Theorem~\ref{thm:absolute-arithmetic-wu-class}
and~\eqref{eq:base-wu-T-power} give
\[
        v_{\mathbf Z[1/2]}
        =
        T(\beta_{\mathbf Z[1/2]})^3.
\]
Since \(\beta_{\FF_p}^2=0\) and
\(\beta_{\FF_p}\neq0\), one obtains
\[
        i_p^*v_{\mathbf Z[1/2]}
        =
        1+\beta_{\FF_p}
        \neq
        1=v_{\FF_p}.
\]
Here \(v_{\FF_p}=1\) follows from duality in degree \(1\) and
instability. A topological example is the inclusion
\(\mathbf{R}\mathbb{P}^2\hookrightarrow S^4\).

Write \(\Lambda\) for the constant \(\FF_2\)-sheaf.
Define the relative Wu class by
\[
        v_{X/B}:=v_X\cdot(f^*v_B)^{-1}
        \in
        \left(H_{\acute et}^*(X;\Lambda)^{\wedge,+}\right)^\times.
\]
The arithmetic Wu formula, Theorem~\ref{thm:absolute-wu-formula},
identifies this class with
\(\Sq^{-1}\bigl(w^{\acute et}(T_{X/B})\bigr)\).
We temporarily set aside this identification and consider its obstruction.
\begin{definition}
We define the obstruction class:
\[
\Delta:=w^{\text{\'{e}t}}(T_{X/B})-\Sq\bigl(v_{X/B}\bigr)
\in H_{\acute et}^*(X;\Lambda)^{\wedge,+},
\qquad
\Delta=\sum_{i\geq 0}\Delta_i,
\qquad
\Delta_i\in H^i(X;\Lambda).
\]
\end{definition}

For a closed point \(s\in B\), write
\(X_s:=X\times_B\Spec k(s)\) and let \(i_s\colon X_s\hookrightarrow X\).
Granting the specialization compatibility
\[
        i_s^*v_{X/B}=v_{X_s},
\]
Feng's formula~\cite[Theorem~6.5]{FengEtaleSteenrod}, together
with naturality of Steenrod operations and Stiefel--Whitney
classes, gives
\begin{equation}\label{eq:fibrewise-delta-vanishes}
        i_s^*\Delta=0
        \qquad
        \text{for every closed point }s\in B.
\end{equation}
We examine what these fiberwise identities alone imply, without
using the global arithmetic Wu formula.
In degree \(1\), they suffice:
\begin{claim}\label{claim:delta1-vanishes}
Assume~\eqref{eq:fibrewise-delta-vanishes}. Then the degree-one piece $\Delta_1\in H^1(X;\Lambda)$ vanishes.
\end{claim}

\begin{proof}
The Leray spectral sequence
\[
E_2^{a,b}=H^a\bigl(B;R^bf_*\Lambda\bigr)
\Longrightarrow
H^{a+b}(X;\Lambda)
\]
yields the low-degree exact sequence
\[
0\to H^1(B;\Lambda)
\xrightarrow{f^*}
H^1(X;\Lambda)
\xrightarrow{\epsilon}
H^0\bigl(B;R^1f_*\Lambda\bigr)
\xrightarrow{d_2^{0,1}}
H^2(B;\Lambda),
\]
because $f_*\Lambda=\Lambda$. By proper and smooth base change, $R^1f_*\Lambda$ is a finite locally constant sheaf on $B$, and for every geometric point $\bar s\to B$ there is a canonical identification
\[
\bigl(R^1f_*\Lambda\bigr)_{\bar s}\simeq H^1(X_{\bar s};\Lambda).
\]
Under this identification, the stalk of $\epsilon(\Delta_1)$ at $\bar s$ is the image of $i_s^*(\Delta_1)$ in $H^1(X_{\bar s};\Lambda)$.
By~\eqref{eq:fibrewise-delta-vanishes}, this stalk is zero for every closed point $s$, hence $\epsilon(\Delta_1)=0$.
Therefore, there exists a class $\chi\in H^1(B;\Lambda)$ such that
\[
\Delta_1=f^*\chi.
\]
Suppose that $\chi\neq 0$.
Since
\[
H^1(B;\Lambda)\cong \operatorname{Hom}_{\mathrm{cont}}\bigl(\pi_1^{\text{\'{e}t}}(B);\Lambda\bigr),
\]
Chebotarev's density theorem yields a closed point $s\in B$ such that $\chi(\Frob_s)\neq 0$.
On the other hand, for the geometrically connected fiber $X_s/k(s)$, the exact sequence of étale fundamental groups
\[
1\to \pi_1^{\text{\'{e}t}}(X_{\bar s})\to \pi_1^{\text{\'{e}t}}(X_s)\to G_{k(s)}\to 1
\]
implies, by the associated low-degree sequence in cohomology, that the pullback map
\[
H^1\bigl(k(s);\Lambda\bigr)\longrightarrow H^1(X_s;\Lambda)
\]
is injective.
Hence,
\[
i_s^*(\Delta_1)=f_s^*\bigl(\chi|_{k(s)}\bigr)\neq 0,
\]
contrary to~\eqref{eq:fibrewise-delta-vanishes}.
Thus, $\chi=0$, and therefore $\Delta_1=0$.
\end{proof}

We now describe the possible degree-two ambiguity.
Let \(B_{\mathrm{cl}}\) denote the set of closed points of \(B\),
and put
\[
        \Obs_2(X/B)
        :=
        \ker\Bigl(
                H^2(X;\Lambda)
                \longrightarrow
                \prod_{s\in B_{\mathrm{cl}}}H^2(X_s;\Lambda)
        \Bigr).
\]
Under~\eqref{eq:fibrewise-delta-vanishes}, one has
\(\Delta_2\in\Obs_2(X/B)\).

For the Leray spectral sequence of $f$, let
\[
0\subseteq F^2H^2(X;\Lambda)\subseteq F^1H^2(X;\Lambda)\subseteq H^2(X;\Lambda)
\]
denote the induced filtration on $H^2(X;\Lambda)$. Then
\[
F^2H^2(X;\Lambda)=\operatorname{im}\bigl(H^2(B;\Lambda)\to H^2(X;\Lambda)\bigr)
\cong \operatorname{coker}\!\bigl(d_2^{0,1}\colon H^0(B;R^1f_*\Lambda)\to H^2(B;\Lambda)\bigr).
\]
While
\[
F^1H^2(X;\Lambda)=\operatorname{ker}\bigl(H^2(X;\Lambda)\to H^0(B;R^2f_*\Lambda)\bigr),
\]
and there is a short exact sequence
\begin{equation}\label{eq:leray-degree-two}
0\to F^2H^2(X;\Lambda)
\to F^1H^2(X;\Lambda)
\xrightarrow{\rho}
\ker\!\bigl(d_2^{1,1}\colon H^1(B;R^1f_*\Lambda)\to H^3(B;\Lambda)\bigr)
\to 0.
\end{equation}
For each closed point $s\in B$, we write
\[
\operatorname{loc}_s\colon H^1(B;R^1f_*\Lambda)\longrightarrow H^1\bigl(k(s);(R^1f_*\Lambda)_{\bar s}\bigr)
\]
for the restriction map.
We also define the fiberwise Shafarevich--Tate group
\[
\Sha^1_{\mathrm{fib}}(B;R^1f_*\Lambda)
:=
\ker\Bigl(
H^1(B;R^1f_*\Lambda)
\xrightarrow{(\operatorname{loc}_s)_s}
\prod_{s\in B_{\mathrm{cl}}} H^1\bigl(k(s);(R^1f_*\Lambda)_{\bar s}\bigr)
\Bigr).
\]

\begin{proposition}\label{prop:obs2-exact-sequence}
There is a natural short exact sequence
\[
\begin{aligned}
0 &\to \operatorname{coker}\!\bigl(d_2^{0,1}\colon H^0(B;R^1f_*\Lambda)\to H^2(B;\Lambda)\bigr)
\longrightarrow \Obs_2(X/B) \\
&\xrightarrow{\rho}
\ker\!\bigl(d_2^{1,1}\bigr)\cap \Sha^1_{\mathrm{fib}}(B;R^1f_*\Lambda)
\to 0.
\end{aligned}
\]
Thus, \(\Obs_2(X/B)\) is an extension of the indicated subgroup of
\(\Sha^1_{\mathrm{fib}}(B;R^1f_*\Lambda)\) by the image of
\(H^2(B;\Lambda)\) in \(H^2(X;\Lambda)\).
\end{proposition}

\begin{proof}
Fix $\alpha\in \Obs_2(X/B)$.
By proper and smooth base change, $R^2f_*\Lambda$ is a finite locally constant sheaf, and the edge map
\[
H^2(X;\Lambda)\longrightarrow H^0\bigl(B;R^2f_*\Lambda\bigr)
\]
has the property that its stalk at $\bar s$ is the image of $i_s^*(\alpha)$ in
$H^2(X_{\bar s};\Lambda)$.
Since $i_s^*(\alpha)=0$ for every closed $s$, all these stalks vanish, so the image of $\alpha$ in
$H^0(B;R^2f_*\Lambda)$ is zero.
Thus, $\alpha\in F^1H^2(X;\Lambda)$.
By~\eqref{eq:leray-degree-two}, its image $\rho(\alpha)$ lies in $\ker(d_2^{1,1})$.

Now fix a closed point $s\in B$.
Because $k(s)$ is a finite field, it has cohomological dimension $1$ for torsion coefficients prime to $\operatorname{char} k(s)$.
Hence, the Leray spectral sequence for $X_s\to \Spec k(s)$ yields a short exact sequence
\begin{equation}\label{eq:HS-fibre-degree-two}
0\to H^1\bigl(k(s);H^1(X_{\bar s};\Lambda)\bigr)
\to H^2(X_s;\Lambda)
\to H^0\bigl(k(s);H^2(X_{\bar s};\Lambda)\bigr)
\to 0.
\end{equation}
Using proper base change again, we identify
\[
H^1(X_{\bar s};\Lambda)\simeq (R^1f_*\Lambda)_{\bar s},
\qquad
H^2(X_{\bar s};\Lambda)\simeq \bigl(R^2f_*\Lambda\bigr)_{\bar s}.
\]
By functoriality of the Leray spectral sequence, the image of $\operatorname{loc}_s\bigl(\rho(\alpha)\bigr)$ in $H^2(X_s; \Lambda)$ under the inclusion of~\eqref{eq:HS-fibre-degree-two} is precisely $i_s^*(\alpha)$, hence vanishes.
As this holds for every closed $s$, we obtain
\[
\rho(\alpha)\in \ker\!\bigl(d_2^{1,1}\bigr)\cap \Sha^1_{\mathrm{fib}}(B;R^1f_*\Lambda).
\]
This defines a homomorphism
\[
\Obs_2(X/B)
\longrightarrow
\ker\!\bigl(d_2^{1,1}\bigr)\cap \Sha^1_{\mathrm{fib}}(B;R^1f_*\Lambda).
\]
Its kernel consists exactly of those classes in $\Obs_2(X/B)$ lying in $F^2H^2(X;\Lambda)$, which is all of $F^2H^2(X;\Lambda)$. Indeed, due to the functoriality of Leray with respect to the filtration, we have $i_s^*: F^2H^2(X;\Lambda)\to F^2H^2(X_s;\Lambda) = 0$. Therefore, the kernel is isomorphic to $F^2H^2(X;\Lambda) \cong \operatorname{coker}(d_2^{0,1})$. This implies the exactness of the sequence:
\[
0\to \operatorname{coker}\!\bigl(d_2^{0,1}\colon H^0(B;R^1f_*\Lambda)\to H^2(B;\Lambda)\bigr)
\longrightarrow
\Obs_2(X/B)
\xrightarrow{\rho}
\ker\!\bigl(d_2^{1,1}\bigr)\cap \Sha^1_{\mathrm{fib}}(B;R^1f_*\Lambda).
\]

It remains to show that the map $\Obs_2(X/B) \xrightarrow{\rho} \ker\!\bigl(d_2^{1,1}\bigr)\cap \Sha^1_{\mathrm{fib}}(B;R^1f_*\Lambda)$ is surjective. Let
\[
\beta\in \ker\!\bigl(d_2^{1,1}\bigr)\cap \Sha^1_{\mathrm{fib}}(B;R^1f_*\Lambda).
\]
By~\eqref{eq:leray-degree-two}, choose a lift $\alpha\in F^1H^2(X;\Lambda)$ with $\rho(\alpha)=\beta$.
For every closed point \(s\), the class \(i_s^*(\alpha)\) lies in
\(F^1H^2(X_s;\Lambda)\). Its unique preimage under the injection
in~\eqref{eq:HS-fibre-degree-two} is
\(\operatorname{loc}_s(\beta)=0\), by functoriality of Leray.
Hence, \(i_s^*(\alpha)=0\) for every \(s\).
Thus, $\alpha\in \Obs_2(X/B)$, and the displayed map is surjective.
This proves the proposition.
\end{proof}

\begin{example}\label{ex:obs2-nonzero}
To see that $\Obs_2(X/B)$ can be nonzero, consider the case $B=\Spec \mathbf Z[1/2]$ and let $f\colon \mathbb P^1_B\to B$ be the structural morphism.
Then
\[
R^1f_*\Lambda=0,
\]
so Proposition~\ref{prop:obs2-exact-sequence} gives a canonical isomorphism
\[
\Obs_2(\mathbf P^1_B/B)
\cong H^2(B;\Lambda).
\]
The non-vanishing of $H^2(B;\Lambda)$ now follows from the non-vanishing of $\beta_B^2$ (see Remark~\ref{rem:beta-square-Q}).
\end{example}

Thus, the degree-two obstruction group can be nonzero. Consequently, fiberwise vanishing on all closed fibers does \emph{not} force global vanishing already in degree $2$, at least not formally via a standard spectral sequence argument; the arithmetic Wu theorem establishes its global
vanishing.

\printbibliography{}
\end{document}